\documentclass{article}

\PassOptionsToPackage{numbers, compress}{natbib}

\usepackage[final]{neurips_2022}

\usepackage[utf8]{inputenc} 
\usepackage[T1]{fontenc}    
\usepackage{hyperref}       
\usepackage{url}            
\usepackage{booktabs}       
\usepackage{amsfonts}       
\usepackage{nicefrac}       
\usepackage{microtype}      
\usepackage{xcolor}         
\usepackage{wrapfig}
\usepackage{soul}

\usepackage{url}            
\usepackage{booktabs}       
\usepackage{amsfonts, bm, amssymb}       
\usepackage{nicefrac}       
\usepackage{microtype}      
\usepackage{tikz} 
\usepackage{float}
\usepackage{array}
\usepackage{diagbox}
\usepackage{mathtools}
\usepackage{multirow, makecell}
\usepackage{rotating}
\usepackage{graphicx}
\usepackage{subfigure}
\usepackage{mathdots}
\usepackage{float}
\usepackage{caption}

\usepackage[colorinlistoftodos,bordercolor=orange,backgroundcolor=orange!20,linecolor=orange,textsize=scriptsize]{todonotes}

\usepackage{setspace}
\usepackage[bottom]{footmisc}
\usepackage{footnotebackref}

\usepackage{threeparttable}

\usepackage{arydshln}

\usepackage{enumerate}
\usepackage{enumitem} 

\usepackage{hyperref}
\hypersetup{colorlinks,%
citecolor=blue,%
filecolor=black,%
linkcolor=blue,%
urlcolor=blue
}
\usepackage{cancel}

\usepackage{amsmath}
\usepackage{amssymb}
\usepackage{mathtools}
\usepackage{amsthm}
\usepackage[capitalize,noabbrev]{cleveref}

\newcommand\Exp{\mathbb{E}}
\newcommand\R{\mathbb{R}}

\newcommand\bound{b}

\usepackage[colorinlistoftodos,bordercolor=orange,backgroundcolor=orange!20,linecolor=orange,textsize=scriptsize]{todonotes}
\usepackage{amsmath,amsthm}
\usepackage{thmtools}
\usepackage{thm-restate}
\newtheoremstyle{theoremdd}
  {\topsep}
  {\topsep}
  {\itshape}
  {0pt}
  {\bfseries}
  {. }
  { }
  {\thmname{#1}\thmnumber{ #2}\textnormal{\thmnote{ (#3)}}}
\theoremstyle{theoremdd}

\theoremstyle{remark}
\newtheorem{remark}{Remark}

\theoremstyle{definition}
\newtheorem{definition}{Definition}
\usepackage{pifont}
\newcommand{\cmark}{\ding{51}}%
\newcommand{\xmark}{\ding{55}}%

\usepackage{tcolorbox}
\tcbset{boxsep=0mm,boxrule=0pt,colframe=white,arc=0mm,left=0.5mm,right=0.5mm}

\newcommand\SC{\mathcal{S}}

\title{Dynamics of SGD with Stochastic Polyak Stepsizes: Truly Adaptive Variants and Convergence \\ to Exact Solution}

\author{%
  Antonio Orvieto\thanks{Corresponding author: \texttt{antonio.orvieto@inf.ethz.ch}. Part of this work was done while interning at Mila, Universit\'e de Montr\'eal under the supervision of Nicolas Loizou and Simon Lacoste-Julien.} \\
  Department of Computer Science,\\
  ETH Z\"urich
  \And
  Simon Lacoste-Julien\thanks{Canada CIFAR AI Chair} \\
  Mila and DIRO, \\ Universit\'e de Montr\'eal
  \And
  Nicolas Loizou\\
 AMS and MINDS,\\
  Johns Hopkins University
}

\begin{document}

\maketitle

\begin{abstract}
  Recently~\citet{loizou2021stochastic}, proposed and analyzed stochastic gradient descent (SGD) with stochastic Polyak stepsize (SPS). The proposed SPS comes with strong convergence guarantees and competitive performance; however, it has two main drawbacks when it is used in non-over-parameterized regimes: (i) It requires a priori knowledge of the optimal mini-batch losses, which are not available when the interpolation condition is not satisfied (e.g., regularized objectives), and (ii) it guarantees convergence only to a neighborhood of the solution. In this work, we study the dynamics and the convergence properties of SGD equipped with new variants of the stochastic Polyak stepsize and provide solutions to both drawbacks of the original SPS. We first show that a simple modification of the original SPS that uses lower bounds instead of the optimal function values can directly solve issue (i). On the other hand, solving issue (ii) turns out to be more challenging and leads us to valuable insights into the method's behavior. We show that if interpolation is not satisfied, the correlation between SPS and stochastic gradients introduces a bias, which effectively distorts the expectation of the gradient signal near minimizers, leading to non-convergence - even if the stepsize is scaled down during training. To fix this issue, we propose DecSPS, a novel modification of SPS, which guarantees convergence to the exact minimizer - without a priori knowledge of the problem parameters. For strongly-convex optimization problems, DecSPS is the first stochastic adaptive optimization method that converges to the exact solution without restrictive assumptions like bounded iterates/gradients.
\end{abstract}
\section{Introduction}
We consider the stochastic optimization problem:
\begin{equation}
\label{MainProb}
\min_{x\in\mathbb{R}^d} \left[ f(x) = \frac{1}{n} \sum_{i=1}^n f_i(x) \right],
\end{equation}
where each $f_i$ is convex and lower bounded. We denote by $\mathcal{X}^*$ the non-empty set of optimal points $x^*$ of equation~\eqref{MainProb}. We set $f^*:=\min_{x\in\mathbb{R}^d} f(x)$, and $f_i^* := \inf_{x\in\mathbb{R}^d} f_i(x)$. 

In this setting, the algorithm of choice is often Stochastic Gradient Descent~(SGD), i.e. $x^{k+1} =  x^k - \gamma_k \nabla f_{\SC_k}(x^k)$, where $\gamma_k>0$ is the stepsize at iteration $k$, $\SC_k\subseteq[n]$ a random subset of datapoints~(minibatch) with cardinality $B$ sampled independently at each iteration $k$, and $\nabla f_{\SC_k}(x):=\frac{1}{B}\sum_{i \in \SC_k} \nabla f_i(x)$ is the minibatch gradient. \looseness=-1

A careful choice of $\gamma_k$ is crucial for most applications~\citep{bottou2018optimization,goodfellow2016deep}. The simplest option is to pick $\gamma_k$ to be constant over training, with its value inversely proportional to the  Lipschitz constant of the gradient. While this choice yields fast convergence to the neighborhood of a minimizer, two main problems arise: (a) the optimal $\gamma$ depends on (often unknown) problem parameters --- hence often requires heavy tuning ; and (b) it cannot be guaranteed that $\mathcal{X}^*$ is reached in the limit~\citep{ghadimi2013stochastic,gower2019sgd, gower2021sgd}. A simple fix for the last problem is to allow polynomially decreasing stepsizes~(second option)~\citep{nemirovski2009robust}: this choice for $\gamma_k$ often leads to convergence to $\mathcal{X}^*$, but hurts the overall algorithm speed. The third option, which became very popular with the rise of deep learning, is to implement an \textit{adaptive} stepsize. These methods do not commit to a fixed schedule, but instead use the optimization statistics~(e.g. gradient history, cost history) to tune the value of $\gamma_k$ at each iteration. These stepsizes are known to work very well in deep learning~\citep{zhang2019adaptive}, and include Adam~\citep{kingma2014adam}, Adagrad~\citep{duchi2011adaptive}, and RMSprop~\citep{tieleman2012lecture}. 

Ideally, a theoretically grounded adaptive method should yield fast convergence to $\mathcal{X}^*$ without knowledge of problem dependent parameters, such as the gradient Lipshitz constant or the strong convexity constant. As a result, an ideal adaptive method should require very little tuning by the user, while matching the performance of a fine-tuned $\gamma_k$. However, while in practice this is the case for common adaptive methods such as Adam and AdaGrad, the associated convergence rates often rely on strong assumptions --- e.g. that the iterates live on a bounded domain, or that gradients are uniformly bounded in norm~\citep{duchi2011adaptive,ward2019adagrad, vaswani2020adaptive}. While the above assumptions are valid in the constrained setting, they are problematic for problems defined in the whole~$\R^d$. 

A promising new direction in the adaptive stepsizes literature is based on the idea of Polyak stepsizes, introduced by~\citep{polyak1987introduction} in the context of deterministic convex optimization. Recently~\citep{loizou2021stochastic} successfully adapted Polyak stepsizes to the stochastic setting, and provided convergence rates matching fine-tuned SGD --- while the algorithm does not require knowledge of the unknown quantities such as the gradient Lipschitz constant. The results especially shines in the overparameterized strongly convex setting, where linear convergence to $x^*$ is shown. This result is especially important since, under the same assumption, no such rate exists for AdaGrad~(see e.g.~\citep{vaswani2020adaptive} for the latest results) or other adaptive stepsizes. Moreover, the method was shown to work surprisingly well on deep learning problems, without requiring heavy tuning~\citep{loizou2021stochastic}.

Even if the stochastic Polyak stepsize~(SPS)~\citep{loizou2021stochastic} comes with strong convergence guarantees, it has two main drawbacks when it is used in non-over-parameterized regimes: (i)  It requires \textit{a priori} knowledge of the optimal mini-batch losses, which are not often available for big batch sizes or regularized objectives~(see discussion in~\S\ref{sec:background}) and (ii) it guarantees convergence only to a neighborhood of the solution. 
In this work, we study the dynamics and the convergence properties of SGD equipped with new variants of SPS for solving general convex optimization problems. Our new proposed variants provide solutions to both drawbacks of the original SPS. 
 
 \subsection{Background and Technical Preliminaries}
\label{sec:background}
The stepsize proposed by~\citep{loizou2021stochastic} is
\vspace{-3mm}
\begin{equation}
\label{SPSmax}
\tag{SPS$_{\max}$}
\gamma_k = \min \left\{ \frac{f_{\SC_k}(x^k)-f_{\SC_k}^*}{c\|\nabla f_{\SC_k}(x^k)\|^2}, \gamma_b \right\},
\end{equation}
where $\gamma_b, c>0$ are problem-independent constants, $f_{\SC_k} :=\frac{1}{|\SC_k|}\sum_{i\in\SC_k} f_i$, $f_{\SC_k}^* = \inf_{x\in\mathbb{R}^d} f_{\SC_k}(x)$.

\vspace{-2mm}
\paragraph{Dependency on $\boldsymbol{f_{\SC_k}^*}$.} Crucially the algorithm requires knowledge of $f_{\SC_k}^*$ for every realization of the mini-batch $\SC_k$. In the non-regularized overparametrized setting~(e.g. neural networks), $f_{\SC_k}$ is often zero for every subset $\SC$~\citep{zhang2021understanding}. However, this is not the only setting where $f_{\SC}^*$ is computable: e.g., in the regularized logistic loss with batch size 1, it is possible to recover a cheap closed form expression for each $f_{i}^*$~\citep{loizou2021stochastic}. Unfortunately, if the batch-size is bigger than $1$ or the loss becomes more demanding~(e.g. cross-entropy), then \textit{no such closed-form computation is possible}.

\vspace{-3mm}
\paragraph{Rates and comparison with AdaGrad.} In the convex overparametrized setting~(more precisely, under the interpolation condition, i.e. $\exists~x^* \in \mathcal{X}^*$ : 
$\inf_{x\in\mathbb{R}^d} f_\SC(x) = f_\SC(x^*)$ for all $\SC$, see also~\S\ref{sec:SPS}), SPS$_{\max}$ enjoys a convergence speed of $\mathcal{O}(1/k)$, without requiring knowledge of the gradient Lipschitz constant or other problem parameters. Recently,~\citep{vaswani2020adaptive} showed that the same rate can be achieved for AdaGrad in the same setting. However, there is an important difference: the rate of~\citep{vaswani2020adaptive} is technically $\mathcal{O}(dD^2/k)$, where $d$ is the problem dimension and $D^2$ is a global bound on the squared distance to the minimizer, which is assumed to be finite. Not only does SPS$_{\max}$ not have this dimension dependency, which dates back to crucial arguments in the AdaGrad literature~\citep{duchi2011adaptive,levy2018online}, but also does not require bounded iterates. While this assumption is satisfied in the constrained setting, it has no reason to hold in the unconstrained scenario. Unfortunately, this is a common problem of all AdaGrad variants: with the exception of~\citep{xie2020linear}~(which works in a slightly different scenario), no rate can be provided in the stochastic setting without the bounded iterates/gradients~\citep{ene2020adaptive} assumption --- even after assuming strong convexity. However, in the non-interpolated setting, AdaGrad enjoys a convergence guarantee of $\mathcal{O}(1/\sqrt{k})$~(with the bounded iterates assumption). A similar rate does not yet exist for SPS, and our work aims at filling this gap.

\vspace{-1mm}
\subsection{Main Contributions}

As we already mentioned, in the non-interpolated setting SPS$_{\max}$ has the following issues:
\vspace{-2mm}
\begin{enumerate}[leftmargin=44pt]
    \item[\textbf{Issue (1)}:] For $B>1$ (minibatch setting), SPS$_{\max}$ requires the exact knowledge of $f_{\SC}^*$. This is not practical. 
    \item[\textbf{Issue (2)}:] SPS$_{\max}$ guarantees convergence to a neighborhood of the solution. It is not clear how to modify it to yield convergence to the exact minimizer.
\end{enumerate}
\vspace{-2mm}
Having the above two issues in mind, the main contributions of our work (see also Table~\ref{tab:results} for a summary of the main complexity results obtained in this paper) are summarized as follows:
\vspace{-2mm}
\begin{itemize}[leftmargin=*]
    \item In \S\ref{sec:estimation}, we provide a direct solution for Issue (1). We explain how only a lower bound on $f_{\SC}^*$~(trivial if all $f_i$s are non-negative) is required for convergence to a neighborhood of the solution. While this neighborhood is bigger that the one for SPS$_{\max}$, our modified version provides a practical baseline for the solution to the second issue.
    \item We explain why Issue (2) is highly non-trivial and requires an in-depth study of the bias induced by the interaction between gradients and Polyak stepsizes. Namely, we show that simply multiplying the stepsize of SPS$_{\max}$ by $1/\sqrt{k}$ --- which would work for vanilla SGD~\citep{nemirovski2009robust} --- yields a bias in the solution found by SPS~(\S\ref{sec:bias}), regardless of the estimation of $f_{\SC}^*$.
    \item In \S\ref{sec:convergence_decr}, we provide a solution to the problem (Issue (2)) by introducing additional structure --- as well as the fix to Issue (1) ---  into the stepsize. We call the new algorithm \emph{Decreasing SPS}~(DecSPS), and provide a convergence guarantee under the bounded domain assumption --- matching the standard AdaGrad results.
    \item In \S\ref{sec:no_bound} we go one step further and show that, if strong convexity is assumed, iterates are bounded with probability 1 and hence we can remove the bounded iterates assumption. To the best of our knowledge, DecSPS, is the first stochastic adaptive optimization method that converges to the exact solution without assuming strong assumptions like bounded iterates/gradients.
    \item In \S\ref{sec:non-smooth} we provide extensions of our approach to the non-smooth setting.
    \item In \S\ref{sec:exp}, we corroborate our theoretical results with experimental testing.
\end{itemize}
\vspace{-1mm}

\begin{table}[ht]
    \centering
    \scalebox{0.71}{
    \begin{tabular}{l|c|l|c|c|l}
    \toprule
        \textbf{Stepsize} &  \textbf{Citation}& \textbf{Assumptions} & \begin{tabular}{cc} \textbf{No Knowledge} \\ \textbf{of} $\boldsymbol{f_{\SC}^*}$ \end{tabular} & \begin{tabular}{cc} \textbf{Exact} \\ \textbf{Convergence} \end{tabular} &  \textbf{Theorem}\\ 
        \midrule
        SPS$_{\max}$ &\citep{loizou2021stochastic} &convex, smooth &\xmark& \xmark&  Thm. \ref{thm:loizou}\\ 
        \midrule
        SPS$_{\max}^\ell$ & This paper & convex, smooth & \cmark&\xmark  &  Thm. \ref{thm:loizou_bs}, Cor. \ref{cor:loizou_bs_cor}\\ 
        \midrule
        DecSPS & This paper &  convex, smooth, bounded iterates&\cmark &\cmark & Cor.~\ref{cor:sqrt_smooth}, $\mathcal{O}(1/\sqrt{K})$ \\ 
         \cline{2-6}
         & This paper &  strongly-convex, smooth &\cmark & \cmark & Thm.~\ref{thm:no_bound}, $\mathcal{O}(1/\sqrt{K})$ \\ 
        \midrule
        DecSPS-NS & This paper &  convex, bounded iterates/grads&\cmark &\cmark &  Cor.~\ref{cor:sqrt_NS}, $\mathcal{O}(1/\sqrt{K})$ \\ 
        \bottomrule
    \end{tabular}}
    \vspace{4mm}
    \caption{\small Summary of the considered stepsizes and the corresponding theoretical results in the non-interpolated setting. The studied quantity in all Theorems, with respect to which all rates are expressed is $\Exp \left[f(\bar{x}^K)-f(x^*)\right]$, where $\bar{x}^K=\frac{1}{K}\sum_{k=0}^{K-1} x^k$. In addition, for all converging methods, we consider the stepsize scaling factor $c_k = \mathcal{O}(\sqrt{k})$, formally defined in the corresponding sections. For the methods without exact convergence, we show in \S\ref{sec:bias} that any different scaling factor cannot make the algorithm convergent.}
    \label{tab:results}
    \vspace{-3mm}
\end{table}

\vspace{-3.5mm}
\section{Background on Stochastic Polyak Stepsize}
\label{sec:SPS}
\vspace{-2mm}
In this section, we provide a concise overview of the results in~\citep{loizou2021stochastic}, and highlight the main assumptions and open questions.

To start, we remind the reader that problem~\eqref{MainProb} is said to be interpolated if there exists a problem solution $x^*\in\mathcal{X}^*$ such that $\inf_{x\in\mathbb{R}^d} f_i(x) = f_i(x^*)$ for all $i\in[n]$. The degree of interpolation at batch size $B$ can be quantified by the following quantity, introduced by~\citep{loizou2021stochastic} and studied also in~\citep{vaswani2020adaptive,d2021stochastic}: fix a batch size $B$, and let $\SC\subseteq [n]$ with $|\SC| = B$.
\begin{equation}
    \sigma^2_B := \Exp_{\SC}[ f_\SC(x^*)-f_\SC^*] = f(x^*)-\Exp_\SC [ f_\SC^*]
\end{equation}
It is easy to realize that as soon as problem~\eqref{MainProb} is interpolated, then $\sigma^2_B=0$ for each $B\le n$. In addition, note that $\sigma^2_B$ is non-increasing as a function of $B$.

We now comment on the main result from~\citep{loizou2021stochastic}.

\begin{restatable}[Main result of~\citep{loizou2021stochastic}]{thm}{loizou}
\label{thm:loizou}
Let each $f_i$ be $L_i$-smooth convex functions. Then SGD with SPS$_{\max}$, mini-batch size $B$, and $c=1$, converges as:  
$\Exp \left[f(\bar{x}^K)-f(x^*)\right] 
\leq \frac{\|x^0-x^*\|^2}{\alpha \, K} + \frac{2\gamma_{b}\sigma^2_B}{\alpha},$
where $\alpha=\min \left\{\frac{1}{2cL_{\max}},\gamma_{b}\right\}$ and $\bar{x}^K=\frac{1}{K}\sum_{k=0}^{K-1} x^k$. If in addition $f$ is $\mu$-strongly convex, then, for any $c\geq1/2$, SGD with SPS$_{\max}$ converges as:
$\Exp \|x^{k}-x^*\|^2 
\leq \left(1-\mu \alpha \right)^k  \|x^0-x^*\|^2 + \frac{2\gamma_{b} \sigma^2_B }{\mu \alpha},$
where again $\alpha=\min \{\frac{1}{2cL_{\max}},\gamma_{b}\}$ and $L_{\max}=\max \{L_i\}_{i=1}^n$ is the maximum smoothness constant.
\end{restatable}

In the overparametrized setting, the result guarantees convergence to the exact minimizer, without knowledge of the gradient Lipschitz constant~(as vanilla SGD would instead require) and without assuming bounded iterates~(in contrast to~\citep{vaswani2020adaptive}). 

As soon as (1) a \textit{regularizer} is applied to the loss~(e.g. $L2$ penalty), or (2) the number of datapoints gets comparable to the dimension, then the problem is not interpolated and SPS$_{\max}$ only converges to a neighborhood and it gets impractical to compute $f_\SC^*$ --- \textit{this is the setting we study in this paper.}

\begin{remark}[What if $\|\nabla f_{\SC_k}\|=0$?]
In the rare case that $\|\nabla f_{S_k}(x^k)\|^2 =0$, there is no need to evaluate the stepsize. In this scenario, the update direction $\nabla f_{S_k}(x^k)=0$ and thus the iterate is not updated irrespective of the choice of step-size. If this happens, the user should simply sample a different minibatch. We note that in our experiments~(see \S\ref{sec:exp}), such event never occurred.
\end{remark}
\vspace{-4mm}
\paragraph{Related work on Polyak stepsize:} The classical Polyak stepsize \citep{polyak1987introduction} has been successfully used in the analysis of deterministic subgradient methods in different settings \citep{boyd2003subgradient, davis2018subgradient, hazan2019revisiting}. First attempts on providing an efficient variant of the stepsize that works well in the stochastic setting were made in \citep{berrada2019training, oberman2019stochastic}. However, as explained in \citep{loizou2021stochastic}, none of these approaches provide a natural stochastic extension with strong theoretical convergence guarantees, and thus \citet{loizou2021stochastic} proposed the stochastic Polyak stepsize SPS$_{\max}$ as a better alternative.\footnote{A variant of SGD with SPS$_{\max}$ was also proposed by~\citet{asi2019stochastic} as a special case of a model-based method called the lower-truncated model. \citet{asi2019stochastic} also proposed a decreasing step-size variant of SPS$_{\max}$ which is closely related but different than the DecSPS that we propose in \S\ref{sec:convergence_decr}. Among some differences, they assume interpolation for their convergence results whereas we do not in~\S\ref{sec:convergence_decr}. We describe the differences between our work and~\citet{asi2019stochastic} in more detail in Appendix~\ref{sec:app_related_work}.} Despite its recent appearance, SPS$_{\max}$ has already been used and analyzed as a stepsize for SGD for solving structured non-convex problems \citep{gower2021sgd}, in combination with other adaptive methods \citep{vaswani2020adaptive}, with a moving target~\citep{gower2021stochastic} and in the update rule of stochastic mirror descent~\citep{d2021stochastic}. These extensions are orthogonal to our approach, and we speculate that our proposed variants can also be used in the above settings. We leave such extensions for future work. 

\vspace{-2mm}
\section[]{Removing $\boldsymbol{f_\SC^*}$ from SPS}
\label{sec:estimation}
\vspace{-3mm}

As motivated in the last sections, computing $f_\SC^*$ in the non-interpolated setting is not practical. In this section, we explore the effect of using a lower bound $\ell_\SC^*\le f_{\SC}^*$ instead in the SPS$_{\max}$ definition.
\begin{tcolorbox}
\vspace{-1mm}
\begin{equation}
\tag{SPS$_{\max}^\ell$}
\gamma_k = \min \left\{ \frac{f_{\SC_k}(x^k)-\ell_{\SC_k}^*}{c\|\nabla f_{\SC_k}(x^k)\|^2}, \gamma_b \right\},
\end{equation}
\end{tcolorbox}
Such a lower bound is easy to get for many problems of interest: indeed, for standard regularized regression and classification tasks, the loss is non-negative hence one can pick $\ell_{\SC}^* = 0$, for any $\SC\subseteq[n]$. 

The obvious question is: what is the effect of estimating $\ell_{\SC}^*$ on the convergence rates in Thm.~\ref{thm:loizou}? We found that the proof of~\citep{loizou2021stochastic} is easy to adapt to this case, by using the following fundamental bound~(see also Lemma~\ref{lemma:basic_loizou}): $\frac{1}{2c L_{\SC_k}}\le \frac{f_{\SC_k}(x^k)-f_{\SC_k}^*}{c\|\nabla f_{\SC_k}(x^k)\|^2} \le\frac{f_{\SC_k}(x^k)-\ell_{\SC_k}^*}{c\|\nabla f_{\SC_k}(x^k)\|^2}$.

The following results can be seen as an easy extension of the main result of~\citep{loizou2021stochastic}, under a newly defined suboptimality measure:
\begin{equation}
\label{sigmahat}
    \hat\sigma_B^2 := \Exp_{\SC_k}[f_{\SC_k}(x^*)-\ell_{\SC_k}^*] = f(x^*)-\Exp_{\SC_k}[\ell_{\SC_k}^*].
\end{equation}

\begin{restatable}[]{thm}{loizoubs}
\label{thm:loizou_bs}
Under SPS$_{\max}^\ell$, the same exact rates in Thm.~\ref{thm:loizou} hold~(under the corresponding assumptions), after replacing $\sigma_B^2$ with $\hat\sigma_B^2$.
\end{restatable}
\vspace{-2mm}
And we also have an easy practical corollary.
\begin{restatable}[]{cor}{loizoubscor}
\label{cor:loizou_bs_cor}
In the context of Thm.~\ref{thm:loizou_bs}, assume all $f_i$s are non-negative and estimate $\ell_{\SC}^* = 0$ for all $\SC\subseteq[n]$. Then the same exact rates in Thm.~\ref{thm:loizou} hold for SPS$_{\max}^\ell$, after replacing $\sigma^2_B$ with $f^*=f(x^*)$.
\end{restatable}
\vspace{-2mm}
A numerical illustration of this result can be found in Fig.~\ref{fig:ell_i_plot}. In essence, both theory and experiments confirm that, if interpolation is not satisfied, then we have a linear rate until a convergence ball, where the size is optimal under exact knowledge of $f^*_{\SC}$. Instead, under interpolation, if all the $f_i$s are non-negative and $f^*=0$, then SPS$_{\max}$=SPS$^{\ell}_{\max}$. Finally, in the less common case in practice where $f^*>0$ but we still have interpolation, then SPS$_{\max}$ converges to the exact solution while SPS$^{\ell}_{\max}$ does not. To conclude SPS$^{\ell}_{\max}$ does not (of course) work better than SPS$_{\max}$, but it is a practical variant which we can use as a baseline in \S\ref{sec:convergence_decr} for an adaptive stochastic Polyak stepsize with convergence to the true $x^*$ in the non-interpolated setting.
\begin{figure*}[t]
    \centering
    \textbf{\scriptsize No Interpolation and $\boldsymbol{f^*>0}$  \quad\quad\quad\quad\quad Interpolation and $\boldsymbol{f^*>0}$\quad\quad\quad\quad\quad\quad Interpolation and $\boldsymbol{f^*=0}$}
    \includegraphics[height = 0.21\textwidth]{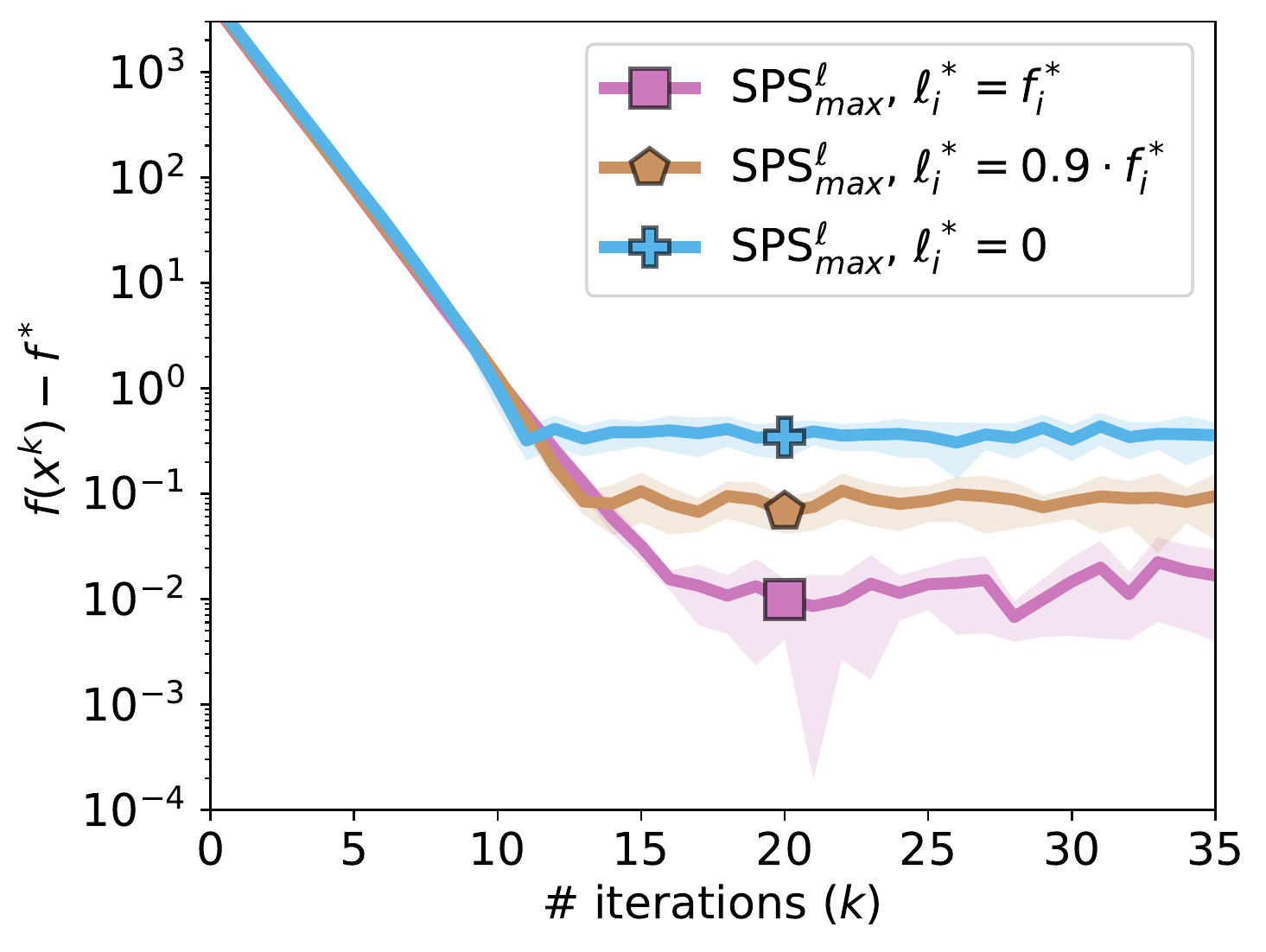}
    \includegraphics[height = 0.21\textwidth]{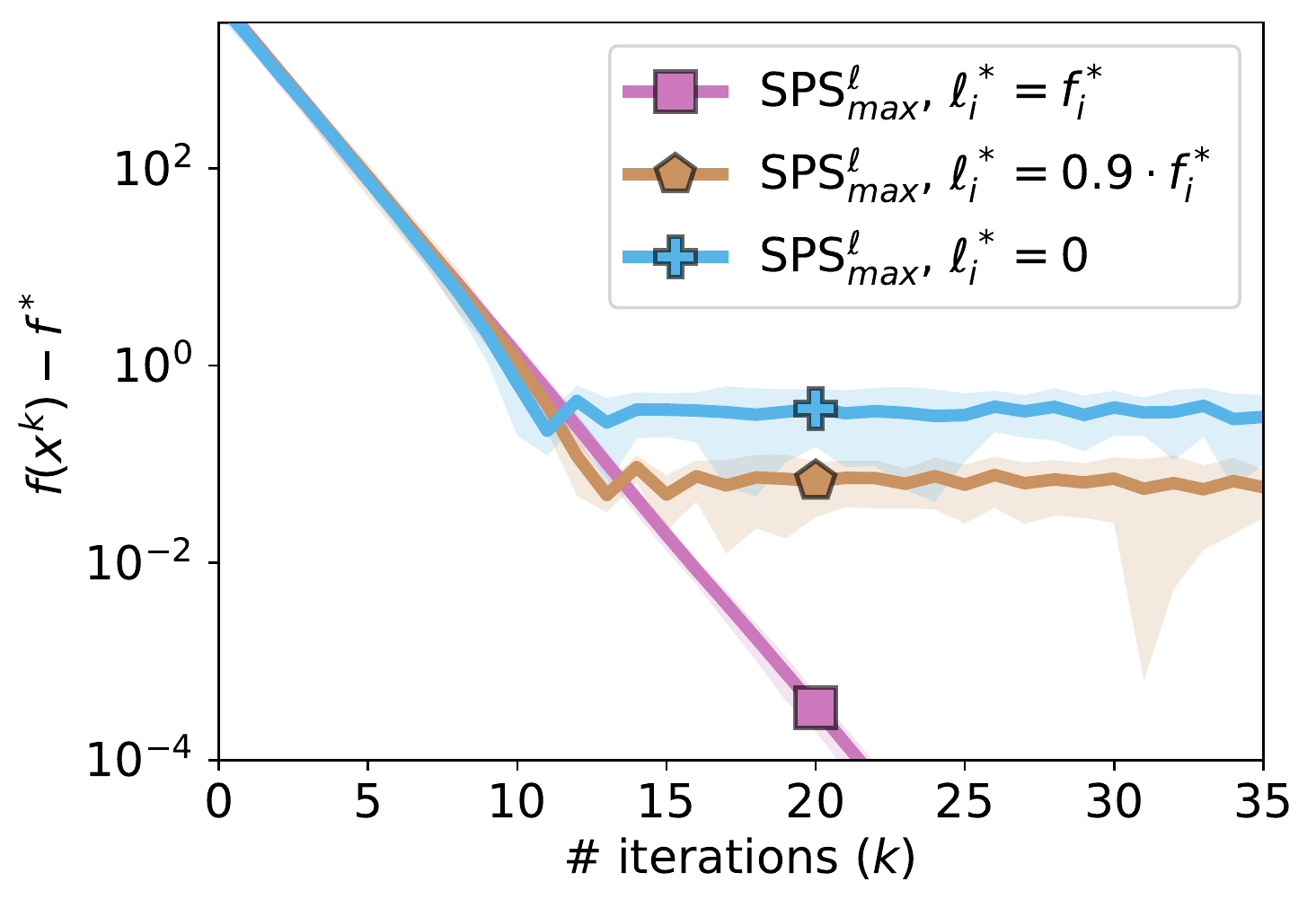}
    \includegraphics[height = 0.21\textwidth]{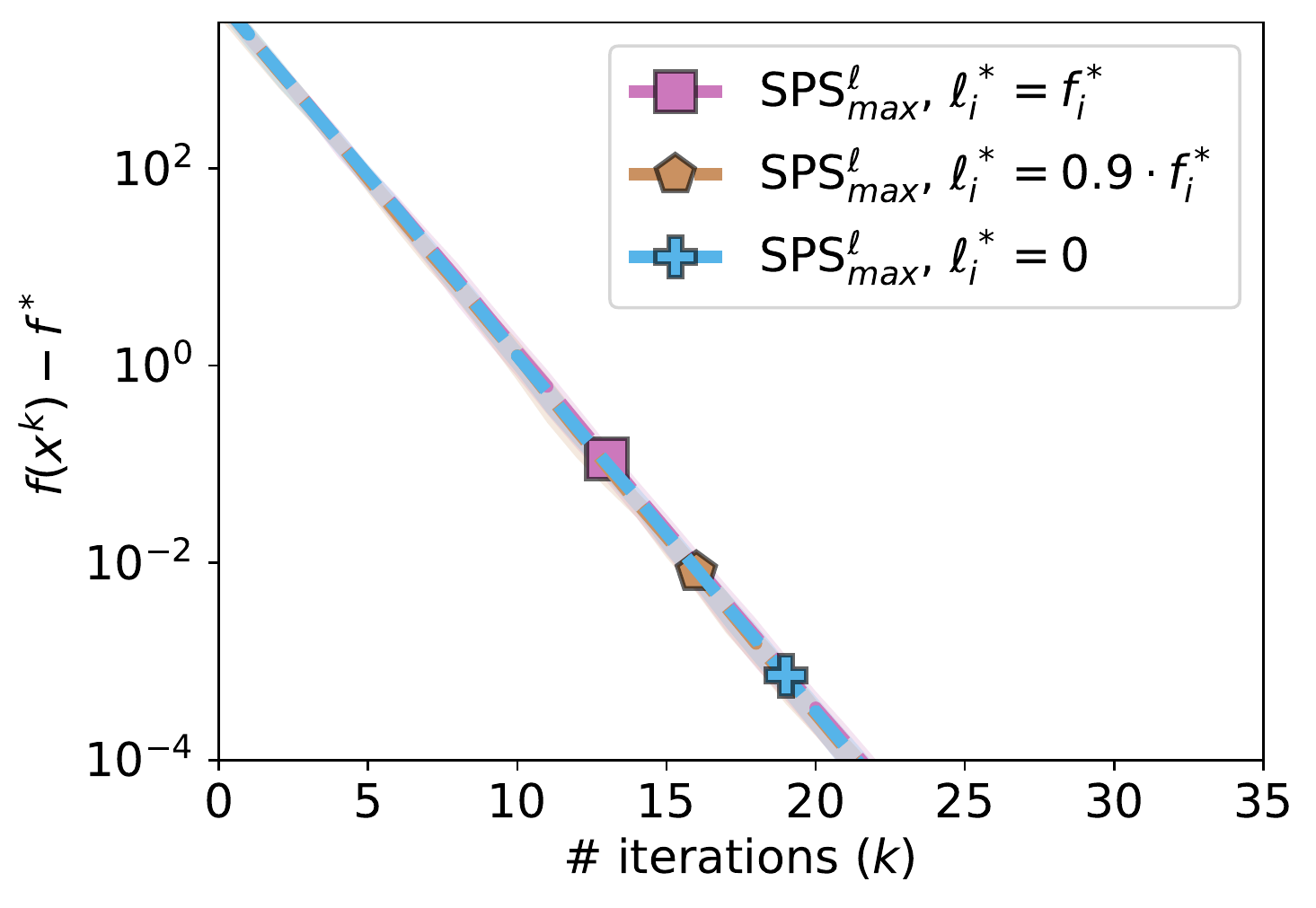}
    \vspace{-3mm}
    \caption{ \small We consider a $100$ dim problem with $n=100$ datapoints where each $f_i = \frac{1}{2}(x-x_i^*)^\top H_i(x-x_i^*) + f_i^*$, with $f_i^*=1$ for all $i\in[n]$ and $H_i$ a random SPD matrix generated using the standard Gaussian matrix $A_i\in\mathbb{R}^{d\times 3d}$ as $H_i = A_iA_i^\top/(3d)$. If $x_i^*\ne x_j^*$ for $i\ne j$, then the problem does \textbf{not satisfy interpolation (left plot)}. Instead, if all $x_i^*$s are equal, \textbf{then the problem is interpolated (central plot)}. The plot shows the behaviour of SPS$^{\ell}_{\max}$~($\gamma_b=2)$ for different choices of the approximated suboptimality $\ell_i^*$. We plot~(mean and std deviation over 10 runs) the function suboptimality level $f(x)-f(x^*)$ for different values of $\ell_i^*$. Note that, if instead $f_i^*=0$ for all $i$ then all the shown \textbf{algorithms coincide~(right plot)} and converge to the solution.}
    \label{fig:ell_i_plot}
    \vspace{-3.5mm}
\end{figure*}

\vspace{-2mm}
\section{Bias in the SPS dynamics.}
\label{sec:bias}
\vspace{-3mm}

In this section, we study convergence of the standard SPS$_{\max}$ in the non-interpolated regime, under an additional~(decreasing) multiplicative factor, in the most ideal setting: batch size 1, and we have knowledge of each $f_i^*$. That is, we consider $\gamma_k =\min\{\frac{f_{i_k}(x^k)-f_{i_k}^*}{c_k\|\nabla f_{i_k}(x^k)\|^2},\gamma_b\}$ with $c_k\to \infty$, e.g. $c_k=\mathcal{O}(\sqrt{k})$ or $c_k=\mathcal{O}(k)$. We note that, in the SGD case, simply picking e.g. $\gamma_k = \gamma_0/\sqrt{k+1}$ would guarantee convergence of $f(x^k)$ to $f(x^*)$, in expectation and with high probability~\citep{kushner2003stochastic,nemirovski2009robust}. Therefore, it is natural to expect a similar behavior for SPS, if $1/c_k$ safisfies the usual Robbins-Monro conditions~\citep{robbins1951stochastic}: $\sum_{k=0}^\infty 1/c_k =\infty, \sum_{k=0}^\infty 1/c_k^2 <\infty$.

\textit{We show that this is not the case}: quite interestingly, $f(x^k)$ converges to a biased solution due to the \textit{correlation between $\nabla f_{i_k}$ and $\gamma_k$}. we show this formally, in the case of non-interpolation~(otherwise both SGD and SPS do not require a decreasing learning rate).
\vspace{-4mm}
\paragraph{Counterexample.} 
Consider the following finite-sum setting: $f(x) = \frac{1}{2} f_1(x) + \frac{1}{2} f_2(x)$ with $f_1(x)=\frac{a_1}{2}(x-1)^2, f_2(x)=\frac{a_2}{2}(x+1)^2$. 
To make the problem interesting, we choose $a_1=2$ and $a_2=1$: this introduces asymmetry in the average landscape with respect to the origin. During optimization, we sample $f_1$ and $f_2$ independently and seek convergence to the unique minimizer $x^*=\frac{a_1-a_2}{a_1+a_2}=1/3$. The first thing we notice is that $x^*$ is not a stationary point for the dynamics under SPS. 
Indeed note that since $f_i^*=0$ for $i=1,2$ we have~(assuming $\gamma_b$ large enough): $\gamma_k \nabla f_{i_k}(x) =  \frac{x-1}{2c_k}$, if $i_k = 1$, and $\gamma_k \nabla f_{i_k}(x) = \frac{x+1}{2c_k}$ if $i_k = 2$.

Crucially, note that this update is \textit{curvature-independent}. The expected update is $\Exp_{i_k}[\gamma_k\nabla f_{i_k}(x)] =\frac{x-1}{4c_k}+\frac{x+1}{4c_k}=\frac{1}{2c_k} x$. Hence, the iterates can only converge to $x=0$ --- because this is the only fixed point for the update rule. The proof naturally extends to the multidimensional setting, an illustration can be found in Fig.~\ref{fig:convergence_wrong_SPS}. 

\begin{wrapfigure}[21]{r}{6.7 cm}
\vspace{-8mm}
\centering
    \textbf{\small\ \     No interpolation\ \ \ \ \ \ \ Interpolation\ \ \ \ \ }\\
    \includegraphics[width=0.99\linewidth]{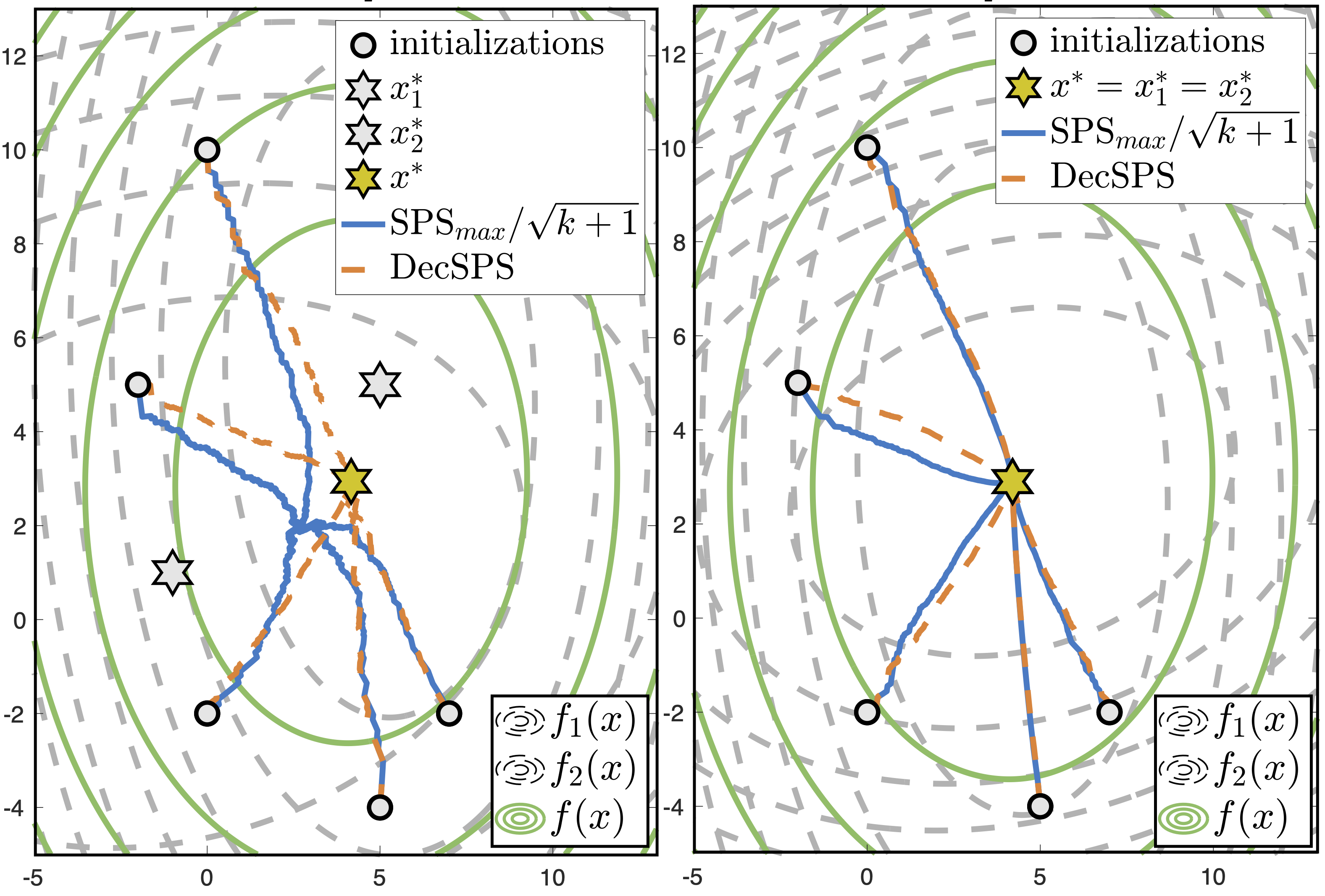}
    \caption{\small Dynamics of SPS$_{\max}$ with decreasing multiplicative constant~(SGD style) compared with DecSPS. We compared both in the \textbf{interpolated setting~(right)} and in the \textbf{non-interpolated setting~(left)}. In the non-interpolated setting, a simple multiplicative factor introduces a bias in the final solution, as discussed in this section. We consider two dimensional $f_i=\frac{1}{2}(x-x_i^*)^\top H_i (x-x_i^*)$, for $i = 1,2$ and plot the contour lines of the corresponding landscapes, as well as the average landscape $(f_1+f_2)/2$ we seek to minimize. Solution is denoted with a gold star.}
    \label{fig:convergence_wrong_SPS}
\end{wrapfigure}

In the same picture, we show how our modified variant of the vanilla stepsize --- we call this new algorithm DecSPS, see \S\ref{sec:convergence_decr} --- instead converges to the correct solution.

\begin{remark}
SGD with~(non-adaptive) stepsize $\gamma_k$ instead keeps the curvature, and therefore is able to correctly estimate the average $\Exp_{i_k}[\gamma_k \nabla f_{i_k}(x)] = \frac{\gamma_k}{2}(a_1+a_2)\left[x-\frac{a_1-a_2}{a_1+a_2}\right]$ --- precisely because $\gamma_k$ is independent from $\nabla f_{i_k}$. From this we can see that SGD can only converge to the correct stationary point $x^* = \frac{a_1-a_2}{a_1+a_2}$ --- because again this is the only fixed point for the update rule.
\end{remark}
\vspace{-2mm}
In the appendix, we go one step further and provide an analysis of the bias of SPS in the one-dimensional quadratic case~(Prop.~\ref{prop:bias}). Yet, we expect the precise characterization of the bias phenomenon in the non-quadratic setting to be particularly challenging. We provide additional insights in \S\ref{app:bias_corr}. Instead, in the next section, we show how to effectively modify $\gamma_k$ to yield convergence to $x^*$ without further assumptions.
\vspace{-1mm}
\section{DecSPS: Convergence to the exact solution}
\label{sec:convergence_decr}
\vspace{-3mm}
We propose the following modification of the vanilla SPS proposed in~\citep{loizou2021stochastic}, designed to yield convergence to the exact minimizer while keeping the main adaptiveness properties\footnote{Similar choices are possible. We found that this leads to the biggest stepsize magnitude, allowing for faster convergence in practice.}. We call it Decreasing SPS~(DecSPS), since it combines a steady stepsize decrease with the adaptiveness of SPS.
\begin{tcolorbox}
\begin{equation}
\tag{DecSPS}
{\small
 \gamma_k:=\frac{1}{c_k}\min\left\{\frac{f_{\SC_k}(x^k)-\ell^*_{\SC_k}}{\|\nabla f_{\SC_k}(x^k)\|^2}, \ c_{k-1}\gamma_{k-1}\right\},}
    \label{eq:decr_step}
\end{equation}
\end{tcolorbox}
for $k\in\mathbb{N}$, where $c_k\ne 0$ for every $k\in\mathbb{N}$. We set $c_{-1}=c_0$ and $\gamma_{-1} = \gamma_b>0$~(stepsize bound, similar to~\citep{loizou2021stochastic}), to get $\gamma_0:=\frac{1}{c_0}\cdot\min\left\{\frac{f_{\SC_0}(x^0)-\ell^*_{\SC_0}}{\|\nabla f_{\SC_0}(x^k)\|^2},\quad c_0\gamma_{b}\right\}$.

\begin{restatable}[]{lem}{sandwichbounds}
\label{lemma:easy_bounds}
Let each $f_i$ be $L_i$ smooth and let $(c_k)_{k=0}^\infty$ be any non-decreasing positive sequence of real numbers.  Under DecSPS, we have $\min\left\{\frac{1}{2c_k L_{\max}},\frac{c_0\gamma_b}{c_k}\right\}\le\gamma_k\le \frac{c_0\gamma_b}{c_k}$, and $\gamma_{k-1}\le\gamma_k$
\end{restatable}
\begin{remark}
As stated in the last lemma, under the assumption of $c_k$ non-decreasing, $\gamma_k$ is trivially \textit{non-increasing} since $\gamma_k\le c_{k-1}\gamma_{k-1}/c_k$.
\end{remark}
The proof can be found in the appendix, and is based on a simple induction argument.

\vspace{-2mm}

\subsection{Convergence under bounded iterates}
\vspace{-2mm}
The following result provides a proof of convergence of SGD for the $\gamma_k$ sequence defined above. 
\begin{restatable}[]{thm}{decSPScvx}
\label{thm:SPS_bounded_domain_cvx} Consider SGD with DecSPS and let $(c_k)_{k=0}^\infty$ be any non-decreasing sequence such that $c_k\ge1, \forall k\in\mathbb{N}$. Assume that each $f_i$ is convex and $L_i$ smooth. We have:
\begin{equation}
 \Exp[f(\bar{x}^K)-f(x^*)] \leq \frac{2c_{K-1} \tilde L D^2}{K} +  \frac{1}{K}\sum_{k=0}^{K-1} \frac{\hat\sigma^2_{B}}{c_k},
\end{equation}
\vspace{-2mm}
where $D^2 :=\max_{k \in [K-1]}\|x^k-x^*\|^2$, $\tilde L := \max\left \{ \max_i\{L_i\},\frac{1}{2c_0\gamma_b}\right\}$ and $\bar{x}^K=\frac{1}{K}\sum_{k=0}^{K-1} x^k$.
\end{restatable}
If $\hat \sigma^2_B=0$, then $c_k=1$ for all $k\in\mathbb{N}$ leads to a rate $\mathcal{O}(\frac{1}{K})$, well known from~\citep{loizou2021stochastic}. If $\hat\sigma^2_B>0$, as for the standard SGD analysis under decreasing stepsizes, the choice $c_k=\mathcal{O}(\sqrt{k})$ leads to an optimal asymptotic trade-off between the deterministic and the stochastic terms, hence to the asymptotic rate $\mathcal{O}(1/\sqrt{k})$ since $\sum_{k=0}^{K-1}\frac{1}{\sqrt{k+1}}\le 2\sqrt{K}$. Moreover, picking $c_0=1$ minimizes the speed of convergence for the deterministic factor. Under the assumption that $\hat \sigma^2_B\ll \tilde L D^2$~(e.g. reasonable distance initialization-solution and $L_{\max} > 1/\gamma_b$), this factor is dominant compared to the factor involving $\hat \sigma^2_B$. For this setting, the rate simplifies as follows.
\begin{restatable}[]{cor}{CorSqrtSmooth}
\label{cor:sqrt_smooth}
Under the setting of Thm.~\ref{thm:SPS_bounded_domain_cvx}, for $c_k=\sqrt{k+1}$~($c_{-1}=c_0$) we have 
\begin{equation}
 \Exp[f(\bar{x}^K)-f(x^*)] \leq \frac{2\tilde L D^2 + 2\hat\sigma^2_{B}}{\sqrt{K}}.
\end{equation}
\end{restatable}

\begin{remark}[Beyond bounded iterates]
The result above crucially relies on the bounded iterates assumption: $D^2<\infty$. To the best of our knowledge, if no further regularity is assumed, modern convergence results for adaptive methods~(e.g. variants of AdaGrad) in convex stochastic programming require\footnote{Perhaps the only exception is the result of~\citep{xie2020linear}, where the authors work on a different setting: i.e. they introduce the RUIG inequality.} this assumption, or else require gradients to be globally bounded. To mention a few: \citep{duchi2011adaptive, reddi2019convergence, ward2019adagrad, defossez2020simple,vaswani2020adaptive}. A simple algorithmic fix to this problem is adding a cheap projection step onto a large bounded domain~\citep{levy2018online}. We can of course include this projection step in DecSPS, and the theorem above will hold with no further modification. Yet we found this to be not necessary: the strong guarantees of SPS in the strongly convex setting~\citep{loizou2021stochastic} let us go one step beyond: in \S\ref{sec:no_bound} we show that, if each $f_i$ is strongly convex~(e.g. regularizer is added), then one can bound the iterates globally with probability one, without knowledge of the gradient Lipschitz constant. To the best of our knowledge, no such result exist for AdaGrad --- except~\citep{traore2021sequential}, for the deterministic case.
\end{remark}

\begin{remark}[Dependency on the problem dimension]
In standard results for AdaGrad, a dependency on the problem dimension often appears~(e.g. Thm.~1 in~\citep{vaswani2020adaptive}). This dependency follows from a bound on the AdaGrad preconditioner that can be found e.g. in Thm.~4 in~\citep{levy2018online}. In the SPS case no such dependency appears --- specifically because the stepsize is lower bounded by $1/(2c_kL_{\max})$.
\end{remark}

\subsection[]{Removing the bounded iterates assumption}
\label{sec:no_bound}
\vspace{-1mm}
We prove that under DecSPS the iterates live in a set of diameter $D_{\max}$ almost surely. This can be done by assuming strong convexity of each $f_i$. 

The result uses this alternative definition of \textit{neighborhood}: $\hat\sigma^2_{B,\max} := \max_{\SC\subseteq[n], |\SC| =B}[ f_{\SC}(x^*)-\ell_{\SC}^*]$.

Note that trivially $\hat\sigma^2_{B,\max}<\infty$ under the assumption that all $f_{i}$ are lower bounded and $n<\infty$.

\begin{restatable}{prop}{PropSCUnbound}
\label{prop:strong_convex_bound_2}
Let each $f_i$ be $\mu_i$-strongly convex and $L_i$-smooth. The iterates of SGD with DecSPS with $c_k =\sqrt{k+1}$~(and $c_{-1}=c_0$) are such that $\|x^k-x^*\|^2\le D^2_{\max}$ almost surely $\forall k\in\mathbb{N}$, where $D^2_{\max}:=\max\left\{\|x^{0}-x^*\|^2, \frac{2c_0\gamma_b\hat\sigma^2_{B,\max}}{\min\left\{\frac{\mu_{\min}}{2 L_{\max}},\mu_{\min}\gamma_b\right\}}\right\}$, with $\mu_{\min} = \min_{i\in[n]}\mu_i$ and $L_{\max} = \max_{i\in[n]}L_i$.
\end{restatable}
\vspace{-1mm}
The proof relies on the variations of constants formula and an induction argument --- it is provided in the appendix. We are now ready to state the main theorem for the unconstrained setting, which follows  from Prop.~\ref{prop:strong_convex_bound_2} and Thm.~\ref{thm:SPS_bounded_domain_cvx}.
\begin{restatable}{thm}{ThmNoBound}
\label{thm:no_bound}
Consider SGD with the DecSPS stepsize $\gamma_k:=\frac{1}{\sqrt{k+1}}\cdot\min\left\{\frac{f_{\SC_k}(x^k)-\ell^*_{\SC_k}}{\|\nabla f_{\SC_k}(x^k)\|^2},\gamma_{k-1}\sqrt{k}\right\}$, for $k\ge 1$ and $\gamma_0$ defined as at the beginning of this section. Let each $f_i$ be $\mu_i$-strongly convex and $L_i$-smooth:
\vspace{-2mm}
\begin{equation}
 \Exp[f(\bar{x}^K)-f(x^*)] \leq \frac{2\tilde L D^2_{\max}+2\hat\sigma^2_{B}}{\sqrt{K}}.
\end{equation}
\end{restatable}
\begin{remark}[Strong Convexity]
The careful reader might notice that, while we assumed strong convexity, our rate is slower than the optimal $\mathcal{O}(1/K)$. This is due to the adaptive nature of DecSPS. It is indeed notoriously hard to achieve a convergence rate of $\mathcal{O}(1/K)$ for adaptive methods in the strongly convex regime. While further investigations will shed light on this interesting problem, we note that \textit{the result we provide is somewhat unique in the literature}: we are not aware of any adaptive method that enjoys a similar convergence rate without either (a) assuming bounded iterates/gradients or (b) assuming knowledge of the gradient Lipschitz constant or the strong convexity constant.
\end{remark}

\begin{remark}[Comparison with Vanilla SGD]
On a convex problem, the non-asymptotic performance of SGD with a decreasing stepsize $\gamma_k=\eta/\sqrt{k}$ strongly depends on the choice of $\eta$. The optimizer might diverge if $\eta$ is too big for the problem at hand. Indeed, most bounds for SGD, under no access to the gradient Lipschitz constant, display a dependency on the size of the domain and rely on projections after each step. If one applies the method in the unconstrained setting, such convergence rates technically do not hold, and tuning is sometimes necessary to retrieve stability and good performance. Instead, for DecSPS, simply by adding a small regularizer, the method is guaranteed to converge at the non-asymptotic rate we derived even in the unconstrained setting.
\end{remark}

\vspace{-2mm}
\subsection{Extension to the non-smooth setting}
\label{sec:non-smooth}
\vspace{-2mm}

For any $\SC\subseteq[n]$, we denote in this section by $g_{\SC}(x)$ the subgradient of $f_{\SC}$ evaluated at $x$. We discuss the extension of DecSPS to the non-smooth setting.

A straightforward application of DecSPS leads to a stepsize $\gamma_k$ which is no longer lower bounded~(see Lemma~\ref{lemma:easy_bounds}) by the positive quantity $\min\left\{\frac{1}{2c_k L_{\max}},\frac{c_0\gamma_b}{c_k}\right\}$. Indeed, the gradient Lipschitz constant in the non-smooth case is formally $L_{\max} = \infty$. Hence, $\gamma_k$ prescribed by DecSPS can get arbitrarily small\footnote{Take for instance the deterministic setting one-dimensional setting $f(x) = |x|$. As $x\to 0$, the stepsize prescribed by DecSPS converges to zero. This is not the case e.g. in the quadratic setting.} for finite $k$. One easy solution to the problem is to enforce a lower bound, and adopt a new proof technique. Specifically we propose the following:
\begin{tcolorbox}
\begin{equation}
\tag{DecSPS-NS}
{\small
    \gamma_k:=\frac{1}{c_k}\cdot\min\left\{\max\left\{c_0\gamma_{\ell}, \frac{f_{\SC_k}(x^k)-\ell^*_{\SC_k}}{\|g_{\SC_k}(x^k)\|^2}\right\}, c_{k-1}\gamma_{k-1}\right\},}
\end{equation}
\end{tcolorbox}
where $c_k\ne 0$ for every $k\ge0$, $\gamma_\ell\ge\gamma_b$ is a small positive number and all the other quantities are defined as in DecSPS. In particular, as for DecSPS, we set $c_{-1}=c_0$ and $\gamma_{-1} = \gamma_b$. Intuitively, $\gamma_k$ is selected to live in the interval $[c_0\gamma_\ell/c_k, c_0\gamma_b/c_k]$~(see proof in~\S\ref{sec:app_proof4}, appendix), but has subgradient-dependent adaptive value. In addition, this stepsize is enforced to be monotonically decreasing.

\begin{restatable}[]{thm}{decSPScvxNS}
\label{thm:SPS_bounded_domain_cvx_NS} For any non-decreasing positive sequence $(c_k)_{k=0}^\infty$, consider SGD with DecSPS-NS. Assume that each $f_i$ is convex and lower bounded. We have
\vspace{-2mm}
\begin{equation}
 \Exp[f(\bar{x}^K)-f(x^*)]\leq \frac{c_{K-1} D^2}{\gamma_\ell c_0 K } + \frac{1}{K} \sum_{k=0}^{K-1}\frac{c_0 \gamma_b G^2}{c_k} ,
\end{equation}
where $D^2 :=\max_{k \in [K-1]}\|x^k-x^*\|^2$ and $G^2:=\max_{k \in [K-1]}\|g_{\SC_k}(x^k)\|^2$.
\end{restatable}
\vspace{-2mm}
One can then easily derive an $\mathcal{O}(1/\sqrt{k})$ convergence rate. This is presented in \S\ref{sec:app_proof4}~(appendix).

\vspace{-3mm}

\section{Numerical Evaluation}
\label{sec:exp}
\vspace{-2mm}

We evaluate the performance of DecSPS with $c_k = c_0\sqrt{k+1}$ on binary classification tasks, with regularized logistic loss $f(x) = \frac{1}{n}\sum_{i=1}^n \log(1+\exp(y_i\cdot a_i^\top x))+\frac{\lambda}{2} \|x\|^2$, where $a_i\in\mathbb{R}^d$ is the feature vector for the $i$-th datapoint and $y_i\in\{-1,1\}$ is the corresponding binary target.

We study performance on three datasets: (1) a Synthetic
Dataset, (2) The A1A dataset~\citep{chang2011libsvm} and (3) the
Breast Cancer dataset~\citep{Dua:2019}. We choose
different regularization levels and batch sizes bigger than 1.
Details are reported in ~\S\ref{app:exp}, and the code is available at \url{https://github.com/aorvieto/DecSPS}. \\At the batch sizes and regularizer levels we choose, the problems do not satisfy interpolation. Indeed, running full batch gradient descent yields $f^*>0$. While running SPS$_{\max}$ on these problems (1) does not guarantee convergence to $f^*$ and (2) requires full knowledge of the set of optimal function values $\{f^*_\SC\}_{|\SC|=B}$, in DecSPS we can simply pick the lower bound $0=\ell^*_\SC \le f^*_\SC$ for every $\SC$. Supported by Theorems~\ref{thm:SPS_bounded_domain_cvx}~\&~\ref{thm:no_bound}~\&~\ref{thm:SPS_bounded_domain_cvx_NS}, we expect SGD with DecSPS to converge to the minimum $f^*$.

\vspace{-3mm}
\paragraph{Stability of DecSPS.} DecSPS has two hyperparameters: the upper bound $\gamma_b$ on the first stepsize and the scaling constant $c_0$. While Thm.~\ref{thm:SPS_bounded_domain_cvx_NS} guarantees convergence for any positive value of these hyperparameters, the result of Thm.~\ref{thm:SPS_bounded_domain_cvx} suggests that using $c_0=1$ yields the best performance under the assumption that $\hat \sigma^2_B\ll \tilde L D^2$~(e.g. reasonable distance of initialization from the solution, and $L_{\max} > 1/\gamma_b$). In Fig.~\ref{fig:sensitivity}, we show on the synthetic dataset that (1) $c_0=1$ is indeed the best choice in this setting and (2) the performance of SGD with DecSPS is almost independent of $\gamma_b$. Similar findings are reported and commented in Figure~\ref{fig:A1A_Breast_tuning}~(Appendix) for the other datasets. Hence, \textit{for all further experiments}, we choose the hyperparameters $\gamma_b=10, c_0=1$.

\vspace{-3mm}

\begin{wrapfigure}[12]{r}{7 cm}
\vspace{-0.5cm}
\centering
    \centering
    \includegraphics[width = 0.23\textwidth]{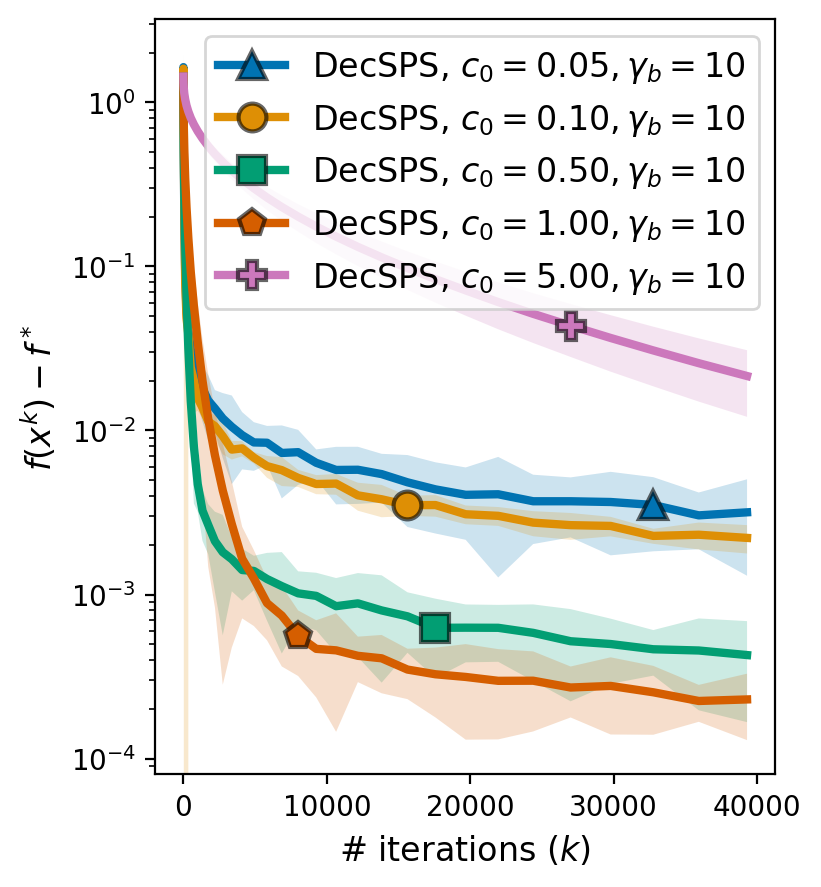}
     \includegraphics[width = 0.23\textwidth]{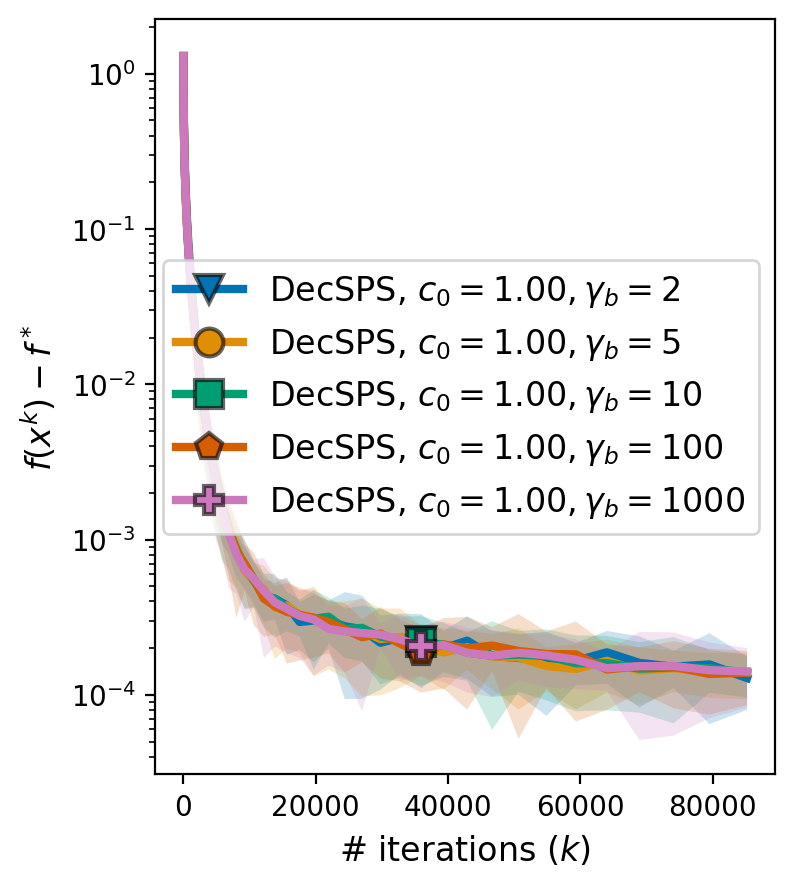}
    \vspace{-2.5mm}
    \caption{\small DecSPS~($c_k=c_0\sqrt{k+1}$) sensitivity to hyperparameters on the \textit{Synthetic Dataset}, with $\lambda=0$. Repeated 10 times and plotted is mean and std.}
    \vspace{-2mm}
    \label{fig:sensitivity}
\end{wrapfigure}
\paragraph{Comparison with vanilla SGD with decreasing stepsize.} We compare the performance of DecSPS against the classical decreasing SGD stepsize $\eta/\sqrt{k+1}$, which guarantees convergence to the exact solution at the same asymptotic rate as DecSPS. We show that, while the asymptotics are the same, DecSPS with hyperparameters $c_0=1, \gamma_b=10$ performs competitively to a fine-tuned $\eta$ --- where crucially the optimal value of $\eta$ depends on the problem. This behavior is shown on all the considered datasets, and is reported in Figures~\ref{Fig_SGDVsDecSPS_A1A}~(\textit{Breast} and \textit{Synthetic} reported in the appendix for space constraints). If lower regularization~($1e-4, 1e-6$) is considered, then DecSPS can still match the performance of tuned SGD --- but further tuning is needed~(see Figure~\ref{fig_light_reg}. Specifically, since the non-regularized problems do not have strong curvature, we found that DecSPS works best with a much higher $\gamma_b$ parameter and $c_0 = 0.05$.
\vspace{-3mm}
\paragraph{DecSPS yields a truly adaptive stepsize.} We inspect the value of $\gamma_k$ returned by DecSPS, shown in Figures~\ref{Fig_SGDVsDecSPS_A1A}~\&~\ref{Fig_SGDVsDecSPS_app}~(in the appendix). Compared to the vanilla SGD stepsize $\eta/\sqrt{k+1}$, a crucial difference appears: $\gamma_k$ decreases faster than $\mathcal{O}(1/\sqrt{k})$. This showcases that, while the factor $\sqrt{k+1}$ can be found in the formula of DecSPS\footnote{We pick $c_k=c_0\sqrt{k+1}$, as suggested by Cor.~\ref{cor:sqrt_smooth}~\&~\ref{cor:sqrt_NS} }, the algorithm structure provides additional adaptation to curvature. Indeed, in~(regularized) logistic regression, the local gradient Lipschitz constant increases as we approach the solution. Since the optimal stepsize for steadily-decreasing SGD is $1/(L\sqrt{k+1})$, where $L$ is the global Lipschitz constant~\cite{ghadimi2013stochastic}, it is pretty clear that $\eta$ should be decreased over training for optimal converge~(as $L$ effectively increases). This is precisely what DecSPS is doing.
\vspace{-4mm}
\paragraph{Comparison with AdaGrad stepsizes.} Last, we compare DecSPS with another adaptive coordinate-independent stepsize with strong theoretical guarantees: the norm version of AdaGrad~(a.k.a. AdaGrad-Norm, AdaNorm), which guarantees the exact solution at the same asymptotic rate as DecSPS~\cite{ward2019adagrad}. AdaGrad-norm at each iteration updates the scalar $b_{k+1}^2=b_{k}^2+\|\nabla f_{\SC_k}(x_k)\|^2$ and then selects the next step as $x_{k+1}=x_{k}-\frac{\eta}{b_{k+1}}\nabla f_i(x_k)$. Hence, it has tuning parameters $b_0$ and $\eta$. In Fig.~\ref{Fig_SGDVsDecSPS_A1A} we show that, on the Breast Cancer dataset, after fixing $b_0=0.1$ as recommended in~\citep{ward2019adagrad}~(see their Figure 3), tuning $\eta$ cannot quite match the performance of DecSPS. This behavior is also observed on the other two datasets we consider~(see Fig.~\ref{Fig_SGDVsADA_sklearn} in the Appendix). Last, in Fig.~\ref{fig:synthetic_ada_tuning_1}\&~\ref{fig:synthetic_ada_tuning_2} in the Appendix, we show that further tuning of $b_0$ likely does not yield a substantial improvement.
\vspace{-2mm}

\begin{figure}[ht]
\vspace{1mm}
    \centering
\textbf{\scriptsize \qquad \qquad\quad Regularized A1A -- SGD Vs DecSPS \quad \qquad\quad Regularized Breast Cancer -- AdaGrad-Norm VS DecSPS}\\
    \includegraphics[height=0.24\linewidth]{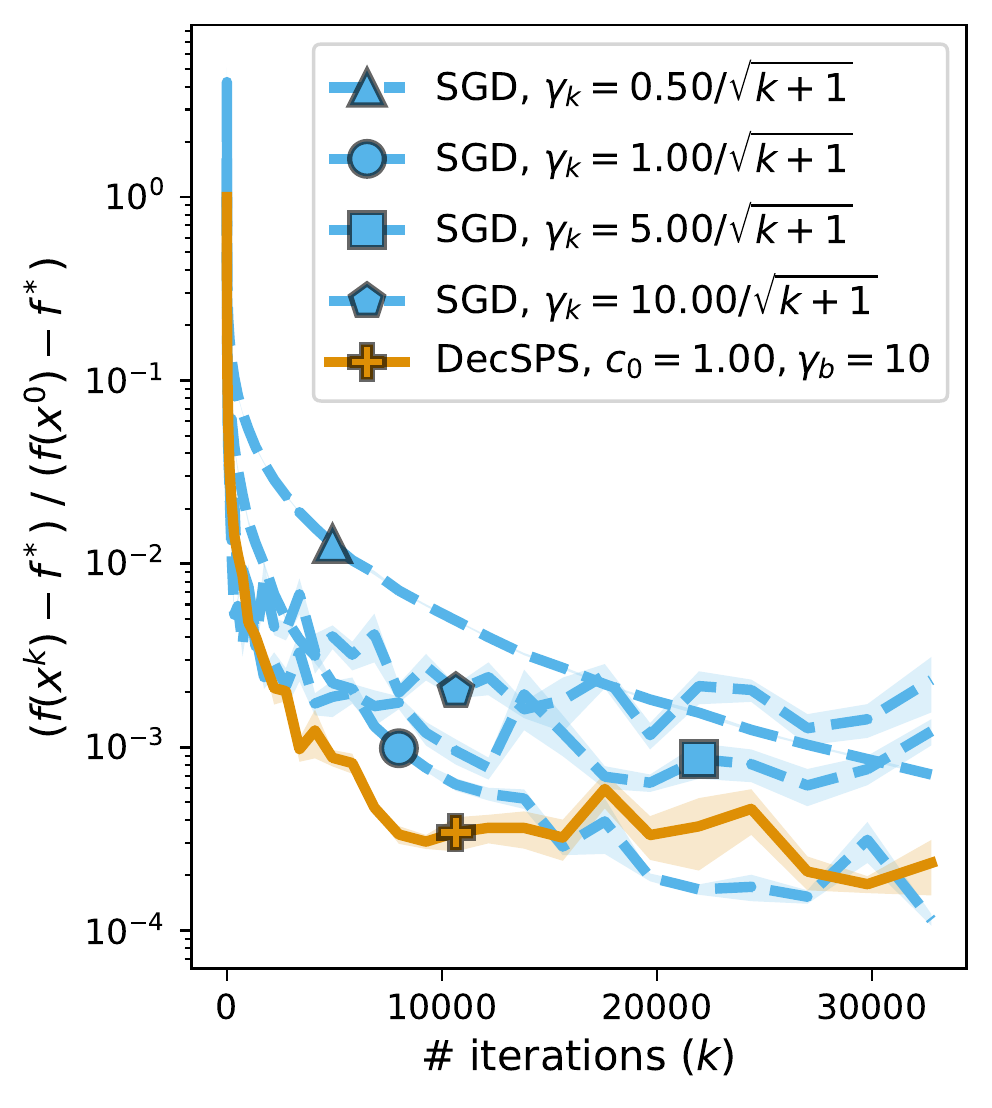}
    \includegraphics[height=0.24\linewidth]{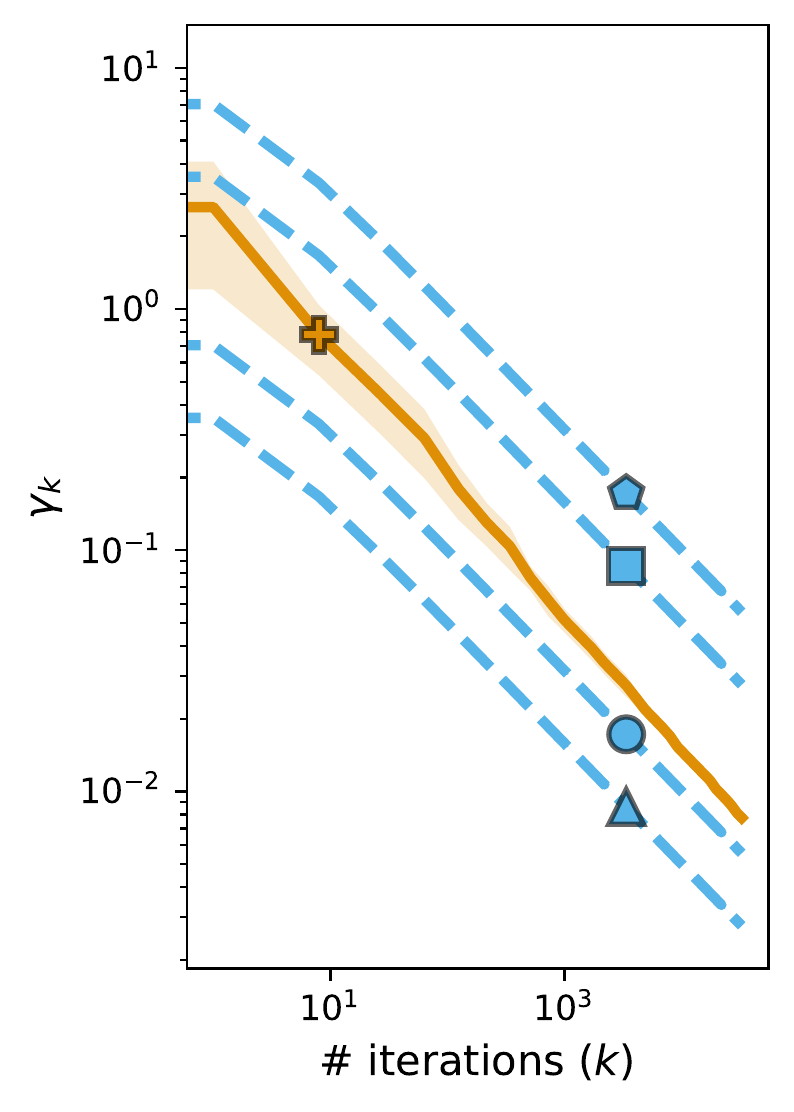}
    \includegraphics[height=0.24\linewidth]{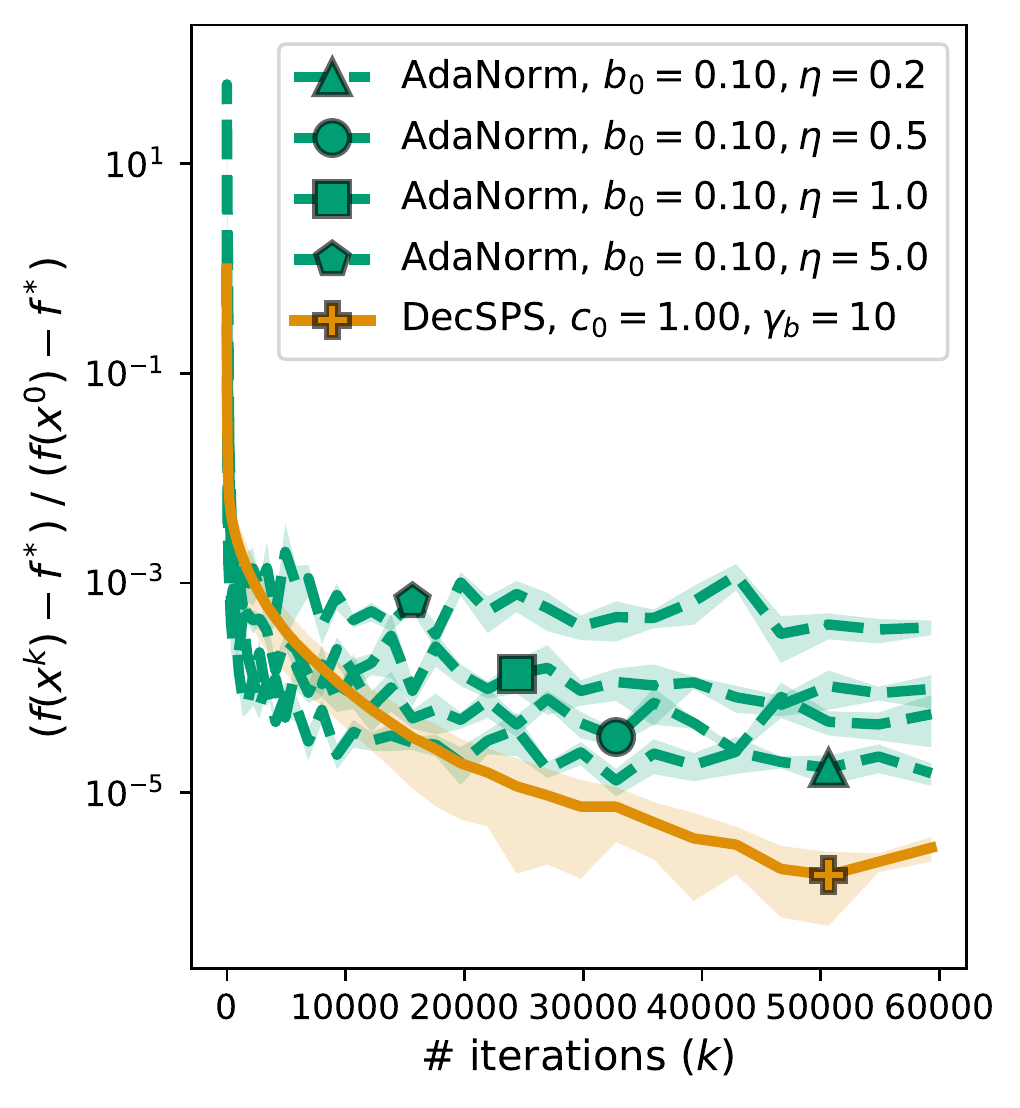}
    \includegraphics[height=0.24\linewidth]{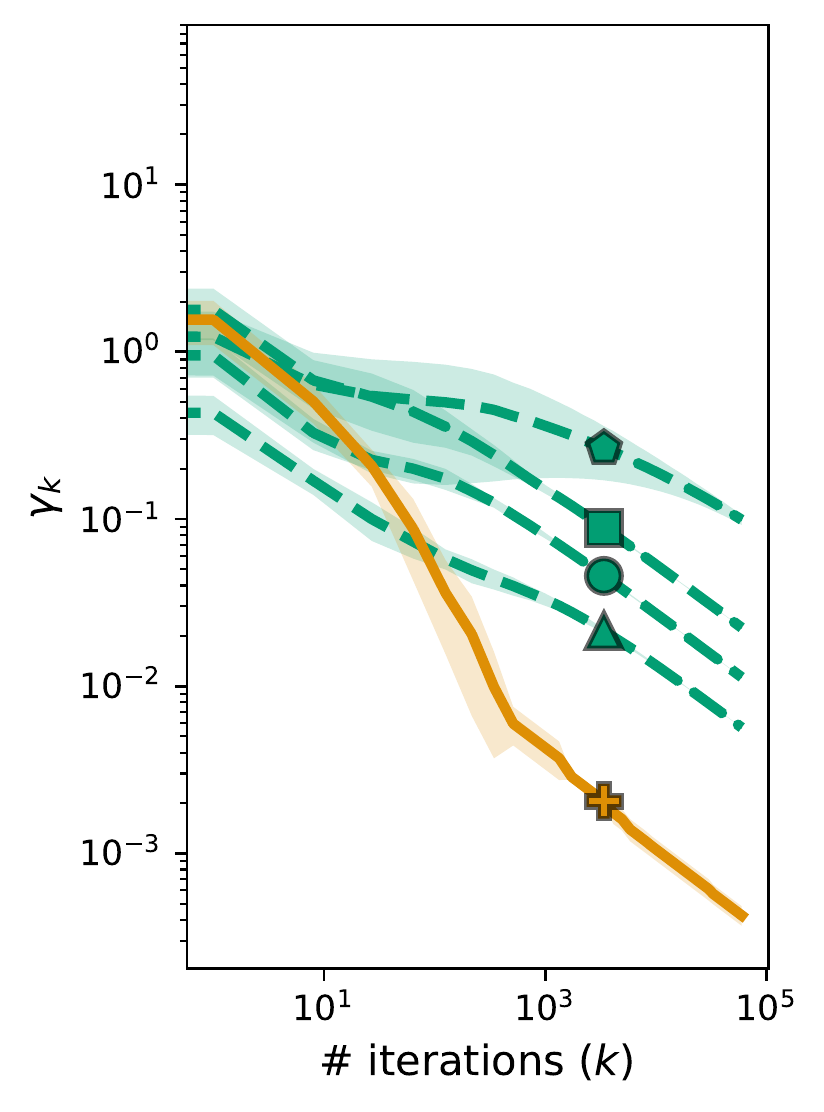}
    \vspace{-2mm}
    \caption{\small \textbf{Left: } performance of DecSPS, on the A1A Dataset~($\lambda = 0.01$).\textbf{ Right:} performance of DecSPS on the Breast Cancer Dataset~($\lambda = 1e-1$). Further experiments can be found in \S\ref{app:exp}~(appendix).}
    \label{Fig_SGDVsDecSPS_A1A}
    \vspace{-1mm}
\end{figure}

\vspace{-4mm}
\paragraph{Comparison with Adam and AMSgrad without momentum.} In Figures~\ref{Fig_SPSVsAdam}\&\ref{Fig_SGDVsAdam_A1A}\&\ref{Fig_SGDVsAdam_Breast} we compare DecSPS with Adam~\cite{kingma2014adam} and AMSgrad~\cite{reddi2019convergence} on the A1A and Breast Cancer datasets. Results show that DecSPS with the usual hyperparameters is comparable to the fine-tuned version of both these algorithms --- which however do not enjoy convergence guarantees in the unbounded domain setting.
\begin{figure}[ht]
    \centering
\textbf{\scriptsize \qquad \qquad Regularized A1A -- DecSPS Vs Adam \qquad \qquad \quad Regularized Breast Cancer -- DecSPS Vs AMSgrad}\\
    \includegraphics[height=0.24\linewidth]{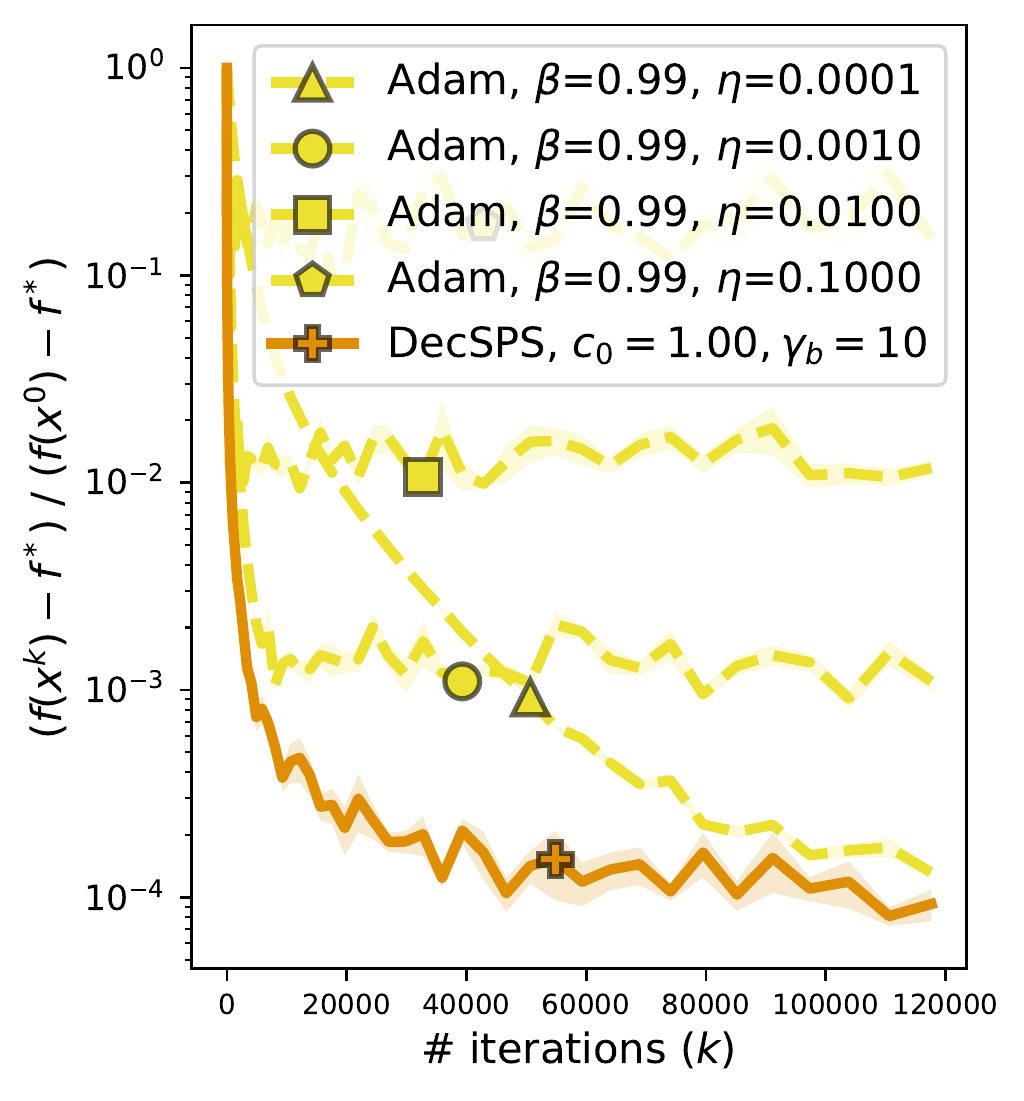}
    \includegraphics[height=0.24\linewidth]{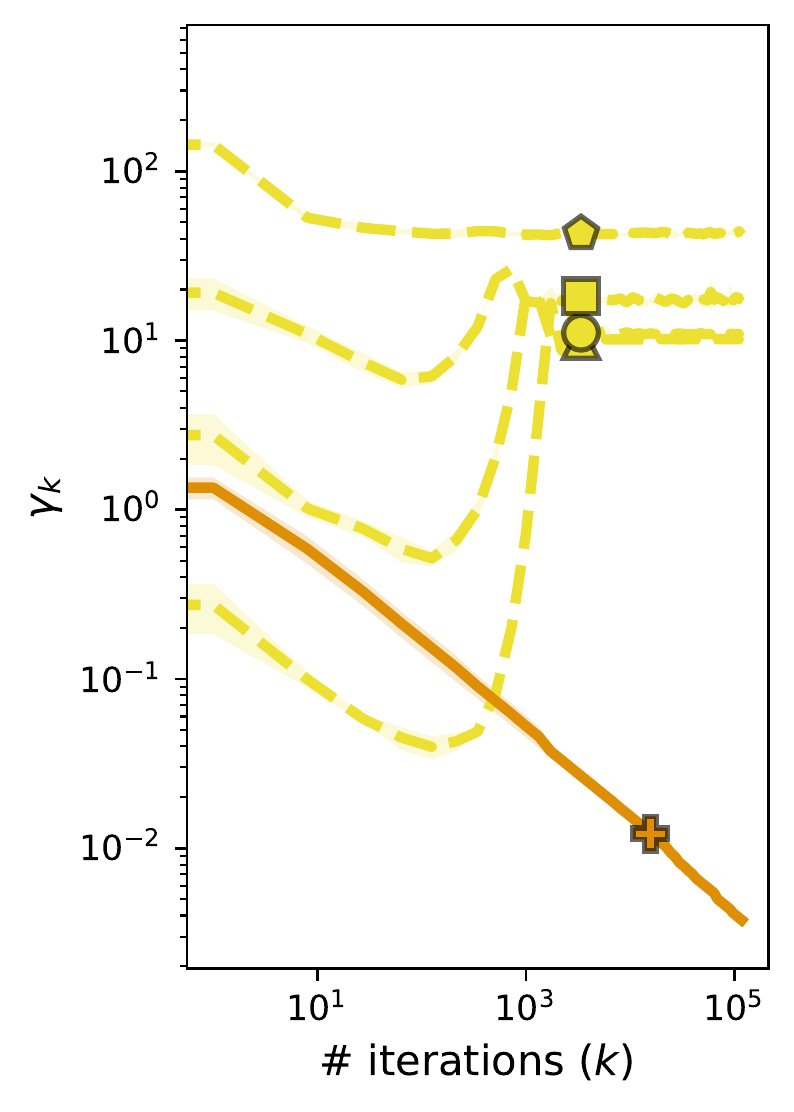}
\includegraphics[height=0.24\linewidth]{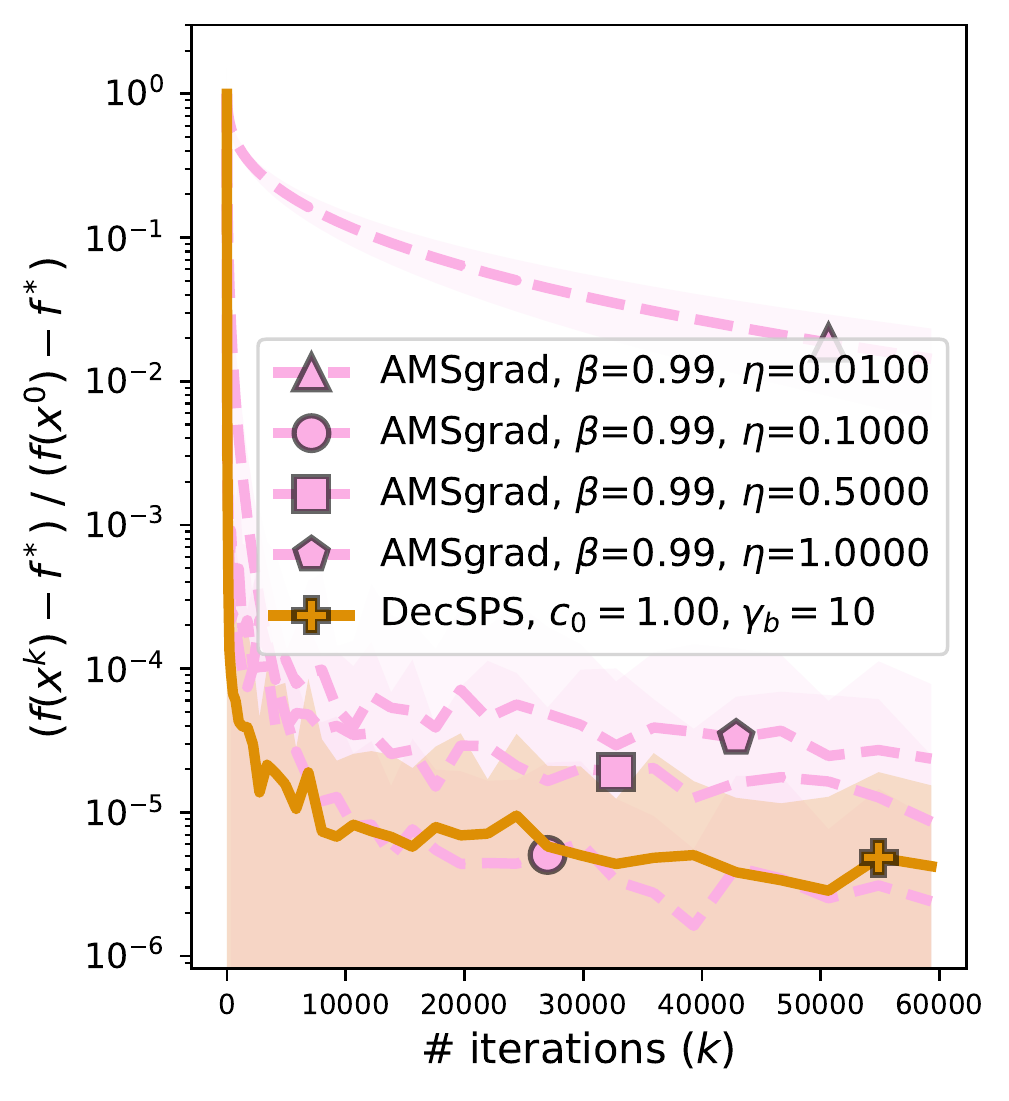}
    \includegraphics[height=0.24\linewidth]{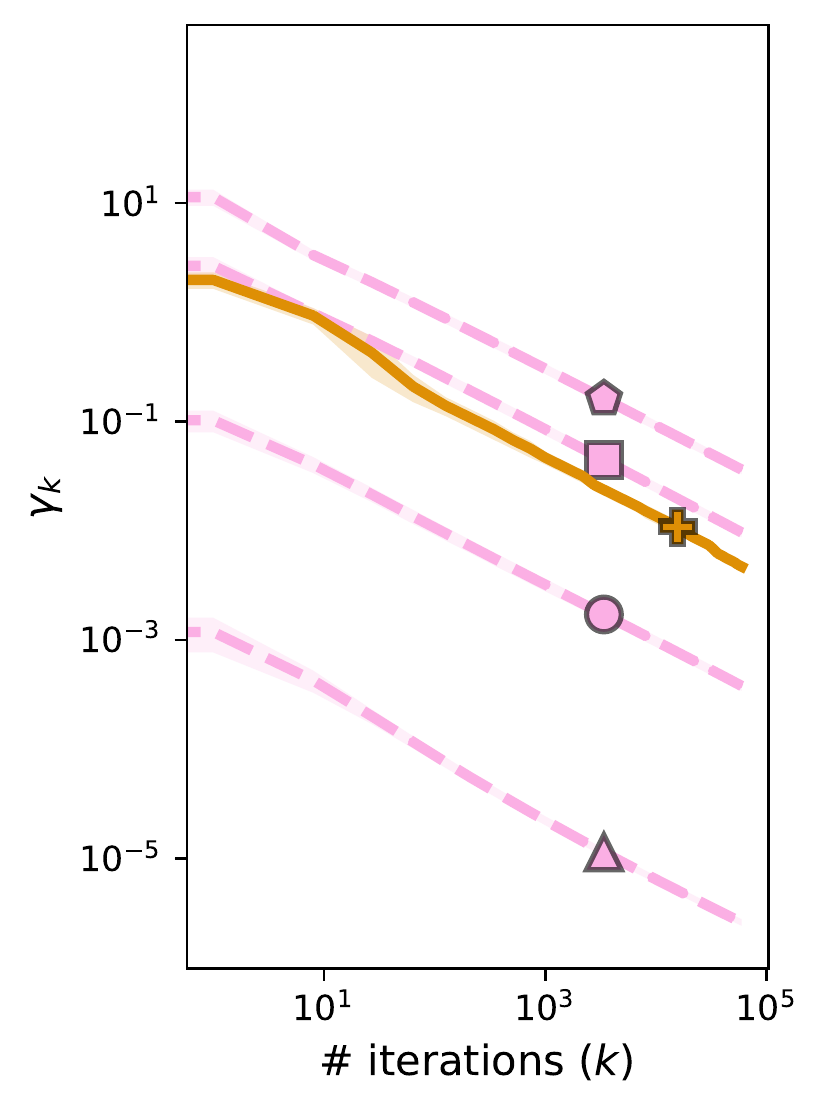}

    \vspace{-2mm}
\caption{\small \textbf{Left}: Performance of Adam~(with fixed stepsize and no momentum) and \textbf{Right}: AMSgrad~(with sqrt decreasing stepsize and no momentum) compared to DecSPS on the A1A and Breast Cancer dataset, respectively. Plots comparing the performance of Adam with DecSPS on the Breast Cancer Dataset can be found in Figure~\ref{Fig_SGDVsAdam_Breast}, and plots comparing  AMSgrad with DecSPS on the A1A Dataset can be found in Figure~\ref{Fig_SGDVsAdam_A1A}. Plotted is also the average stepsize~(each parameter evolves with a different stepsize).}    
\label{Fig_SPSVsAdam}
\end{figure}
\vspace{-4mm}
\section{Conclusions and Future Work} \vspace{-2mm}
We provided a practical variant of SPS~\citep{loizou2021stochastic}, which converges to the true problem solution without the interpolation assumption in convex stochastic problems --- matching the rate of AdaGrad. If in addition, strong convexity is assumed, then we show how, in contrast to current results for AdaGrad, the bounded iterates assumption can be dropped. The main open direction is a proof of a faster rate $\mathcal{O}(1/K)$ under strong convexity. Other possible extensions of our work include using the proposed new variants of SPS with accelerated methods, studying further the effect of mini-batching and non-uniform sampling of DecSPS, and extensions to the distributed and decentralized settings.

\section*{Acknowledgements}
This work was partially supported by the Canada CIFAR AI Chair Program. Simon Lacoste-Julien is a CIFAR Associate Fellow in the Learning in Machines \& Brains program.

\bibliographystyle{abbrvnat}
\bibliography{paper}

\begin{thebibliography}{35}
\providecommand{\natexlab}[1]{#1}
\providecommand{\url}[1]{\texttt{#1}}
\expandafter\ifx\csname urlstyle\endcsname\relax
  \providecommand{\doi}[1]{doi: #1}\else
  \providecommand{\doi}{doi: \begingroup \urlstyle{rm}\Url}\fi

\bibitem[Asi and Duchi(2019{\natexlab{a}})]{asi2019importance}
H.~Asi and J.~C. Duchi.
\newblock The importance of better models in stochastic optimization.
\newblock \emph{Proceedings of the National Academy of Sciences}, 116\penalty0
  (46):\penalty0 22924--22930, 2019{\natexlab{a}}.

\bibitem[Asi and Duchi(2019{\natexlab{b}})]{asi2019stochastic}
H.~Asi and J.~C. Duchi.
\newblock Stochastic (approximate) proximal point methods: Convergence,
  optimality, and adaptivity.
\newblock \emph{SIAM Journal on Optimization}, 29\penalty0 (3):\penalty0
  2257--2290, 2019{\natexlab{b}}.

\bibitem[Berrada et~al.(2020)Berrada, Zisserman, and
  Kumar]{berrada2019training}
L.~Berrada, A.~Zisserman, and M.~P. Kumar.
\newblock Training neural networks for and by interpolation.
\newblock In \emph{International Conference on Machine Learning}, 2020.

\bibitem[Bottou et~al.(2018)Bottou, Curtis, and
  Nocedal]{bottou2018optimization}
L.~Bottou, F.~E. Curtis, and J.~Nocedal.
\newblock Optimization methods for large-scale machine learning.
\newblock \emph{Siam Review}, 60\penalty0 (2):\penalty0 223--311, 2018.

\bibitem[Boyd et~al.(2003)Boyd, Xiao, and Mutapcic]{boyd2003subgradient}
S.~Boyd, L.~Xiao, and A.~Mutapcic.
\newblock Subgradient methods.
\newblock \emph{Lecture Notes of EE392o, Stanford University, Autumn Quarter},
  2004:\penalty0 2004--2005, 2003.

\bibitem[Chang(2011)]{chang2011libsvm}
C.-C. Chang.
\newblock {LIBSVM}: a library for support vector machines.
\newblock \emph{ACM Transactions on Intelligent Systems and Technology}, 2011.

\bibitem[Davis et~al.(2018)Davis, Drusvyatskiy, MacPhee, and
  Paquette]{davis2018subgradient}
D.~Davis, D.~Drusvyatskiy, K.~J. MacPhee, and C.~Paquette.
\newblock Subgradient methods for sharp weakly convex functions.
\newblock \emph{Journal of Optimization Theory and Applications}, 179\penalty0
  (3):\penalty0 962--982, 2018.

\bibitem[D{\'e}fossez et~al.(2020)D{\'e}fossez, Bottou, Bach, and
  Usunier]{defossez2020simple}
A.~D{\'e}fossez, L.~Bottou, F.~Bach, and N.~Usunier.
\newblock A simple convergence proof of {Adam} and {Adagrad}.
\newblock \emph{arXiv preprint arXiv:2003.02395}, 2020.

\bibitem[D'Orazio et~al.(2021)D'Orazio, Loizou, Laradji, and
  Mitliagkas]{d2021stochastic}
R.~D'Orazio, N.~Loizou, I.~Laradji, and I.~Mitliagkas.
\newblock Stochastic mirror descent: Convergence analysis and adaptive variants
  via the mirror stochastic {Polyak} stepsize.
\newblock \emph{arXiv preprint arXiv:2110.15412}, 2021.

\bibitem[Dua and Graff(2017)]{Dua:2019}
D.~Dua and C.~Graff.
\newblock {UCI} machine learning repository, 2017.
\newblock URL \url{http://archive.ics.uci.edu/ml}.

\bibitem[Duchi et~al.(2011)Duchi, Hazan, and Singer]{duchi2011adaptive}
J.~Duchi, E.~Hazan, and Y.~Singer.
\newblock Adaptive subgradient methods for online learning and stochastic
  optimization.
\newblock \emph{Journal of Machine Learning Research}, 12\penalty0 (7), 2011.

\bibitem[Ene and Nguyen(2020)]{ene2020adaptive}
A.~Ene and H.~L. Nguyen.
\newblock Adaptive and universal algorithms for variational inequalities with
  optimal convergence.
\newblock \emph{arXiv preprint arXiv:2010.07799}, 2020.

\bibitem[Ghadimi and Lan(2013)]{ghadimi2013stochastic}
S.~Ghadimi and G.~Lan.
\newblock Stochastic first-and zeroth-order methods for nonconvex stochastic
  programming.
\newblock \emph{SIAM Journal on Optimization}, 23\penalty0 (4), 2013.

\bibitem[Goodfellow et~al.(2016)Goodfellow, Bengio, and
  Courville]{goodfellow2016deep}
I.~Goodfellow, Y.~Bengio, and A.~Courville.
\newblock \emph{Deep {L}earning}.
\newblock MIT press, 2016.

\bibitem[Gower et~al.(2021{\natexlab{a}})Gower, Sebbouh, and
  Loizou]{gower2021sgd}
R.~Gower, O.~Sebbouh, and N.~Loizou.
\newblock {SGD} for structured nonconvex functions: Learning rates,
  minibatching and interpolation.
\newblock In \emph{International Conference on Artificial Intelligence and
  Statistics}, 2021{\natexlab{a}}.

\bibitem[Gower et~al.(2019)Gower, Loizou, Qian, Sailanbayev, Shulgin, and
  Richt{\'a}rik]{gower2019sgd}
R.~M. Gower, N.~Loizou, X.~Qian, A.~Sailanbayev, E.~Shulgin, and
  P.~Richt{\'a}rik.
\newblock {SGD}: General analysis and improved rates.
\newblock In \emph{International Conference on Machine Learning}, 2019.

\bibitem[Gower et~al.(2021{\natexlab{b}})Gower, Defazio, and
  Rabbat]{gower2021stochastic}
R.~M. Gower, A.~Defazio, and M.~Rabbat.
\newblock Stochastic {Polyak} stepsize with a moving target.
\newblock \emph{arXiv preprint arXiv:2106.11851}, 2021{\natexlab{b}}.

\bibitem[Hazan and Kakade(2019)]{hazan2019revisiting}
E.~Hazan and S.~Kakade.
\newblock Revisiting the {Polyak} step size.
\newblock \emph{arXiv preprint arXiv:1905.00313}, 2019.

\bibitem[Kingma and Ba(2014)]{kingma2014adam}
D.~P. Kingma and J.~Ba.
\newblock Adam: A method for stochastic optimization.
\newblock \emph{arXiv:1412.6980}, 2014.

\bibitem[Kushner and Yin(2003)]{kushner2003stochastic}
H.~Kushner and G.~G. Yin.
\newblock \emph{Stochastic {A}pproximation and {R}ecursive {A}lgorithms and
  {A}pplications}, volume~35.
\newblock Springer Science \& Business Media, 2003.

\bibitem[Levy et~al.(2018)Levy, Yurtsever, and Cevher]{levy2018online}
K.~Y. Levy, A.~Yurtsever, and V.~Cevher.
\newblock Online adaptive methods, universality and acceleration.
\newblock \emph{Advances in Neural Information Processing Systems}, 31, 2018.

\bibitem[Loizou et~al.(2021)Loizou, Vaswani, Laradji, and
  Lacoste-Julien]{loizou2021stochastic}
N.~Loizou, S.~Vaswani, I.~H. Laradji, and S.~Lacoste-Julien.
\newblock Stochastic {Polyak} step-size for {SGD}: An adaptive learning rate
  for fast convergence.
\newblock In \emph{International Conference on Artificial Intelligence and
  Statistics}, 2021.

\bibitem[Nemirovski et~al.(2009)Nemirovski, Juditsky, Lan, and
  Shapiro]{nemirovski2009robust}
A.~Nemirovski, A.~Juditsky, G.~Lan, and A.~Shapiro.
\newblock Robust stochastic approximation approach to stochastic programming.
\newblock \emph{SIAM Journal on Optimization}, 19\penalty0 (4):\penalty0
  1574--1609, 2009.

\bibitem[Oberman and Prazeres(2019)]{oberman2019stochastic}
A.~M. Oberman and M.~Prazeres.
\newblock Stochastic gradient descent with {Polyak}'s learning rate.
\newblock \emph{arXiv preprint arXiv:1903.08688}, 2019.

\bibitem[Polyak(1987)]{polyak1987introduction}
B.~Polyak.
\newblock \emph{Introduction to {O}ptimization}.
\newblock Inc., Publications Division, New York, 1987.

\bibitem[Reddi et~al.(2018)Reddi, Kale, and Kumar]{reddi2019convergence}
S.~J. Reddi, S.~Kale, and S.~Kumar.
\newblock On the convergence of {Adam} and beyond.
\newblock In \emph{International Conference on Learning Representations}, 2018.

\bibitem[Robbins and Monro(1951)]{robbins1951stochastic}
H.~Robbins and S.~Monro.
\newblock A stochastic approximation method.
\newblock \emph{The Annals of Mathematical Statistics}, 1951.

\bibitem[Rolinek and Martius(2018)]{rolinek2018l4}
M.~Rolinek and G.~Martius.
\newblock L4: Practical loss-based stepsize adaptation for deep learning.
\newblock \emph{Advances in neural information processing systems}, 31, 2018.

\bibitem[Tieleman and Hinton(2012)]{tieleman2012lecture}
T.~Tieleman and G.~Hinton.
\newblock Lecture 6.5 - {RMSprop}, {Coursera}: Neural networks for machine
  learning.
\newblock \emph{University of Toronto, Technical Report}, 2012.

\bibitem[Traor{\'e} and Pauwels(2021)]{traore2021sequential}
C.~Traor{\'e} and E.~Pauwels.
\newblock Sequential convergence of {AdaGrad} algorithm for smooth convex
  optimization.
\newblock \emph{Operations Research Letters}, 49\penalty0 (4):\penalty0
  452--458, 2021.

\bibitem[Vaswani et~al.(2020)Vaswani, Laradji, Kunstner, Meng, Schmidt, and
  Lacoste-Julien]{vaswani2020adaptive}
S.~Vaswani, I.~Laradji, F.~Kunstner, S.~Y. Meng, M.~Schmidt, and
  S.~Lacoste-Julien.
\newblock Adaptive gradient methods converge faster with over-parameterization
  (but you should do a line-search).
\newblock \emph{arXiv preprint arXiv:2006.06835}, 2020.

\bibitem[Ward et~al.(2019)Ward, Wu, and Bottou]{ward2019adagrad}
R.~Ward, X.~Wu, and L.~Bottou.
\newblock Adagrad stepsizes: Sharp convergence over nonconvex landscapes.
\newblock In \emph{International Conference on Machine Learning}, 2019.

\bibitem[Xie et~al.(2020)Xie, Wu, and Ward]{xie2020linear}
Y.~Xie, X.~Wu, and R.~Ward.
\newblock Linear convergence of adaptive stochastic gradient descent.
\newblock In \emph{International Conference on Artificial Intelligence and
  Statistics}, 2020.

\bibitem[Zhang et~al.(2021)Zhang, Bengio, Hardt, Recht, and
  Vinyals]{zhang2021understanding}
C.~Zhang, S.~Bengio, M.~Hardt, B.~Recht, and O.~Vinyals.
\newblock Understanding deep learning (still) requires rethinking
  generalization.
\newblock \emph{Communications of the ACM}, 64\penalty0 (3):\penalty0 107--115,
  2021.

\bibitem[Zhang et~al.(2020)Zhang, Karimireddy, Veit, Kim, Reddi, Kumar, and
  Sra]{zhang2019adaptive}
J.~Zhang, S.~P. Karimireddy, A.~Veit, S.~Kim, S.~Reddi, S.~Kumar, and S.~Sra.
\newblock Why are adaptive methods good for attention models?
\newblock \emph{Advances in Neural Information Processing Systems}, 33, 2020.

\end{thebibliography}

\newpage
\appendix

\begin{center}
\Large
    \textit{Supplementary Material}  \\ Dynamics of SGD with Stochastic Polyak Stepsizes:\\ Truly Adaptive Variants and Convergence to Exact Solution
 \end{center}

The appendix is organized as follows:
\begin{enumerate}
\item In \S\ref{sec:app_related_work} we provide a more detail comparison with closely related works.
\item In \S\ref{sec:app_pre} we present some technical preliminaries.
    \item In \S\ref{sec:app_proof1}, we present convergence guarantees for SPS$_{\text{max}}$ after replacing $f_S^*$ with $\ell_S^*$ (lower bound).
    \item In \S\ref{sec:app_proof2}, we discuss the lack of convergence of SPS$_{\text{max}}$ in the non-interpolated setting.
    \item In \S\ref{sec:app_proof3}, we discuss convergence of DecSPS, our convergent variant of SPS.
    \item In \S\ref{sec:app_proof4}, we discuss convergence of DecSPS-NS, our convergent variant of SPS in the non-smooth setting.
    \item In \S\ref{app:exp} we provide some additional experimental results and describe the datasets in detail.
\end{enumerate}
\section{Comparison with closely related work}
\label{sec:app_related_work}

In this section we present a more detailed comparison to closely related works on stochastic variants of the Polyak stepsize. We start with the work of \citet{asi2019stochastic}, and then continue with a brief presentation of other papers (already presented in \citet{loizou2021stochastic}).

\subsection[]{Comparison with~\citet{asi2019stochastic}} \citet{asi2019stochastic} proposed the following adaptive method for solving Problem~\eqref{MainProb} \textit{under the interpolation assumption} $\inf_{x\in\mathbb{R}^d} f_\SC(x) = f_\SC(x^*)=0$ for all subsets $\SC$ of $[n]$ with $|\SC|=B$:
\begin{equation}
    x^{k+1} = x^k - \min\left\{\alpha_k,\frac{f_{\SC_k}(x^k)}{\|\nabla f_{\SC_k}(x^k)\|^2}\right\}\nabla f_{\SC_k}(x^k),
    \label{eq:duchi_step}
\end{equation}
where $\alpha_k = \alpha_0 k^{-\beta}$ for some $\beta 
\in \mathbb{R} $, is a polynomially decreasing/increasing sequence. We provide here a full comparison of this stepsize with SPS$_{\text{max}}$ and DecSPS.

\paragraph{Comparison with the adaptive stepsizes in \citet{loizou2021stochastic} and our DecSPS.} In~\citet{loizou2021stochastic}, the proposed SPS$_{\max}$ stepsize is
\begin{equation}
    \gamma_k = \min \left\{ \frac{f_{\SC_k}(x^k)-f_{\SC_k}^*}{c\|\nabla f_{\SC_k}(x^k)\|^2}, \gamma_b \right\}.
\end{equation}
This stepsize is similar to the one in~\citet{asi2019stochastic}: in both, a Polyak-like stochastic stepsize is bounded from above in order to guarantee convergence. However there are crucial differences.
\begin{itemize}[leftmargin=*]
    \item SPS$_{\max}$~\cite{loizou2021stochastic} can be applied to non-interpolated problems and leads to fast convergence to a ball around the solution in the non-interpolated setting~(see Theorem~\ref{thm:loizou}). Instead, ~\citet{asi2019importance} only formulated and studied Eq.~\eqref{eq:duchi_step} in the interpolated setting.
    \item As we will see in the next paragraph, one can formulate few conditions under which it is possible to derive linear convergence rates for Eq.~\eqref{eq:duchi_step} in the interpolated setting. As can be easily seen from Theorem~\ref{thm:loizou}, \textit{SPS$_{\max}$ has similar convergence guarantees but works under a more standard/restrictive set of assumptions}. In particular, in the interpolated setting, while~\citet{asi2019stochastic} require some specific assumptions on the noise statistics~(see next paragraph), the rates in~\citet{loizou2021stochastic} can be applied without the need for, e.g., probabilistic bound on the gradient magnitude.
\end{itemize}

In this paper, starting from the SPS$_{\max}$ algorithm we propose the following stepsize for convergence to the \textit{exact solution} in the \textit{non-interpolated setting}:

\begin{equation}
\tag{DecSPS}
{\small
 \gamma_k:=\frac{1}{c_k}\min\left\{\frac{f_{\SC_k}(x^k)-\ell^*_{\SC_k}}{\|\nabla f_{\SC_k}(x^k)\|^2}, \ c_{k-1}\gamma_{k-1}\right\},}
\end{equation}
where $c_k$ is an increasing sequence~(e.g. $c_k=\sqrt{k+1}$, see Theorem~\ref{thm:no_bound}), and $\ell_{\SC_k}^*$ is any lower bound on $f_{\SC_k}^*$. At initialization, we set $c_{-1}=c_0$ and $\gamma_-1 = \gamma_b>0$.

We now compare DecSPS with Eq.~\eqref{eq:duchi_step} and our results with the rates in~\citep{asi2019stochastic}.
\begin{itemize}[leftmargin=*]
    \item \textit{Convergence rates}: The form of Eq.~\eqref{eq:duchi_step} and the convergence guarantees of~\cite{asi2019stochastic} are \textit{restricted to the interpolated setting. Instead, in this paper we focus on the non-interpolated setting}: using DecSPS we provided the first stochastic adaptive optimization method that converges in the non-interpolated setting to the exact solution without restrictive assumptions like bounded iterates/gradients.
    \item \textit{Inspection of the stepsize}: DecSPS provides a version of SPS where $\gamma_k$ is steadily decreasing and is upper bounded by the decreasing quantity $c_0\gamma_b/c_k$, where $c_k = \sqrt{k+1}$ yields the optimal asymptotic rate~(Theorem~\ref{thm:no_bound}) . Hence, DecSPS can be compared to Eq.~\eqref{eq:duchi_step} for $\alpha_k = \alpha_0/\sqrt{k+1}$. However, note that there are two fundamental differences: First, in DecSPS we have that $\gamma_k\ge\gamma_{k-1}$~(see Lemma~\ref{lemma:easy_bounds}), a feature which Eq.~\eqref{eq:duchi_step} does not have. Secondly, compared to our DecSPS, the stepsize in Eq.~\eqref{eq:duchi_step} with $\alpha_k$ decreasing polynomially is \textit{asymptotically non-adaptive}. Indeed, assuming that each $f_i$ has $L_i$-Lipschitz gradients and that each $f_{\SC}^*$ is non-negative, we have~(see~\cite{loizou2021stochastic}) that
\begin{equation}
     \frac{f_{\SC_k}(x^k)}{\|\nabla f_{\SC_k}(x^k)\|^2}\ge\frac{f_{\SC_k}(x^k)-f_{\SC_k}^*}{\|\nabla f_{\SC_k}(x^k)\|^2}\ge\frac{1}{2L_{\max}},
\end{equation}
therefore, after $\log_\beta(2L_{\max}\alpha_0)$ iterations\footnote{Simply plugging in $\alpha_k = \alpha_0 k^{-\beta}$ and solving for $k$.} the algorithm coincides with SGD with stepsize $\alpha_k$.

\end{itemize}

For completeness, we provide in the next paragraph an overview of the results in~\citet{asi2019stochastic}.

\paragraph{Precise theoretical guarantees in~\citet{asi2019stochastic}.} The stepsize in Equation~\eqref{eq:duchi_step} yields linear convergence guarantees under a specific set of assumptions. We summarize the two main results of \citet{asi2019stochastic} below in the case of differentiable losses\footnote{The results of~\citet{asi2019stochastic} also work in the subdifferentiable setting.}:

\begin{restatable}[Proposition 2 in~\citet{asi2019stochastic}]{prop}{duchi1}
\label{lemma:duchi1}
Let each $f_i$ be a convex and differentiable function which satisfies a specific set of technical assumptions~(see conditions C.i and C.iii in~\citep{asi2019stochastic}). For a fixed batch-size $B$ assume $\inf_{x\in\mathbb{R}^d} f_\SC(x) = f_\SC(x^*)=0$ for all subsets $\SC$ of $[n]$ with $|\SC|=B$~(i.e. interpolation). Assume in addition that there exist constants $\lambda_0,\lambda_1>0$ such that for
all $\alpha>0$, $x\in\mathbb{R}^d$ and $x^*\in\mathcal{X}^*$~(set of solutions) we have~(sharp growth with shared minimizers assumption)
\begin{equation}
    \mathbb{E}_{\SC}\left[\min\left\{\alpha[f_{\SC}(x)-f_{\SC}^*], \frac{(f_{\SC}(x)-f_{\SC}^*)^2}{\|\nabla f_{\SC}(x)\|^2}\right\}\right]\ge\min\{\gamma_0\alpha,\lambda_1 \|x-x^*\|\}\cdot\|x-x^*\|.
    \label{ass_duchi}
\end{equation}
Then, for $\alpha_k = \alpha_0 k^{-\beta}$ with $\beta \in (-\infty,1)$ the stepsize of Equation~\eqref{eq:duchi_step} yields a linear convergence rate dependent on $ \lambda_1$ and the choice of $\beta$.
\end{restatable}
Sufficient conditions for Equation~\eqref{ass_duchi} to hold is that there exist $\lambda,p>0$ such that, for all $x\in\mathbb{R}^d$, $\mathbb{P}_{\SC}[f_{\SC}(x)-f_{\SC}(x^*)]\ge \lambda \|x-x^*\|^2\ge p$ and $\mathbb{E}_{\SC}[\|\nabla f_{\SC}(x)\|^2]\le M^2$.

\begin{restatable}[Proposition 3 in~\citet{asi2019stochastic}]{prop}{duchi2}
\label{lemma:duchi2}
Let each $f_i$ be a convex and differentiable function which satisfies a specific set of technical assumptions~(see conditions C.i and C.iii in~\citep{asi2019stochastic}). Under the same interpolation assumptions as Lemma~\ref{lemma:duchi1}, assume that there exist constants $\lambda_0, \lambda_1>0$ such that for
all $\alpha>0$, $x\in\mathbb{R}^d$ and $x^*\in\mathcal{X}^*$ we have~(quadrtic growth with shared minimizers assumption)
\begin{equation}
    \mathbb{E}_{\SC}\left[(f_{\SC}(x)-f_{\SC}^*)\cdot\min\left\{\alpha, \frac{(f_{\SC}(x)-f_{\SC}^*)^2}{\|\nabla f_{\SC}(x)\|^2}\right\}\right]\ge\min\{\alpha\lambda_0,\lambda_1\}\cdot\|x-x^*\|^2.
    \label{ass_duchi2}
\end{equation}
Then, for $\alpha_k = \alpha_0 k^{-\beta}$ with $\beta \in (-\infty,\infty)$ the stepsize of Equation~\eqref{eq:duchi_step} yields a linear convergence rate dependent on $ \lambda_0, \lambda_1$ and the choice of $\beta$.
\end{restatable}
The authors show that Equation~\eqref{ass_duchi2} holds under the assumption that the averaged loss $f$ has quadratic growth and has Lipschitz continuous gradients, if in addition there exist constants $0 < c, C < \infty $, $p > 0$ such that $\mathbb{P}_{\SC}\left[ \|\nabla f_{\SC}(x)\|^2\le C\|\nabla f(x)\|^2,\ [f_{\SC}(x)-f_{\SC}(x^*)] > c(f(x)-f^*(x))\right]\ge p$.

\subsection{Comparisons with other versions of the Polyak stepsize for stochastic problems} 
To the best of our knowledge, no prior work has provided a computationally feasible modification of the Polyak stepsize for convergence to the exact solution in stochastic non-interpolated problems. In the next lines, we outline the details for a few related works on Polyak stepsize for stochastic problems.
\begin{itemize}[leftmargin = *]
    \item \textit{SPS$_{\max}$}: As discussed in the main paper, our starting point is the SPS$_{\max}$ algorithm in~\cite{loizou2021stochastic}, which provides linear~(for strongly convex) or sublinear~(for convex) convergence to a ball around the minimizer, with size dependent of the problems' degree of interpolation. Instead, in this work, we provide convergence guarantees to the exact solution in the non-interpolated setting for a modified version of this algorithm. In addition, when compared to SPS$_{\max}$, our method does not require knowledge of the single $f_i^*$s, but just of lower bounds on these quantities~(see \S\ref{sec:estimation}).
    \item \textit{L4}: A stepsize very similar to SPS$_{\max}$~(the \textit{L4 algorithm}) was proposed back in 2018 by~\cite{rolinek2018l4}. While this stepsize results in promising performance in deep learning, (1) it has no theoretical convergence guarantees, and (2) each update requires an online estimation of the $f_i^*$, which in turn requires tuning up to three hyper-parameters.
    \item \textit{SPLR}: \citet{oberman2019stochastic} instead study convergence of SGD with the following stepsize: $\gamma_k=\frac{2[f(x^k)-f^*]}{\Exp_{i_k}\|\nabla f_{i_k}(x^k)\|^2}$, which requires  knowledge of $\Exp_{i_k}\|\nabla f_{i_k}(x^k)\|^2$ for all iterates $x^k$ and the evaluation of $f(x^k)$ --- the full-batch loss --- at each step. This makes the concrete application of SPLR problematic for sizeable stochastic problems.
    \item \textit{ALI-G}: Last, the ALI-G stepsize proposed by~\citet{berrada2019training} is $\gamma_k=\min \left\{\frac{f_i(x^k)}{\|\nabla f_i(x^k)\|^2+\delta}, \eta \right\}$, where $\delta>0$ is a tuning parameter. Unlike the SPS$_{\max}$ setting, their theoretical analysis relies on an $\epsilon$-interpolation condition. Moreover, the values of the parameter $\delta$ and $\eta$ that guarantee convergence heavily depend on the smoothness parameter of the objective $f$, limiting the method's practical applicability. In addition, in the interpolated setting, while ALI-G converges to a neighborhood of the solution, the SPS$_{\max}$ method~\cite{loizou2021stochastic} is able to provide linear convergence to the solution.
\end{itemize}

\section[]{Technical Preliminaries}
\label{sec:app_pre}

Let us present some basic definitions used throughout the paper.

\begin{definition}[Strong Convexity / Convexity]
The function $f : \R^n \rightarrow \R$, is $\mu$-strongly convex, if there exists a constant
$\mu > 0$ such that $\forall x, y \in \R^n$:
\begin{equation}
\label{stronglyconvex}
 f(x) \geq f(y)+ \langle\nabla f(y) , x-y\rangle + \frac{\mu}{2} \|x-y\|^2
\end{equation}
for all $x \in \R^d$. If inequality \eqref{stronglyconvex} holds with $\mu=0$ the function $f$ is convex.
\end{definition}

\begin{definition}[$L$-smooth]
The function $f : \R^n \rightarrow \R$, $L$-smooth, if there exists a constant
$L > 0$ such that $\forall x, y \in \R^n$:
\begin{equation}
\label{Smooth}
\|\nabla f(x)-\nabla f(y)\| \leq L \|x-y\|,
\end{equation}
or equivalently:
\begin{equation}
\label{Smooth2}
 f(x) \leq f(y)+ \langle\nabla f(y) , x-y\rangle + \frac{L}{2} \|x-y\|^2.
\end{equation}
\end{definition}

\begin{restatable}[]{lem}{basic_L_sc}
\label{lemma:basic_L_sc}
If a function $g$ is $\mu$-strongly convex and $L$-smooth the following bounds holds:
\begin{equation}
    \frac{1}{2 L} \|\nabla g(x)\|^2 \leq g(x)-\inf_x g(x)\leq \frac{1}{2\mu} \|\nabla g(x)\|^2.
\end{equation}
\end{restatable}
The following lemma is the fundamental starting point in~\citep{loizou2021stochastic}.
\begin{restatable}[]{lem}{loizou_lem}
\label{lemma:basic_loizou}

Let $f(x) = \frac{1}{n} \sum_{i=1}^n f_i(x)$ where the functions $f_i$ are $\mu_i$-strongly convex and $L_i$-smooth, then
\vspace{-1ex}
\begin{align}
\label{NewBounds}
\frac{1}{2 L_{\max}} \leq \frac{1}{2 L_{i}} \leq\frac{f_{i}(x^k)-f_i^*}{ \|\nabla f_{i}(x^k)\|^2} \leq \frac{1}{2 \mu_{i}}\leq \frac{1}{2 \mu_{\min}},
\end{align}
\vspace{-1ex}
where $f_i^*:=\inf_x f_i(x)$, $L_{\max}=\max \{L_i\}_{i=1}^n$ and $\mu_{\min}=\min \{\mu_i\}_{i=1}^n$.
\end{restatable}
\begin{proof}
Directly using Lemma~\ref{lemma:basic_L_sc}.
\end{proof}

\section[]{Convergence guarantees after replacing $\boldsymbol{f_S}^*$ in SPS$_{\text{max}}$ with $\boldsymbol{\ell_S}^*$} 
\label{sec:app_proof1}

The proofs in this subsection is an easy adaptation of the proofs appeared in~\citep{loizou2021stochastic}. To avoid redundancy in the literature, we provide skecth of the proofs showing the fundamental differences and invite the interested reader to read the details in ~\citep{loizou2021stochastic}.

\loizoubs*
\begin{proof}

\textcolor{blue}{We highlight in blue text the differences between this proof and the one in~\citep{loizou2021stochastic}.}

Recall the stepsize definition
\begin{equation}
\tag{SPS$_{\max}^\ell$}
\gamma_k = \min \left\{ \frac{f_{\SC_k}(x^k)-\ell_{\SC_k}^*}{c\|\nabla f_{\SC_k}(x^k)\|^2}, \gamma_b \right\},
\end{equation}
where $\ell_{\SC_k}^*$ is any lower bound on $f_{\SC_k}^*$. We also will make use of the bound
\begin{align}
    \frac{1}{2c L_{\SC}}\le \frac{f_{\SC_k}(x^k)-f_{\SC_k}^*}{c\|\nabla f_{\SC_k}(x^k)\|^2} \le\frac{f_{\SC_k}(x^k)-\ell_{\SC_k}^*}{c\|\nabla f_{\SC_k}(x^k)\|^2} = \gamma_k\le\frac{\gamma_b}{c}.
\end{align}

\paragraph{Convex setting.} As in \citep{loizou2021stochastic} we use a standard expansion as well as the stepsize definition.
\begin{align}
&\|x^{k+1}-x^*\|^2\\
&=\|x^k-x^*\|^2-2 \gamma_k \langle x^k-x^*, \nabla f_{\SC_k}(x^k) \rangle + \gamma_k^2 \| \nabla f_{\SC_k}(x^k)\|^2\\
&\leq\|x^k-x^*\|^2-2 \gamma_k \left[f_{\SC_k}(x^k)-f_{\SC_k}(x^*)\right] + \gamma_k^2 \| \nabla f_{\SC_k}(x^k)\|^2\\
&\leq\|x^k-x^*\|^2-2 \gamma_k \left[f_{\SC_k}(x^k)-f_{\SC_k}(x^*)\right] + \textcolor{blue}{\frac{\gamma_k }{c}[f_{\SC_k}(x^k)-\ell_{\SC_k}^*]}\\
&=\|x^k-x^*\|^2-2 \gamma_k \left[f_{\SC_k}(x^k)-f_{\SC_k}^*+f_{\SC_k}^*-f_{\SC_k}(x^*)\right] + \textcolor{blue}{\frac{\gamma_k }{c}[f_{\SC_k}(x^k)-\ell_{\SC_k}^*]}\\
\end{align}
Next, adding and subtracting $f_{\SC_k}^*$ gives
\begin{align}
&\|x^{k+1}-x^*\|^2\\
&\le\|x^k-x^*\|^2-2 \gamma_k \left[f_{\SC_k}(x^k)-f_{\SC_k}^*+f_{\SC_k}^*-f_{\SC_k}(x^*)\right] + \textcolor{blue}{\frac{\gamma_k }{c}[f_{\SC_k}(x^k)-f_{\SC_k}^* + f_{\SC_k}^*-\ell_{\SC_k}^*]}\\
&=\|x^k-x^*\|^2- \gamma_k\left(2-\frac{1}{c}\right) \left[f_{\SC_k}(x^k)-f_{\SC_k}^*\right] + 2 \gamma_k \left[f_{\SC_k}(x^*)-f_{\SC_k}^*\right] +  \textcolor{blue}{\frac {\gamma_k}{c} \underbrace{[f_{\SC_k}^*-\ell_{\SC_k}^*]}_{\ge0}}\\
&\le\|x^k-x^*\|^2- \gamma_k\left(2-\frac{1}{c}\right) \underbrace{\left[f_{\SC_k}(x^k)-f_{\SC_k}^*\right]}_{>0} + 2 \gamma_k\ f_{\SC_k}(x^*)\cancel{- 2\gamma_k \ f_{\SC_k}^*+ \textcolor{blue}{2\gamma_k f_{\SC_k}^*}}- \textcolor{blue}{2\gamma_k \ell_{\SC_k}^*}.
\end{align}
\vspace{-2mm}
Since $c> \frac{1}{2}$ it holds that $\left(2-\frac{1}{c}\right) >0$. We obtain:
\begin{align}
&\|x^{k+1}-x^*\|^2\\ &\leq\|x^k-x^*\|^2- \gamma_k\underbrace{\left(2-\frac{1}{c}\right)\left[f_{\SC_k}(x^k)-f_{\SC_k}^*\right]}_{\ge 0}+ 2 \gamma_k \  \textcolor{blue}{\underbrace{[f_{\SC_k}(x^*)-\ell_{\SC_k}^*]}_{\ge 0}}\label{eq:initial_point_cvx} \\
&\leq\|x^k-x^*\|^2- \alpha \left(2-\frac{1}{c}\right) \left[f_{\SC_k}(x^k)-f_{\SC_k}^*\right] + 2 \gamma_{\bound} \ \textcolor{blue}{[f_{\SC_k}(x^*)-\ell_{\SC_k}^*]}\\
&=\|x^k-x^*\|^2- \alpha \left(2-\frac{1}{c}\right) \left[f_{\SC_k}(x^k)-f_{\SC_k}(x^*)+f_{\SC_k}(x^*)-f_{\SC_k}^*\right]  + 2 \gamma_{\bound}\ \textcolor{blue}{[f_{\SC_k}(x^*)-\ell_{\SC_k}^*]}\\
&=\|x^k-x^*\|^2- \alpha \left(2-\frac{1}{c}\right) \left[f_{\SC_k}(x^k)-f_{\SC_k}(x^*)\right]- \alpha \left(2-\frac{1}{c}\right) \left[f_{\SC_k}(x^*)-f_{\SC_k}^*\right]\\ &\quad  + 2 \gamma_{\bound}\ \textcolor{blue}{[f_{\SC_k}(x^*)-\ell_{\SC_k}^*]}\\
&\leq\|x^k-x^*\|^2- \alpha \left(2-\frac{1}{c}\right) \left[f_{\SC_k}(x^k)-f_{\SC_k}(x^*)\right]+ 2 \gamma_{\bound}\ \textcolor{blue}{[f_{\SC_k}(x^*)-\ell_{\SC_k}^*]}
\end{align}
where in the last inequality we use that $\alpha \left(2-\frac{1}{c}\right) \left[f_{\SC_k}(x^*)-f_{\SC_k}^*\right]>0$.
\textcolor{blue}{ Note that this factor $f_{\SC_k}(x^*)-f_{\SC_k}^*$ pops up in the proof, not in the stepsize!} By rearranging:
\begin{align}
\alpha \left(2-\frac{1}{c}\right) \left[f_{\SC_k}(x^k)-f_{\SC_k}(x^*)\right]
&\leq\|x^k-x^*\|^2-\|x^{k+1}-x^*\|^2 +  \textcolor{blue}{2\gamma_b \left[f_{\SC_k}(x^*)-\ell_{\SC_k}^*\right]}
\end{align}
\textcolor{blue}{The rest of the proof is identical to~\citep{loizou2021stochastic}~(Theorem 3.4). Just, at the instead of $\sigma^2_B$ we have $\hat\sigma^2_B: =\Exp[f_{\SC_k}(x^*)-\ell_{\SC_k}^*]$}. That is, after taking the expectation on both sides (conditioning on $x_k$), we can use the tower property and sum over $k$~(from $0$ to $K-1$) on both sides of the inequality. After dividing by $K$, thanks to Jensen's inequality, we get~(for $c=1$):
\begin{align*}
\Exp \left[f(\bar{x}^K)-f(x^*)\right] 
\leq \frac{\|x^0-x^*\|^2}{\alpha \, K} + \frac{2\gamma_{b}\hat\sigma^2_B}{\alpha},
\end{align*}
where $\bar{x}^K=\frac{1}{K}\sum_{k=0}^{K-1} x^k$, $\alpha:=\min \{\frac{1}{2cL_{\max}},\gamma_{b}\}$ and $L_{\max}=\max \{L_i\}_{i=1}^n$ is the maximum smoothness constant.
\vspace{-4mm}
\paragraph{Strongly Convex setting.} We proceed in the usual way:
\begin{align}
\|x^{k+1}-x^*\|^2 &=\|x^k-x^*\|^2-2 \gamma_k \langle x^k-x^*, \nabla f_{\SC_k}(x^k) \rangle + \gamma_k^2 \| \nabla f_{\SC_k}(x^k)\|^2\\
&\leq\|x^k-x^*\|^2-2 \gamma_k \langle x^k-x^*, \nabla f_{\SC_k}(x^k) \rangle +  \textcolor{blue}{\frac{\gamma_k }{c} \underbrace{\left[f_{\SC_k}(x^k)-\ell_{\SC_k}^*\right]}_{\ge 0}}. \\
\end{align}

Using the fact that $c\geq1/2$, we get
\begin{align}
&\leq \|x^k-x^*\|^2-2 \gamma_k \langle x^k-x^*, \nabla f_{\SC_k}(x^k) \rangle +  \textcolor{blue}{2 \gamma_k \left[f_{\SC_k}(x^k)-\ell_{\SC_k}^*\right]} \\
&= \|x^k-x^*\|^2-2 \gamma_k \langle x^k-x^*, \nabla f_{\SC_k}(x^k) \rangle +  \textcolor{blue}{2 \gamma_k \left[f_{\SC_k}(x^k)-f_{\SC_k}(x^*)+f_{\SC_k}(x^*)-f_{\SC_k}^*\right]} \\
&= \|x^k-x^*\|^2+2 \gamma_k \left[ -\langle x^k-x^*, \nabla f_{\SC_k}(x^k) \rangle + f_{\SC_k}(x^k)-f_{\SC_k}(x^*) \right]+ \textcolor{blue}{2 \gamma_k \left[f_{\SC_k}(x^*)-\ell_{\SC_k}^*\right]}.
\end{align}

From convexity of functions $f_{\SC_k}$ it holds that $-\langle x^k-x^*, \nabla f_{\SC_k}(x^k) \rangle + f_{\SC_k}(x^k)-f_{\SC_k}(x^*)\leq0$, $\forall \SC_k \subseteq [n]$. Thus, 

\begin{align}
\|x^{k+1}-x^*\|^2 &\le \|x^{k}-x^*\|^2 +2 \gamma_k \underbrace{\left[ -\langle x^k-x^*, \nabla f_{\SC_k}(x^k) \rangle + f_{\SC_k}(x^k)-f_{\SC_k}(x^*) \right]}_{\leq 0}\\ &\quad+ \textcolor{blue}{2 \gamma_b \underbrace{\left[f_{\SC_k}(x^*)-\ell_{\SC_k}^*\right]}_{\geq 0}}
\end{align}
\textcolor{blue}{The rest of the proof is identical to~\citep{loizou2021stochastic}~(Theorem 3.1). Just, at the instead of $\sigma^2_B$ we have $\hat\sigma^2_B:\Exp[f_{\SC_k}(x^*)-\ell_{\SC_k}^*]$}. That is, after taking the expectation on both sides~(conditioning on $x_k$), we can use the tower property and solve the resulting geometric series in closed form: for $c\geq1/2$ we get

\begin{align*}
\Exp \|x^{k}-x^*\|^2 
\leq \left(1-\mu \alpha \right)^k  \|x^0-x^*\|^2 + \frac{2\gamma_{b} \hat\sigma^2_B }{\mu \alpha},
\end{align*}
where $\alpha:=\min \{\frac{1}{2cL_{\max}},\gamma_{b}\}$ and $L_{\max}=\max \{L_i\}_{i=1}^n$ is the maximum smoothness constant.
\end{proof}

\section[]{Lack of convergence of SGD with SPS$_{\max}$ in the non-interpolated setting}
\label{sec:app_proof2}
\vspace{-3mm}
\subsection[]{Convergence of SPS with decreasing stepsizes to $\boldsymbol{\tilde x\ne x^*}$ in the quadratic case}

We recall the variation of constants formula
\begin{restatable}[Variation of constants]{lem}{var}
\label{lemma:var_constants_proof}
Let $z\in\mathbb{R}^d$ evolve with time-varying linear dynamics $z_{k+1} = A_k z_k + \varepsilon_k$, where $A_k\in\mathbb{R}^{d\times d}$ and $\varepsilon_k\in \mathbb{R}^{d}$ for all $k$. Then, with the convention that $\prod_{j=k+1}^{k} A_j = I$, 
\begin{equation}
    z_{k} = \left(\prod_{j=0}^{k-1}A_j\right)z_0 +\sum_{i=0}^{k-1}\left(\prod_{j=i+1}^{k-1} A_j\right)\varepsilon_i.
\end{equation}
\end{restatable}
\begin{proof}
For $k=1$ we get $z_1 = A_0 z_0 + \varepsilon_0$. The induction step yields
\begin{align}
    z_{k+1} &= A_k\left(\left( \prod_{j=0}^{k-1}A_j\right)z_0 +\sum_{i=0}^{k-1}\left(\prod_{j=i+1}^{k-1} A_j\right)\varepsilon_i\right) + \varepsilon_k.\\
    &= \left( \prod_{j=0}^{k}A_j\right)z_0 +\sum_{i=0}^{k-1}A_k\left(\prod_{j=i+1}^{k-1} A_j\right)\varepsilon_i + \varepsilon_k.\\
    &= \left( \prod_{j=0}^{k}A_j\right)z_0 +\sum_{i=0}^{k-1}\left(\prod_{j=i+1}^{k} A_j\right)\varepsilon_i + \left(\prod_{j=k+1}^{k} A_j\right)\varepsilon_k.\\
    &= \left( \prod_{j=0}^{k}A_j\right)z_0 +\sum_{i=0}^{k}\left(\prod_{j=i+1}^{k} A_j\right)\varepsilon_i.
\end{align}
This completes the proof of the variations of constants formula.
\end{proof}

\begin{restatable}[ Quadratic 1d]{prop}{biasquadratic}
\label{prop:bias}
Consider $f(x):=\frac{1}{n}\sum_{i=1}^n f_i(x),\quad f_i(x) = \frac{a_i}{2}(x-x_i^*)^2$. We consider SGD with $\gamma_k=\frac{f_i(x^k)-f_i^*}{c_k\|\nabla f_i(x^k)\|^2}$, with $c_k = (k+1)/2$. Then $\Exp|x^{k+1}-\tilde x|^2 = \mathcal{O}(1/k)$, with $\tilde x =\frac{1}{n}\sum_{i=1}^n x_i^*\ne \frac{\sum_{i=1}^n a_ix_i^*}{\sum_{i=1}^n a_i} = x^*.$
\end{restatable}
\begin{proof}To show that $x^k\to \tilde x$, first notice that the curvature gets canceled out in the update, due to correlation between $\gamma_k$ and $\nabla f_{i_k}(x^k)$.
\begin{equation}
x^{k+1} = x^k -\gamma_k \nabla f_{i_k}(x^k) = x^k - \frac{a_i (x^k-x_{i_k}^*)}{2c_ka_i} = x^k - \frac{x^k-x_{i_k}^*}{2c_k}
\end{equation}
Now let's add and subtract $\tilde x$ twice as follows:
\begin{equation}
x^{k+1}-\tilde x = x^k-\tilde x  - \frac{x^k-\tilde x +\tilde x -x_{i_k}^*}{2c_k} = \left(1-\frac{1}{2c_k}\right)(x^k-\tilde x)  + \frac{x_{i_k}^* - \tilde x}{2c_k}.
\end{equation}

From this equality it is already clear that in expectation the update is in the direction of $\tilde x$. To provide a formal proof of convergence the first step is to use the variations of constants formula~( Lemma~\ref{lemma:var_constants_proof}).
Therefore,
\begin{equation}
x^{k+1}-\tilde x =\left[\prod_{j=0}^{k}\left(1-\frac{1}{2c_j}\right)\right](x^0-\tilde x)  + \sum_{\ell=0}^{k}\left[ \prod_{j=\ell+1}^{k}\left(1-\frac{1}{2c_j}\right)\right] \frac{x_{i_\ell}^* - \tilde x}{2c_\ell}.
\end{equation}
If $c_j = (j+1)/2$ then $\left(1-\frac{1}{2c_j}\right) = \frac{j}{j+1}$ and therefore
\begin{equation}
    \prod_{j=\ell+1}^{k}\left(1-\frac{1}{2c_j}\right) = \frac{\ell+1}{\ell+2}\cdot\frac{\ell+2}{\ell+3}\cdots\frac{k-1}{k}\cdot\frac{k}{k+1} = \frac{\ell+1}{k+1}.
\end{equation}
Hence,
\begin{equation}
 \sum_{\ell=0}^{k}\left[ \prod_{j=\ell+1}^{k}\left(1-\frac{1}{2c_j}\right)\right] \frac{x_{i_\ell}^* - \tilde x}{2c_\ell} = \sum_{\ell=0}^{k}\frac{\ell+1}{k}\cdot \frac{x_{i_\ell}^* - \tilde x}{\ell+1} =\frac{1}{k} \sum_{\ell=0}^{k}(x_{i_\ell}^* - \tilde x).
\end{equation}
Moreover, since $\prod_{j=\ell+1}^{k}\left(1-\frac{1}{2c_j}\right)=0$ we have that $x^k\to\tilde x$ in distribution, by the law of large numbers.

Finally, to get a rate on the distance-shrinking, we take the expectation w.r.t. $i_k$ conditioned on $x_k$: the cross-term disappears and we get
\begin{align}
\Exp_{i_k}|x^{k+1}-\tilde x|^2 &= \left(1-\frac{1}{2c_k}\right)^2 |x^k-\tilde x|^2  + \frac{\Exp|x_{i_k}^* - \tilde x|^2}{4c_k^2}\\
&= \left(1-\frac{1}{2c_k}\right)^2 |x^k-\tilde x|^2  + \frac{\Exp|x_{i_k}^* - \tilde x|^2}{4c_k^2}
\end{align}
Plugging in $c_k = (k+1)/2$, we get
\begin{equation}
    \Exp_{i_k}|x^{k+1}-\tilde x|^2 = \left(\frac{k}{k+1}\right)^2 |x^k-\tilde x|^2  + \frac{\Exp|x_{i_k}^* - \tilde x|^2}{(k+1)^2}.
\end{equation}
Therefore, using the tower property and the variation of constants formula,
\begin{align}
\Exp|x^{k+1}-\tilde x|^2 &=\left[\prod_{j=0}^{k}\left(\frac{j}{j+1}\right)^2\right]|x^{0}-\tilde x|^2  + \sum_{\ell=0}^{k}\left[ \prod_{j=\ell+1}^{k}\left(\frac{j}{j+1}\right)^2\right] \frac{\Exp|x_{i_\ell}^* - \tilde x|^2}{(\ell+1)^2}\\
&= \sum_{\ell=0}^{k}\frac{(\ell+1)^2}{(k+1)^2} \frac{\Exp|x_{i_\ell}^* - \tilde x|^2}{(\ell+1)^2} = \frac{\Exp|x_{i_\ell}^* - \tilde x|^2}{k+1}.
\end{align}
This concludes the proof.
\end{proof}

\subsection{Asymptotic vanishing of the SPS bias in 1d quadratics}
\label{app:bias_corr}
As the number of datapoints grows, the bias in the SPS solution~(Prop.~\ref{prop:bias}) is alleviated by an averaging effect. Indeed, if we let each pair $(a_i, x_i^*)$ to be sampled i.i.d, for every $n\in\mathbb{N}$ we have
\begin{multline}
    \Exp_{a_i,x_i^*} \left[\frac{\sum_{i=1}^n a_ix_i^*}{\sum_{i=1}^n a_i}\right] = \Exp_{a_i | x_i^*}\Exp_{x_i^*}\left[\frac{\sum_{i=1}^n a_ix_i^*}{\sum_{i=1}^n a_i}\right] \\= \Exp_{a_i | x_i^*}\left[\frac{\sum_{i=1}^n a_i \Exp_{x_i^*}[x_i^*]}{\sum_{i=1}^n a_i}\right] = \Exp_{a_i | x_i^*}\left[\frac{\sum_{i=1}^n a_i}{\sum_{i=1}^n a_i} x^*\right]  = x^*.
\end{multline}
As $n\to\infty$, it is possible to see that, under some additional assumptions~(e.g. $a_i$ Beta-distributed), the variance of $\frac{\sum_{i=1}^n a_ix_i^*}{\sum_{i=1}^n a_i}$ collapses to zero, hence one has $\frac{\sum_{i=1}^n a_ix_i^*}{\sum_{i=1}^n a_i}\to x^*$ in probability, as $n\to\infty$, with rate $\mathcal{O}(1/n)$. 

First, recall the law of total variance: $\text{Var}[Z] = \text{Var}[\Exp[Z|W]] + \Exp[\text{Var}[Z|W]]$. In our case, setting $Z:=\frac{\sum_{i=1}^n a_ix_i^*}{\sum_{i=1}^n a_i}$ and $W = a_i$, the first term is zero since $x^*$ is independent of $(a_i)_{i=1}^n$. Hence

\begin{align}
    \text{Var}\left[\frac{\sum_{i=1}^n a_ix_i^*}{\sum_{i=1}^n a_i}\right] = \Exp_{a_i | x_i*}\text{Var}_{x_i*}\left[\frac{\sum_{i=1}^n a_ix_i^*}{\sum_{i=1}^n a_i}\right] \\= \Exp\left[\frac{\sum_{i=1}^n a_i^2\text{Var}[x_i^*]}{(\sum_{i=1}^n a_i)^2}\right] = \Exp\left[\frac{\sum_{i=1}^n a_i^2}{(\sum_{i=1}^n a_i)^2}\right]\text{Var}[x_i^*].
\end{align}

Evaluating $\Exp\left[\frac{\sum_{i=1}^n a_i^2}{(\sum_{i=1}^n a_i)^2}\right]$ might be complex, yet if e.g. one assumes e.g. $a_i\sim\Gamma(k, \lambda)$~(positive support, to ensure convexity), then it is possible to show\footnote{This derivation was posted on the Mathematics StackExchange at \url{https://math.stackexchange.com/questions/138290/finding-e-left-frac-sum-i-1n-x-i2-sum-i-1n-x-i2-right-of-a-sam?rq=1} and we report it here for completeness.} that $\Exp\left[\frac{\sum_{i=1}^n a_i^2}{(\sum_{i=1}^n a_i)^2}\right] = \mathcal{O}(1/n)$. First, recall that, for $q\ge0$,
\begin{equation}
    \frac{1}{q^2} = \int_0^\infty t \mathrm{e}^{-q t} dt.
\end{equation}
We rewrite the expectation as follows:
\begin{align}
   \mathbb{E}\left[\frac{\sum_{i=1}^n a_i^2}{\left(\sum_{i=1}^n a_i \right)^2} \right] &= \int_0^\infty t \cdot \mathbb{E}\left[\sum_{i=1}^n a_i^2 \cdot \exp\left(-t \cdot \sum_{i=1}^n a_i\right) \right] \mathrm{d} t \\
  &= n \int_0^\infty t \cdot \mathbb{E}\left[a_1^2 \cdot \exp\left(-t \cdot \sum_{i=1}^n a_i\right) \right] \mathrm{d} t \\
  &= n \int_0^\infty t \cdot \mathbb{E}\left[a_1^2 \cdot \exp\left(-t a_1\right) \right]\cdot  \mathbb{E}\left[ \exp\left(-t \cdot \sum_{i=2}^n a_i\right) \right] \mathrm{d} t \\
   &= n \int_0^\infty t\cdot \mathcal{M}^{\prime\prime}_X(-t) \cdot \left(\mathcal{M}_X(-t)\right)^{n-1} \mathrm{d} t,
\end{align}
where $\mathcal{M}_X(t)$ denotes the moment generating function of the $\Gamma(k,\lambda)$ distribution. Next, we solve the integral using the closed-form expression $\mathcal{M}_X(t) = \left(1-\frac{t}{\lambda}\right)^{-k}$ for $t\le \lambda$ (else does not exist). Note that we integrate only for $t\le0$ so the MGF is always defined:
\begin{align}
   \mathbb{E}\left[\frac{\sum_{i=1}^n a_i^2}{\left(\sum_{i=1}^n a_i \right)^2} \right] &= n \int_0^\infty t \cdot \frac{k(k+1)}{\lambda^2} \left( 1+ \frac{t}{\lambda}\right)^{-2-k} \cdot \left(1+\frac{t}{\lambda}\right)^{-k (n-1)}  \mathrm{d} t  \\
   &=
    n k (k+1) \int_0^\infty \frac{u}{(1+u)^{k n +2}} \mathrm{d} u \\
    &= n k (k+1) \int_0^1 (1-s) s^{k n-1} \mathrm{d} s\\
    &= n k (k+1) \left( \frac{1}{n k} - \frac{1}{n k+1} \right)\\
    &=\frac{k+1}{k \cdot n +1}.
\end{align}
where in the third-last inequality we changed variables $t\to\lambda u$ and in the second last we changed variables $t\to\frac{1-s}{s}$.

All in all, we have that $\text{Var}\left[\frac{\sum_{i=1}^n a_ix_i^*}{\sum_{i=1}^n a_i}\right]= \mathcal{O}(1/n)$. This implies that $\frac{\sum_{i=1}^n a_ix_i^*}{\sum_{i=1}^n a_i}$ converges to $x^*$ in quadratic mean --- hence also in probability.

\section{Convergence of SGD with DecSPS in the smooth setting}
\label{sec:app_proof3}

Here we study Decreasing SPS~(DecSPS), which combines stepsize decrease with the adaptiveness of SPS.
\begin{equation}
\tag{DecSPS}
 \gamma_k:=\frac{1}{c_k}\min\left\{\frac{f_{\SC_k}(x^k)-\ell^*_{\SC_k}}{\|\nabla f_{\SC_k}(x^k)\|^2}, \ c_{k-1}\gamma_{k-1}\right\},
\end{equation}
for $k\ge0$, where we set $c_{-1}=c_0$ and $\gamma_{-1} = \gamma_b$~(stepsize bound, similar to~\citep{loizou2021stochastic}), to get
\begin{equation}
    \gamma_0:=\frac{1}{c_0}\cdot\min\left\{\frac{f_{\SC_k}(x^k)-\ell^*_{\SC_k}}{\|\nabla f_{\SC_k}(x^k)\|^2},\quad c_0\gamma_{b}\right\}.
\end{equation}

\subsection{Proof of Lemma~\ref{lemma:easy_bounds}}

\sandwichbounds*

\begin{proof}
First, note that  $\gamma_k$ is trivially \textit{non-increasing} since $\gamma_k\le c_{k-1}\gamma_{k-1}/c_k$. Next, we prove the bounds on $\gamma_k$.\\
For $k=0$, we can directly use Lemma~\ref{lemma:basic_L_sc}:
\begin{equation}
    \gamma_b\ge\gamma_0=\frac{1}{c_0}\cdot\min\left\{\frac{f_{\SC_k}(x^k)-\ell^*_{\SC_k}}{\|\nabla f_{\SC_k}(x^k)\|^2},\quad c_0\gamma_{b}\right\}\ge \min\left\{\frac{1}{2c_0 L_{\max}},\gamma_b\right\}.
\end{equation}
Next, we proceed by induction: we assume the proposition holds true for $\gamma_k$:
\begin{equation}
    \min\left\{\frac{1}{2c_{k} L_{\max}},\frac{c_0\gamma_b}{c_{k}}\right\}\le\gamma_{k}\le \frac{c_0\gamma_b}{c_{k}}.
\end{equation}
Then, we have : $\gamma_{k+1}=\frac{1}{c_{k+1}}\min\left\{\frac{f_{\SC_{k+1}}(x^{k+1})-f^*_{\SC_{k+1}}}{\|\nabla f_{\SC_{k+1}}(x^{k+1})\|^2},\iota\right\}$, where 
\begin{equation}
    \iota := c_{k}\gamma_{k}\in\left[\min\left\{ \frac{1}{2L_{\max}},c_{0}\gamma_b\right\}, c_{0}\gamma_b\right]
\end{equation}
by the induction hypothesis. This bound directly implies that the proposition holds true for $\gamma_{k+1}$, since again by Lemma~\ref{lemma:basic_L_sc} we have $\frac{f_{\SC_{k+1}}(x^{k+1})-f^*_{\SC_{k+1}}}{\|\nabla f_{\SC_{k+1}}(x^{k+1})\|^2}\ge \frac{1}{2L_{\max}}$. This concludes the induction step.
\end{proof}

\subsection{Proof of Thm.~\ref{thm:SPS_bounded_domain_cvx}}

\begin{remark}[Why was this challenging?]
The fundamental problem towards a proof for DecSPS is that the error to control due to gradient stochasticity does not come from the term $\gamma_k^2\|\nabla f(x^k)\|^2$ in the expansion of $\|x^k-x^*\|^2$, as instead is usual for SGD with decreasing stepsizes. Instead, the error comes from the inner product term $\gamma_k\langle \nabla f(x^k),x^k-x^*\rangle$. Hence, the error is proportional to $\gamma_k$, and not $\gamma^2$. As a result, the usual Robbins-Monro conditions~\citep{robbins1951stochastic} do not yield convergence. A similar problem is discussed for AdaGrad in~\citep{ward2019adagrad}.
\end{remark}

{\color{orange}In the first version of this paper, the proof contained a small error}\footnote{This was in Eq.~\eqref{eq:error_v1}, where we incorrectly bounded $\left(2 -\frac{1}{c_k} \right)$ with the constant $1$. This, of course, cannot be done since the term $f_{\SC_k}(x^k)- f_{\SC_k}(x^*)$ is not positive. However, it's expectation is positive. Since $\left(2 -\frac{1}{c_k} \right)$ is a constant, we avoid bounding it in the early stages and instead bound it after taking the expectation at the end of the proof. The result is unchanged.}{\color{orange}. This is now fixed, the result is unchanged.}

\decSPScvx*
\begin{proof}
Note that from the definition $\gamma_k:=\frac{1}{c_k}\cdot\min\left\{\frac{f_{\SC_k}(x^k)-\ell^*_{\SC_k}}{\|\nabla f_{\SC_k}(x^k)\|^2},\quad c_{k-1}\gamma_{k-1}\right\}$, we have that:
\begin{equation}
    \gamma_k\le\frac{1}{c_k}\cdot\frac{f_{\SC_k}(x^k)-\ell^*_{\SC_k}}{\|\nabla f_{\SC_k}(x^k)\|^2}.
\end{equation}
Multiplying by $\gamma_k$ and rearranging terms we get the fundamental inequality
\begin{equation}
\label{nadsjkda}
    \gamma^2_k \|\nabla f_{\SC_k}(x^k)\|^2\le\frac{\gamma_k}{c_k}[f_{\SC_k}(x^k)-\ell^*_{\SC_k}],
\end{equation}
Using the definition of DecSPS and convexity we get
\begin{align}
&\|x^{k+1}-x^*\|^2 \\
&=\|x^k-\gamma_k \nabla f_{\SC_k}(x^k)-x^*\|^2\\
&\overset{\eqref{nadsjkda}}{\leq} \|x^{k}-x^*\|^2 - 2 \gamma_k \langle\nabla f_{\SC_k}(x^k),  x^k-x^* \rangle +  \frac{\gamma_k }{c_k}(f_{\SC_k}(x^k)-\ell^*_{\SC_k})
\end{align}
Next, using convexity,
\begin{align}
&\|x^{k+1}-x^*\|^2\\&\leq \|x^{k}-x^*\|^2-2 \gamma_k [f_{\SC_k}(x^k)- f_{\SC_k}(x^*)]+  \frac{\gamma_k }{c_k}[f_{\SC_k}(x^k)-f_{\SC_k}(x^*)+ f_{\SC_k}(x^*)-\ell^*_{\SC_k}]\\
&= \|x^{k}-x^*\|^2-2 \gamma_k [f_{\SC_k}(x^k)- f_{\SC_k}(x^*)]+  \frac{\gamma_k }{c_k}[f_{\SC_k}(x^k)-f_{\SC_k}(x^*)]+ \frac{\gamma_k }{c_k} [f_{\SC_k}(x^*)-\ell^*_{\SC_k}]\\
&= \|x^{k}-x^*\|^2-\left(2 -\frac{1}{c_k} \right)\gamma_k [f_{\SC_k}(x^k)- f_{\SC_k}(x^*)]+  \frac{\gamma_k }{c_k}[f_{\SC_k}(x^*)-\ell^*_{\SC_k}].
\end{align}
Let us divide everything by $\gamma_k>0$.

\begin{equation}
\frac{\|x^{k+1}-x^*\|^2}{\gamma_k} \leq \frac{\|x^{k}-x^*\|^2}{\gamma_k}-\left(2 -\frac{1}{c_k} \right) [f_{\SC_k}(x^k)- f_{\SC_k}(x^*)]+  \frac{1 }{c_k}[f_{\SC_k}(x^*)-\ell^*_{\SC_k}].
\end{equation}
Rearranging,
\begin{equation}
\left(2 -\frac{1}{c_k} \right) [f_{\SC_k}(x^k)- f_{\SC_k}(x^*)] \leq \frac{\|x^{k}-x^*\|^2}{\gamma_k}-\frac{\|x^{k+1}-x^*\|^2}{\gamma_k}+  \frac{1 }{c_k}[f_{\SC_k}(x^*)-\ell^*_{\SC_k}].
\label{eq:error_v1}
\end{equation}
Next, summing from $k=0$ to $K-1$:
\begin{equation}
 \sum_{k=0}^{K-1} \left(2 -\frac{1}{c_k} \right) [f_{\SC_k}(x^k)- f_{\SC_k}(x^*)] \leq \sum_{k=0}^{K-1} \frac{\|x^{k}-x^*\|^2}{\gamma_k}-\sum_{k=0}^{K-1} \frac{\|x^{k+1}-x^*\|^2}{\gamma_k}+  \sum_{k=0}^{K-1} \frac{1 }{c_k}[f_{\SC_k}(x^*)-\ell^*_{\SC_k}].
\end{equation}
And therefore
\begin{align}
 &\sum_{k=0}^{K-1} \left(2 -\frac{1}{c_k} \right)[f_{\SC_k}(x^k)- f_{\SC_k}(x^*)]\\
 &\leq \sum_{k=0}^{K-1} \frac{\|x^{k}-x^*\|^2}{\gamma_k}-\sum_{k=0}^{K-1} \frac{\|x^{k+1}-x^*\|^2}{\gamma_k}+  \sum_{k=0}^{K-1} \frac{1 }{c_k}[f_{\SC_k}(x^*)-\ell^*_{\SC_k}]\\
 &\leq \frac{\|x^{0}-x^*\|^2}{\gamma_0}+\sum_{k=1}^{K-1} \frac{\|x^{k}-x^*\|^2}{\gamma_k}-\sum_{k=0}^{K-2} \frac{\|x^{k+1}-x^*\|^2}{\gamma_k}-\frac{\|x^{K}-x^*\|^2}{\gamma_{K-2}} +  \sum_{k=0}^{K-1} \frac{1 }{c_k}[f_{\SC_k}(x^*)-\ell^*_{\SC_k}]\\
 &\leq \frac{\|x^{0}-x^*\|^2}{\gamma_0}+\sum_{k=0}^{K-2} \frac{\|x^{k+1}-x^*\|^2}{\gamma_{k+1}}-\sum_{k=0}^{K-2} \frac{\|x^{k+1}-x^*\|^2}{\gamma_k} +  \sum_{k=0}^{K-1} \frac{1 }{c_k}[f_{\SC_k}(x^*)-\ell^*_{\SC_k}]\\
 &\leq \frac{\|x^{0}-x^*\|^2}{\gamma_0}+\sum_{k=0}^{K-2} \left(\frac{1}{\gamma_{k+1}}-\frac{1}{\gamma_{k}}\right)\|x^{k+1}-x^*\|^2 +  \sum_{k=0}^{K-1} \frac{1 }{c_k}[f_{\SC_k}(x^*)-\ell^*_{\SC_k}]\\
 &\leq D^2\left[\frac{1}{\gamma_0}+\sum_{k=0}^{K-2} \left(\frac{1}{\gamma_{k+1}}-\frac{1}{\gamma_{k}}\right)\right]+  \sum_{k=0}^{K-1} \frac{1 }{c_k}[f_{\SC_k}(x^*)-\ell^*_{\SC_k}]\label{noajsnxao2}\\
  &\leq \frac{D^2}{\gamma_{K-1}} +  \sum_{k=0}^{K-1} \frac{1 }{c_k}[f_{\SC_k}(x^*)-\ell^*_{\SC_k}].
  \label{noajsnxao}
\end{align}

\begin{remark}[Where did we use the modified SPS definition?]
In step \eqref{noajsnxao2}, we are able to collect $D^2$ because $\left(\frac{1}{\gamma_{k+1}}-\frac{1}{\gamma_{k}}\right)\ge0$. This is guaranteed by the new SPS definition (DecSPS), along with the fact that $c_k$ is increasing. Note that one could not perform this step under the original SPS update rule of~\citep{loizou2021stochastic}.
\end{remark}
Thanks to Lemma~\ref{lemma:easy_bounds}, we have:
\begin{equation*}
    \gamma_k\ge \min\left\{\frac{1}{2c_k L_{\max}},\frac{c_0\gamma_b}{c_k}\right\}.
\end{equation*}
Hence,
\begin{equation}
\label{casojnxa}
    \frac{1}{\gamma_k}\le  c_k \cdot \max\left \{2 L_{\max},\frac{1}{c_0\gamma_b}\right\}.
\end{equation}
Let us call $\tilde L = \max\left \{ L_{\max},\frac{1}{2c_0\gamma_b}\right\}$. By combining \eqref{casojnxa} with \eqref{noajsnxao} and dividing by $K$ we get:
\begin{equation}
 \frac{1}{K}\sum_{k=0}^{K-1}\left(2 -\frac{1}{c_k} \right) [f_{\SC_k}(x^k)- f_{\SC_k}(x^*)] \leq \frac{2c_{K-1} \tilde L D^2 }{K} +  \frac{1}{K}\sum_{k=0}^{K-1} \frac{[f_{\SC_k}(x^*)-\ell^*_{\SC_k}]}{c_k},
\end{equation}
Taking the expectation, 
\begin{equation}
 \frac{1}{K}\sum_{k=0}^{K-1}\left(2 -\frac{1}{c_k} \right) \Exp[f(x^k)- f(x^*)] \leq \frac{2c_{K-1} \tilde L D^2 }{K} +  \frac{1}{K}\sum_{k=0}^{K-1} \frac{[f_{\SC_k}(x^*)-\ell^*_{\SC_k}]}{c_k}.
\end{equation}
Since by hypothesis $c_k\ge1$ for all $k\in\mathbb{N}$, we have $\left(2 -\frac{1}{c_k} \right)\ge1$ and therefore
\begin{equation}
 \frac{1}{K}\sum_{k=0}^{K-1} \Exp[f(x^k)- f(x^*)] \leq \frac{2c_{K-1} \tilde L D^2 }{K} +  \frac{1}{K}\sum_{k=0}^{K-1} \frac{[f_{\SC_k}(x^*)-\ell^*_{\SC_k}]}{c_k}.
\end{equation}
We conclude by using Jensen's inequality as follows:
\begin{equation}
   \Exp \left[f(\bar{x}^K)-f(x^*)\right] \overset{\text{Jensen}}{\leq}   \frac{1}{K}\sum_{k=0}^{K-1} \Exp \left[f(x^k)-f(x^*)\right] \leq \frac{2c_{K-1} \tilde L D^2 }{K} +  \frac{1}{K}\sum_{k=0}^{K-1} \frac{\hat\sigma^2_{B}}{c_k}.
\end{equation}
where $\hat\sigma^2_{B}$ is as defined in \eqref{sigmahat}.
\end{proof}

\begin{remark}[Second term does not depend on $\gamma_b$]
Note that, in the convergence rate, the second term does not depend on $\gamma_b$ while the first does. This is different from the original SPS result~\citep{loizou2021stochastic}, and due to the different proof technique: specifically, we divide by $\gamma_k$ early in the proof --- and not at the end. To point to the exact source of this difference, we invite the reader to inspect Equation 24 in the appendix\footnote{\url{http://proceedings.mlr.press/v130/loizou21a/loizou21a-supp.pdf}} of~\citep{loizou2021stochastic}: the last term there is proportional to $\gamma_b/\alpha$, where $\alpha$ is a lower bound on the SPS and $\gamma_b$ is an upper bound. In our proof approach, these terms --- which bound the same quantity --- effectively cancel out~(because we divide by $\gamma_k$ earlier in the proof), at the price of having $D^2$ in the first term.
\end{remark}

\subsection{Proof of Prop.~\ref{prop:strong_convex_bound_2}}

We need the following lemma. An illustration of the result can be found in Fig.~\ref{fig:verification_lemma_A3}.

\begin{restatable}[]{lem}{probonenobound}
\label{lemma:probonenobound}
Let $z^{k+1} = A_k z^k + \varepsilon_k$ with $A_k = (1-a/\sqrt{k+1})$ and $\varepsilon_k =b/\sqrt{k+1}$. If $z^0>0$, $0<a\le 1$,  $b>0$, then $z^k \le\max\{z^0, b/a\}$ for all $k\geq 0$.
\end{restatable}

\begin{proof}
Simple to prove by induction. The base case is trivial, since $z^0 \le\max\{z^0, b/a\}$. Let us now assume the proposition holds true for $z^k$ (that is, $z^k \le\max\{z^0, b/a\}$) , we want to show it holds true for $k+1$. We have
\begin{equation}
    z^{k+1} = \left(1-\frac{a}{\sqrt{k+1}}\right) z^k + \frac{b}{\sqrt{k+1}}.
\end{equation}
If $b/a = \max\{z^0, b/a\}$, then we get, by induction
\begin{equation}
    z^{k+1} \le \left(1-\frac{a}{\sqrt{k+1}}\right) \frac{b}{a} + \frac{b}{\sqrt{k+1}} =\frac{b}{a}=\max\{z^0, b/a\}.
\end{equation}
Else, if  $z^0=\max\{z^0, b/a\}$, then by induction
\begin{equation}
    z^{k+1} \le \left(1-\frac{a}{\sqrt{k+1}}\right) z^0 + \frac{b}{\sqrt{k+1}} = z^0 -\frac{a z^0 -b}{\sqrt{k+1}}\le z^0=\max\{z^0, b/a\},
\end{equation}
where the last inequality holds because $a z^0 -b>0$ and $a$ is positive. This completes the proof. 
\end{proof}

\begin{figure}
    \centering
    \includegraphics[width=0.35\textwidth]{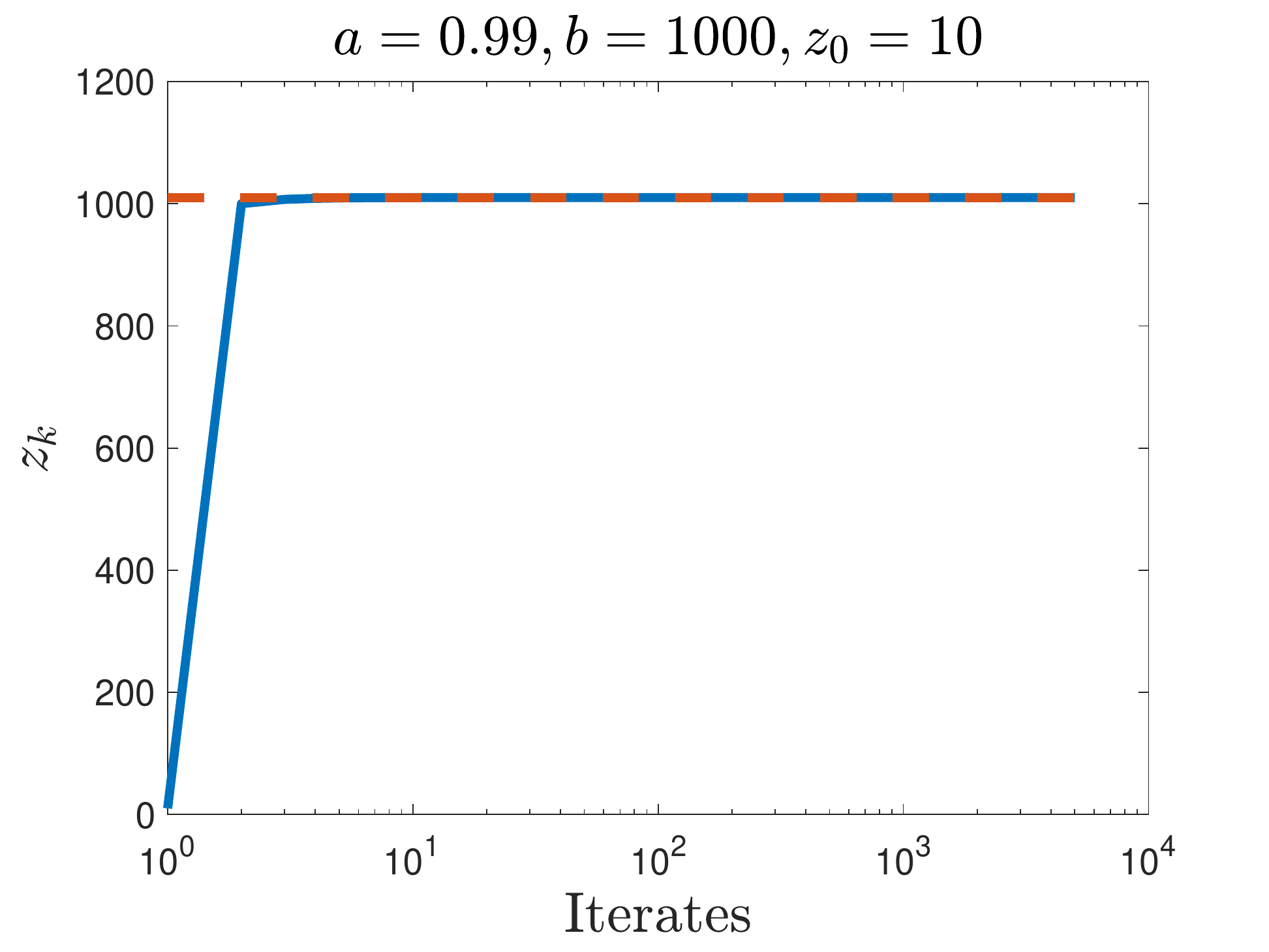}
    \includegraphics[width=0.35\textwidth]{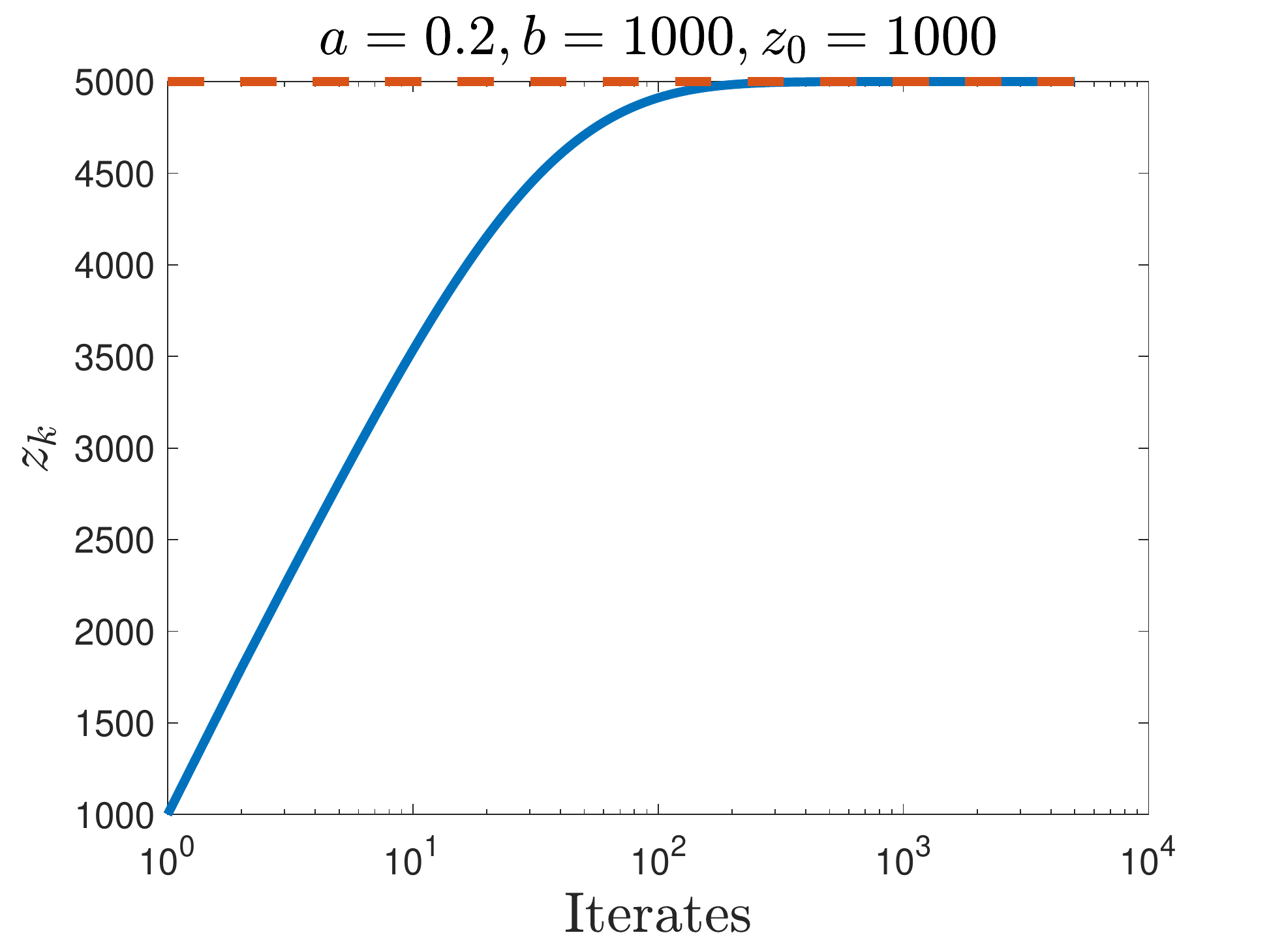}\\
    \includegraphics[width=0.35\textwidth]{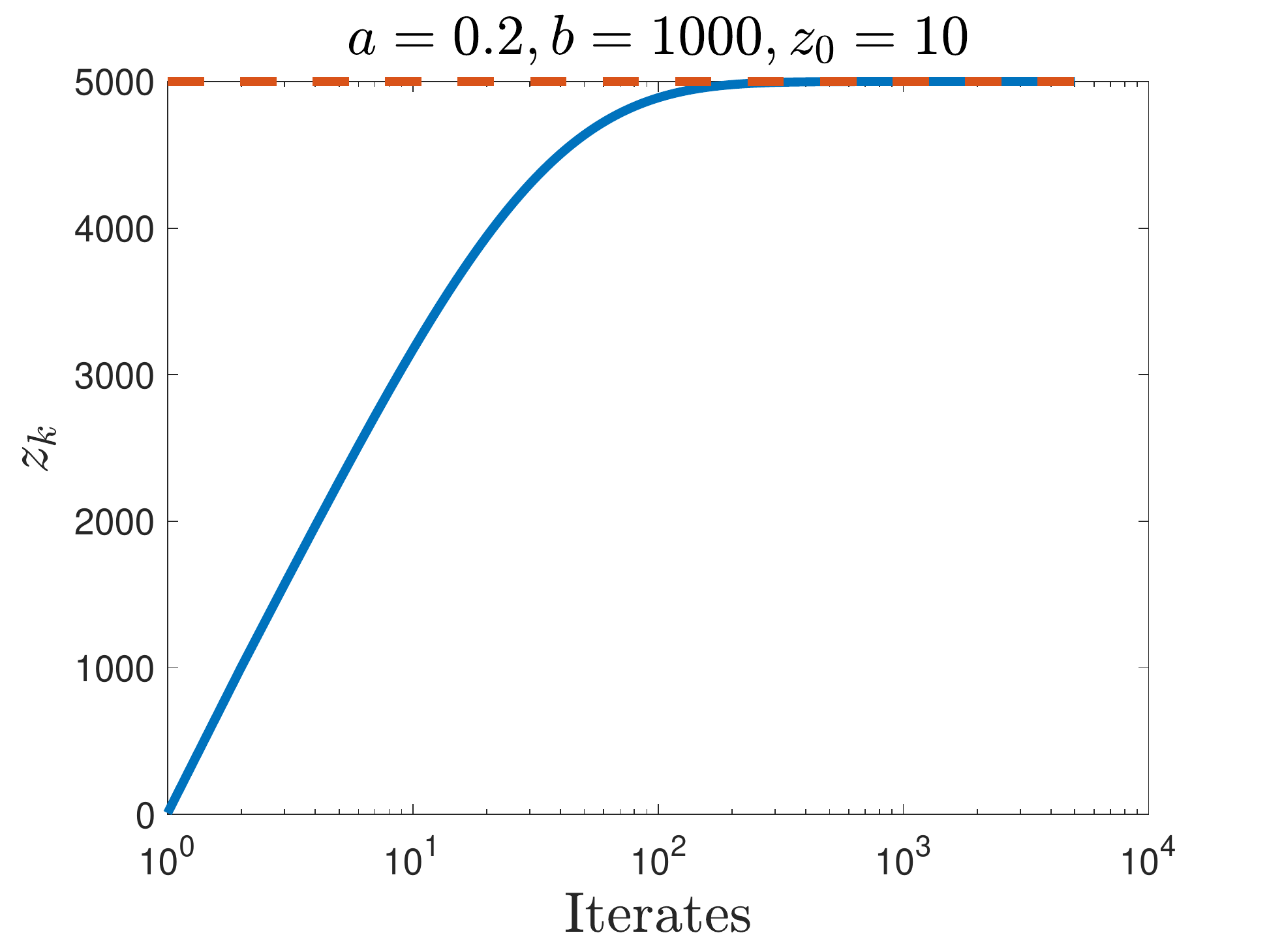}
    \includegraphics[width=0.35\textwidth]{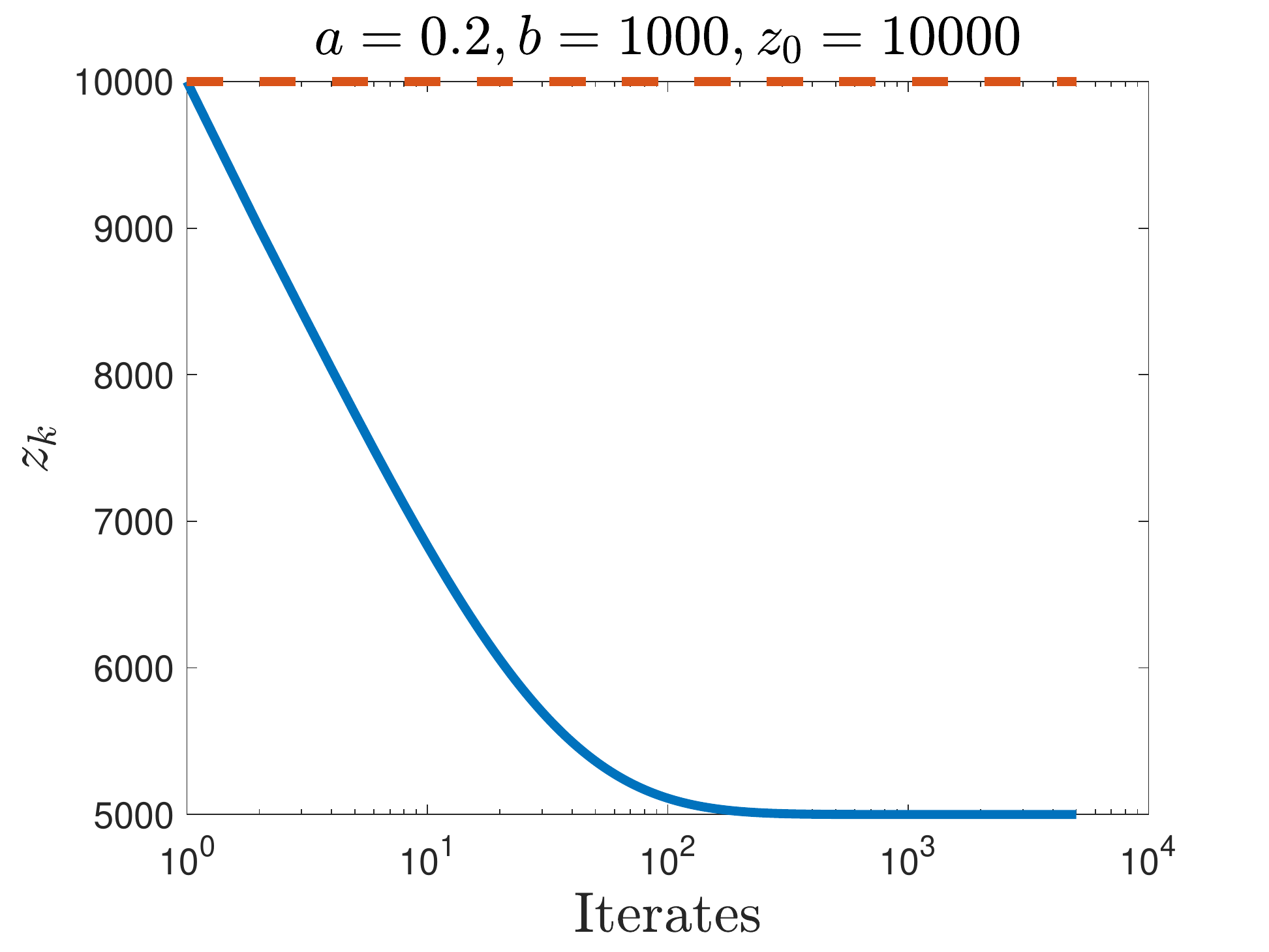}    
    \caption{Numerical Verification of Lemma~\ref{lemma:probonenobound}. Bound in the lemma is indicated with dashed line.}
    \label{fig:verification_lemma_A3}
\end{figure}

\PropSCUnbound*
\begin{proof}
Using the SPS definition we directly get
\begin{align}
\|x^{k+1}-x^*\|^2 &=\|x^k-\gamma_k \nabla f_{\SC_k}(x^k)-x^*\|^2\\
&=\|x^k-x^*\|^2-2 \gamma_k \langle\nabla f_{\SC_k}(x^k), x^k-x^* \rangle + \gamma_k^2 \| \nabla f_{\SC_k}(x^k)\|^2\\
&\leq\|x^k-x^*\|^2-2 \gamma_k \langle\nabla f_{\SC_k}(x^k),  x^k-x^* \rangle +  \frac{\gamma_k }{c_k}(f_{\SC_k}(x^k)-\ell^*_{\SC_k}),
\label{eq:SC-naive-1}
\end{align}
where (as always) we used the fact that since from the definition $\gamma_k:=\frac{1}{c_k}\cdot\min\left\{\frac{f_{\SC_k}(x^k)-\ell^*_{\SC_k}}{\|\nabla f_{\SC_k}(x^k)\|^2},\quad c_{k-1}\gamma_{k-1}\right\}$, then $\gamma_k\le\frac{1}{c_k}\cdot\frac{f_{\SC_k}(x^k)-\ell^*_{\SC_k}}{\|\nabla f_{\SC_k}(x^k)\|^2}$ and we have 
\begin{equation}
    \gamma^2_k \|\nabla f_{\SC_k}(x^k)\|^2\le\frac{1}{c_k}[f_{\SC_k}(x^k)-\ell^*_{\SC_k}].
\end{equation}
Now recall that, if each $f_{i}$ is $\mu_{i}$-strongly convex then for any $x,y\in\R^d$ we have

\begin{equation}
    -\langle \nabla f_{\SC_k}(x),x-y\rangle \le-\frac{\mu_{\min}}{2} \|x-y\|^2 - f_{\SC_k}(x) + f_{\SC_k}(y).
\end{equation}
For $y = x^*$ and $x = x^k$, this implies
\begin{equation}
    -\langle \nabla f_{\SC_k}(x^k),x^k-x^*\rangle \le-\frac{\mu_{\min}}{2} \|x^k-x^*\|^2 - f_{\SC_k}(x^k) + f_{\SC_k}(x^*).
\end{equation}

Adding and subtracting $\ell_{\SC_k}^*$ to the RHS of the inequality above, we get

\begin{equation}
    -\langle \nabla f_{\SC_k}(x^k),x^k-x^*\rangle \le-\frac{\mu_{\min}}{2} \|x^k-x^*\|^2 - (f_{\SC_k}(x^k)-\ell_{\SC_k}^*) + (f_{\SC_k}(x^*)-\ell_{\SC_k}^*).
\end{equation}

Since $\gamma_k>0$, we can substitute this inequality in Equation~\eqref{eq:SC-naive-1} and get
\begin{align}
\|x^{k+1}-x^*\|^2 \leq & \ \|x^k-x^*\|^2+  \frac{\gamma_k }{c_k}(f_{\SC_k}(x^k)-\ell^*_{\SC_k})\\ &  \underbrace{-\mu_{\min}\gamma_k \|x^k-x^*\|^2 - 2\gamma_k(f_{\SC_k}(x^k)-\ell_{\SC_k}^*) + 2\gamma_k(f_{\SC_k}(x^*)-\ell_{\SC_k}^*) }_{\le - 2 \gamma_k \langle\nabla f_{\SC_k}(x^k),  x^k-x^* \rangle}.
\end{align}
Rearranging a few terms we get
\begin{align}
    &\|x^{k+1}-x^*\|^2\\ &\leq (1-\mu_{\min}\gamma_k)\|x^k-x^*\|^2+  \frac{\gamma_k }{c_k}(f_{\SC_k}(x^k)-\ell^*_{\SC_k}) - 2\gamma_k(f_{\SC_k}(x^k)-\ell_{\SC_k}^*) + 2\gamma_k(f_{\SC_k}(x^*)-\ell_{\SC_k}^*) \\
    &\leq (1-\mu_{\min}\gamma_k)\|x^k-x^*\|^2- \left(2-\frac{1}{c_k}\right)\gamma_k(f_{\SC_k}(x^k)-\ell_{\SC_k}^*) + 2\gamma_k(f_{\SC_k}(x^*)-\ell_{\SC_k}^*).
\end{align}

Since we assumed $c_k\ge1/2$ for all $k\in\mathbb{N}$, we can drop the term $- \left(2-\frac{1}{c_k}\right)\gamma_k[f_{\SC_k}(x^k)-f_{\SC_k}^*]$, since also $f_{\SC_k}^*\ge\ell_{\SC_k}^*$. Hence, we get the following bound:
\begin{align}
    \|x^{k+1}-x^*\|^2 &\leq (1-\mu_{\min}\gamma_k) \|x^k-x^*\|^2 + 2\gamma_k(f_{\SC_k}(x^*)-\ell_{\SC_k}^*)\\
    &\leq\left(1-\mu_{\min}\gamma_k\right) \|x^k-x^*\|^2 + \frac{2c_0\gamma_b}{ c_{k}}(f_{\SC_k}(x^*)-\ell_{\SC_k}^*)\\
    &\leq\left(1-\mu_{\min}\gamma_k\right) \|x^k-x^*\|^2 + \frac{2c_0\gamma_b \hat\sigma^2_{B,\max}}{ c_{k}}
\end{align}
where we used the inequality $\min\left\{\frac{1}{2c_k L_{\max}},\frac{c_0\gamma_b}{c_k}\right\}\le\gamma_k\le \frac{c_0\gamma_b}{c_k}$~(Lemma~\ref{lemma:easy_bounds}). 

Now we seek an upper bound for the contraction factor. Under $c_k = \sqrt{k+1}$, using again Lemma~\ref{lemma:easy_bounds} we have, since $c_0=1$,
\begin{equation}
    1-\mu_{\min}\gamma_k \ge 1-\frac{\min\left\{\frac{\mu_{\min}}{2 L_{\max}},\mu_{\min}\gamma_b\right\}}{\sqrt{k+1}}.
\end{equation}
Now have all ingredients to bound the iterates: the result follows from Lemma~\ref{lemma:probonenobound} using $a = \min\left\{\frac{\mu_{\min}}{2 L_{\max}},\mu_{\min}\gamma_b\right\}$ and $b = 2c_0\gamma_b\hat\sigma^2_{B,\max}$. So, we get
\begin{equation}
    \|x^{k+1}-x^*\|^2 \leq \max\left\{\|x^{0}-x^*\|^2, \frac{2c_0\gamma_b\hat\sigma^2_{B,\max}}{\min\left\{\frac{\mu_{\min}}{2 L_{\max}},\mu_{\min}\gamma_b\right\}}\right\}, \text{for all } k\ge 0.
\end{equation}
This completes the proof.
\end{proof}

\section{Convergence of stochastic subgradient method with DecSPS-NS in the non-smooth setting}
\label{sec:app_proof4}

In this subsection we consider the DecSPS-NS stepsize in the non-smooth setting:
\begin{equation}
    \gamma_k:=\frac{1}{c_k}\cdot\min\left\{\max\left\{c_0\gamma_{\ell}, \frac{f_{\SC_k}(x^k)-\ell^*_{\SC_k}}{\|g_{\SC_k}(x^k)\|^2}\right\}, c_{k-1}\gamma_{k-1}\right\},
\end{equation}
where $g_{\SC_k}(x^k)$ is the stochastic subgradient using batch size $\SC_k$ at iteration $k$, and we set $c_{-1}=c_0$ and $\gamma_{-1} = \gamma_b$ to get

\begin{equation}
    \gamma_0:=\frac{1}{c_0}\cdot\min\left\{\max\left\{c_0\gamma_{\ell},\frac{f_{\SC_k}(x^k)-\ell^*_{\SC_k}}{\|g_{\SC_k}(x^k)\|^2}\right\},\quad c_0\gamma_{b}\right\}.
\end{equation}

\subsection{Proof stepsize bounds}
\begin{restatable}[Non-smooth bounds]{lem}{NSsandwichbounds}
\label{lemma:easy_bounds_NS}
Let $(c_k)_{k=0}^\infty$ be any non-decreasing positive sequence. Then, under DecSPS-NS, we have that for every $k\in\mathbb{N}$, $\frac{c_0\gamma_\ell}{c_k}\le\gamma_k\le \frac{c_0\gamma_b}{c_k}$, $\gamma_{k-1}\le\gamma_k$.
\end{restatable}
\begin{proof}
First, note that  $\gamma_k$ is trivially \textit{non-increasing} since $\gamma_k\le c_{k-1}\gamma_{k-1}/c_k$. Next, we prove the bounds on $\gamma_k$.

Without loss of generality, we can work with the simplified stepsize
\begin{equation}
    \gamma_k:=\frac{1}{c_k}\cdot\min\left\{\max\left\{c_0\gamma_{\ell}, \alpha_k \right\}, c_{k-1}\gamma_{k-1}\right\},
\end{equation}
where $\alpha_k\in\mathbb{R}$ is any number. We proceed by induction: at $k=0$~(base case) we get
\begin{equation}
    \gamma_k:=\frac{1}{c_0}\cdot\min\left\{\max\left\{c_0\gamma_{\ell}, \alpha \right\}, c_{0}\gamma_{b}\right\} = \min\left\{\max\left\{\gamma_{\ell}, \alpha_0/c_0 \right\}, \gamma_{b}\right\},
\end{equation}
if $\alpha_0/c_0\le \gamma_{\ell}$ then $\gamma_k = \min\left\{\gamma_\ell, \gamma_{b}\right\} = \gamma_\ell$. Otherwise $\alpha_0/c_0> \gamma_{\ell}$ and therefore $\gamma_k = \min\left\{\alpha_0/c_0, \gamma_{b}\right\}$. If in addition, if $\alpha_0/c_0\ge\gamma_{b}$ then $\gamma_0=\gamma_b$, else $\gamma_0 = \alpha_0/c_0\in[\gamma_\ell,\gamma_b]$. In all these cases, we get $\gamma_0\in[\gamma_\ell, \gamma_b]$; hence, the base case holds true.

We now proceed with the induction step by assuming $\frac{c_0\gamma_\ell}{c_k}\le\gamma_k\le \frac{c_0\gamma_\ell}{c_k}$. Using the definition of DecSPS-NS we will then show that the same inequalities hold for $\gamma_{k+1}$. We start by noting that, since $\gamma_k\in\left[\frac{c_0\gamma_\ell}{c_k},\frac{c_0\gamma_b}{c_k}\right]$, it holds that,
\begin{equation}
    \gamma_{k+1}:=\frac{1}{c_{k+1}}\cdot\min\left\{\max\left\{c_{0}\gamma_{\ell}, \alpha_{k+1} \right\}, c_{k}\gamma_{k}\right\} = \frac{1}{c_{k+1}}\cdot\min\left\{\max\left\{c_{0}\gamma_{\ell}, \alpha_{k+1} \right\}, c_0\iota\right\}, \quad \iota \in [\gamma_\ell, \gamma_b].
\end{equation}
Similarly to the base case, we can write:
\begin{equation}
     \gamma_{k+1} = \frac{c_0}{c_{k+1}}\cdot\min\left\{\max\left\{\gamma_{\ell}, \alpha_{k+1}/c_0 \right\}, \iota\right\}.
\end{equation}
With a procedure identical to the setting $k=0$ we get that $\min\left\{\max\left\{\gamma_{\ell}, \alpha_{k+1}/c_0 \right\}, \iota\right\}\in[\gamma_\ell,\iota]\subseteq[\gamma_\ell,\gamma_b]$. This concludes the proof.
\end{proof}

\subsection{Proof of Thm.~\ref{thm:SPS_bounded_domain_cvx_NS}}
\decSPScvxNS*
\begin{proof}
Let us consider the DecSPS stepsize in the non-smooth setting. Using convexity and the gradient bound we get
\begin{align}
\|x^{k+1}-x^*\|^2 
&=\|x^k-\gamma_k g_{S_k}-x^*\|^2\\
&= \|x^{k}-x^*\|^2 - 2 \gamma_k \langle g_{S_k},  x^k-x^* \rangle +  \gamma_k^2 \|g_{S_k}\|^2\\
&\leq \|x^{k}-x^*\|^2 - 2 \gamma_k [f_{\SC_k}(x^k)- f_{\SC_k}(x^*)] +  \gamma_k  \frac{c_0 \gamma_b G^2}{c_k},
\end{align}
where the last line follows from definition of subgradient and Lemma~\ref{lemma:easy_bounds_NS}.

By dividing by $\gamma_k>0$,
\begin{equation}
\frac{\|x^{k+1}-x^*\|^2}{\gamma_k} \leq \frac{\|x^{k}-x^*\|^2}{\gamma_k} - 2[f_{\SC_k}(x^k)- f_{\SC_k}(x^*)]+  \frac{1}{c_k} c_0 \gamma_b G^2.
\end{equation}
Using the same exact steps as Thm~\ref{thm:SPS_bounded_domain_cvx}, and using the fact that $\gamma_k$ is decreasing, we arrive at the equation
\begin{equation}
 \frac{1}{K}\sum_{k=0}^{K-1} f_{\SC_k}(x^k)- f_{\SC_k}(x^*)\leq \frac{D^2}{K\gamma_{K-1}} +  \sum_{k=0}^{K-1}\frac{c_0 \gamma_b G^2}{K c_k} .
\end{equation}
Now we use the fact that, since
$\frac{c_0\gamma_\ell}{c_k}\le\gamma_k$ by Lemma~\ref{lemma:easy_bounds_NS}, we have

\begin{equation}
 \frac{1}{K}\sum_{k=0}^{K-1} f_{\SC_k}(x^k)- f_{\SC_k}(x^*)\leq \frac{c_K D^2}{\gamma_\ell c_0 K } + \frac{1}{K} \sum_{k=0}^{K-1}\frac{c_0 \gamma_b G^2}{c_k}.
\end{equation}
We conclude by taking the expectation and using Jensen's inequality.
\end{proof}

\begin{restatable}[]{cor}{CorSqrtNS}
\label{cor:sqrt_NS}
In the setting of Thm.~\ref{thm:SPS_bounded_domain_cvx_NS}, if $c_k=c_0\sqrt{k+1}$ we have 
\begin{equation}
    \Exp[f(\bar{x}^K)-f(x^*)] \leq \frac{D^2/\gamma_\ell+2\gamma_b G^2}{\sqrt{K}}.
\end{equation}
\end{restatable}

\begin{remark}
The bound in Cor.~\ref{cor:sqrt_NS} does not depend on $\sigma^2_{B}$, while the one in Cor.~\ref{cor:sqrt_smooth} does. This is because the proof is different, and does not rely on bounding squared gradients with function suboptimalities~(one cannot, if smoothness does not hold). Similarly, usual bounds for non-smooth optimization do not depend on subgradient variance but instead on $G$~\citep{nemirovski2009robust,duchi2011adaptive,ene2020adaptive}.
\end{remark}
\newpage
\section{Further experimental results}
\label{app:exp}
\begin{itemize}[leftmargin=*]
  \setlength\itemsep{0.01em}
    \item \textit{Synthetic Dataset} : Following~\citep{gower2019sgd} we generate $n = 500$ datapoints from a standardized Gaussian distribution in $\R^d$, with $d=100$. We sample the corresponding labels at random. We consider a batch size $B=20$ and either $\lambda=0$ or $\lambda = 1e-4$.
    \item \textit{A1A dataset}~(standard normalization) from~\citep{chang2011libsvm}, consisting in $1605$ datapoints in 123 dimensions. We consider again $B=20$ but a substantial regularization with $\lambda = 0.01$. 
    \item \textit{Breast Cancer dataset}~(standard normalization)~\citep{Dua:2019}, consisting in $569$ datapoints in 39 dimensions. We consider a small batch size $B=5$ with strong regularization $\lambda = 0.1$.
\end{itemize}

All experiments reported below are repeated 5 times. Shown is mean and 2 standard deviations.

\paragraph{Tuning of DecSPS.} DecSPS has two hyperparameters: the upper bound $\gamma_b$ on the first stepsize and the scaling constant $c_0$. As stated in the main paper, while Thm.~\ref{thm:SPS_bounded_domain_cvx_NS} guarantees convergence for any positive value of these hyperparameters, the result of Thm.~\ref{thm:SPS_bounded_domain_cvx} suggests that using $c_0=1$ yields the best performance under the assumption that $\hat \sigma^2_B\ll \tilde L D^2$~(e.g. reasonable distance of initialization from the solution, and the maximum gradient Lipschitz constant $L_{\max}=\max_i L_i> 1/\gamma_b$). For the definition of these quantities plese refer to the main paper. In Fig.~\ref{fig:sensitivity} in the main paper we showed that (1) $c_0=1$ is optimal in this setting~(under $\gamma_b=10)$ and (2) the performance of SGD with DecSPS is almost independent of $\gamma_b$ at $c_0=1$. Similar findings hold for the A1A and Breast Cancer datasets, as shown in Figure~\ref{fig:A1A_Breast_tuning}. For A1A, we can see that the dynamics is almost independent of $\gamma_b$ at $c_0=1$ and that, at $\gamma_b=10$, $c_0=1$ indeed yields the best performance. The findings are similar for the Breast Cancer dataset; however, there we see that at $\gamma_b=10$, $c_0=5$ yields the best final suboptimality --- yet $c_0=1$ is clearly the best tradeoff between convergence speed and final accuracy.

\begin{figure}[ht]
    \centering
    \textbf{\scriptsize\quad \quad A1A - Sensitivity $\bf{\gamma_b}$\qquad\quad\quad \quad A1A - Sensitivity $\bf{c_0}$\qquad\quad\quad \quad Breast - Sensitivity $\bf{\gamma_b}$\qquad\quad\quad\quad Breast - Sensitivity $\bf{c_0}$\qquad\quad}\\
    \includegraphics[height = 0.26\textwidth]{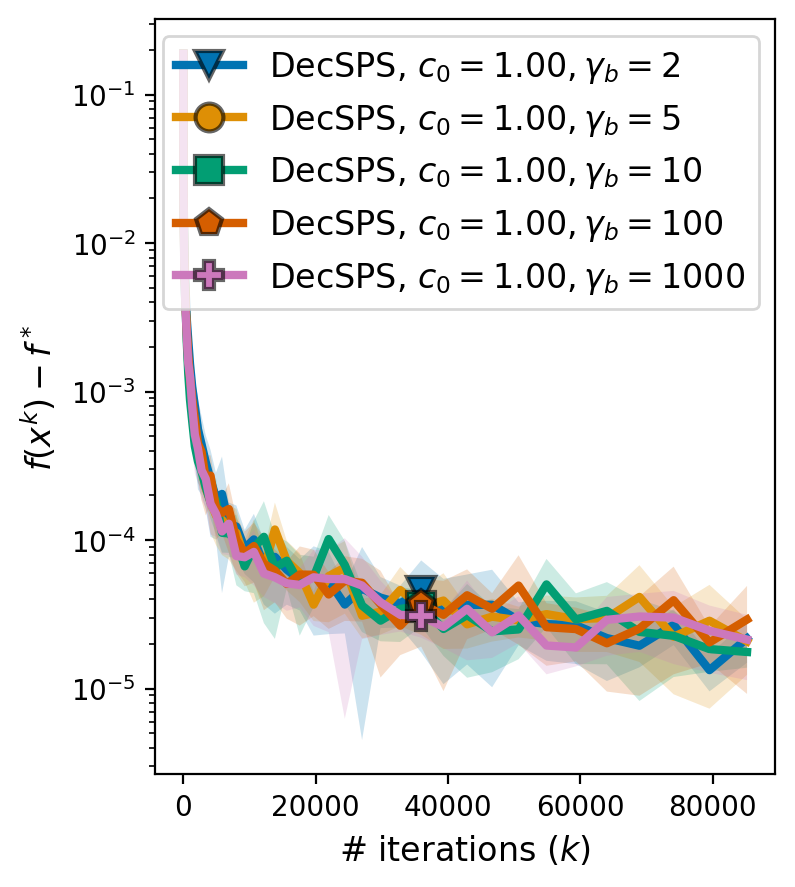}
    \includegraphics[height = 0.26\textwidth]{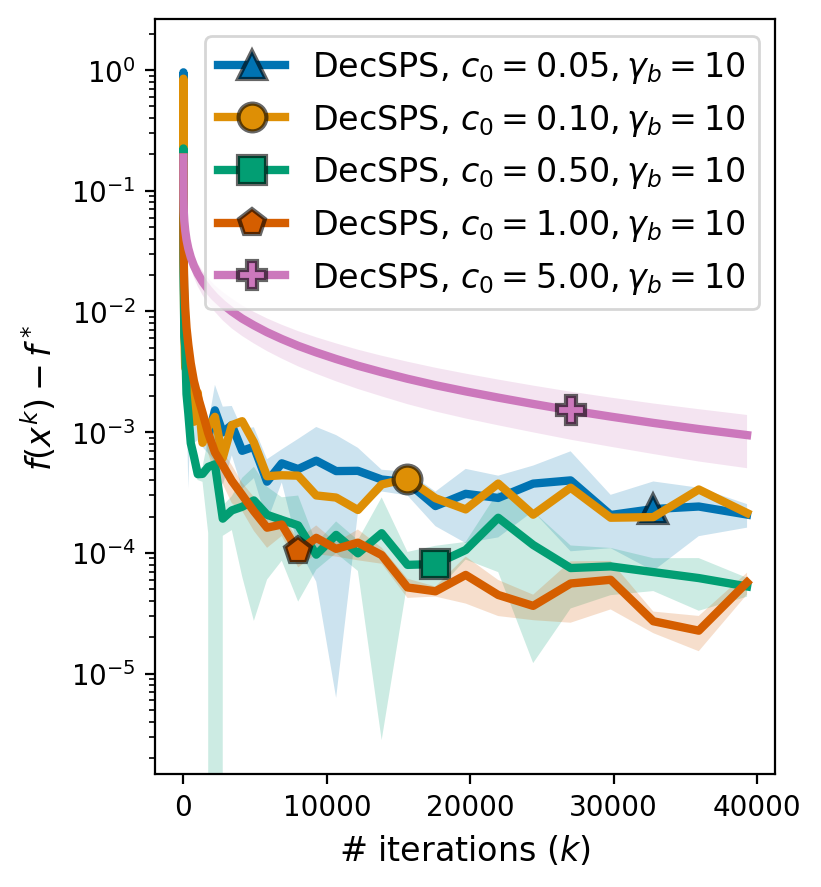}
    \includegraphics[height = 0.26\textwidth]{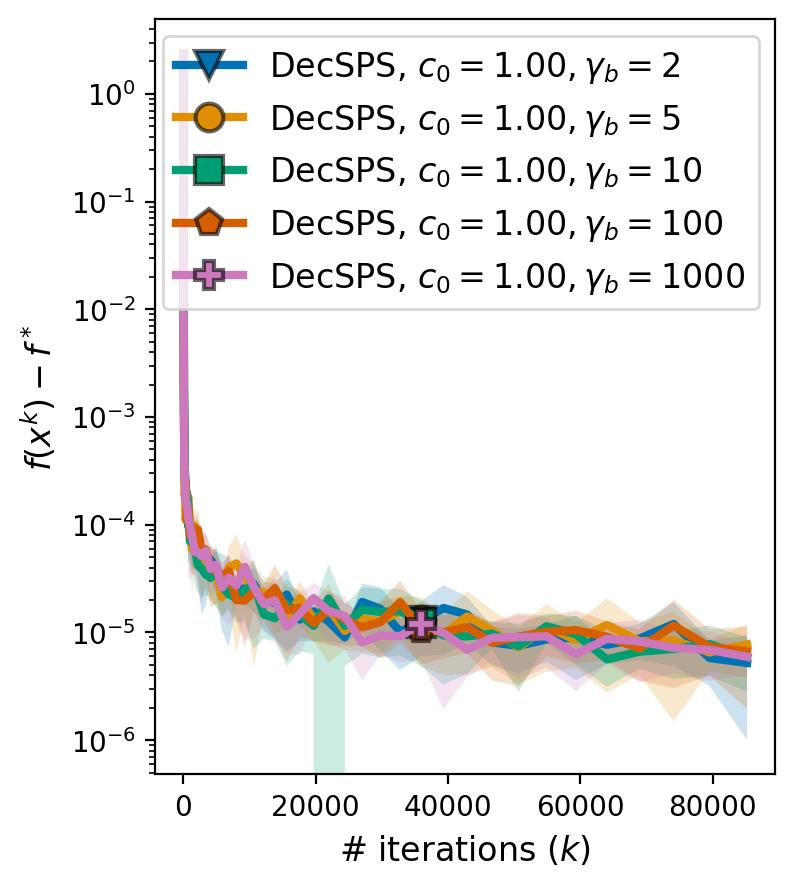}
    \includegraphics[height = 0.26\textwidth]{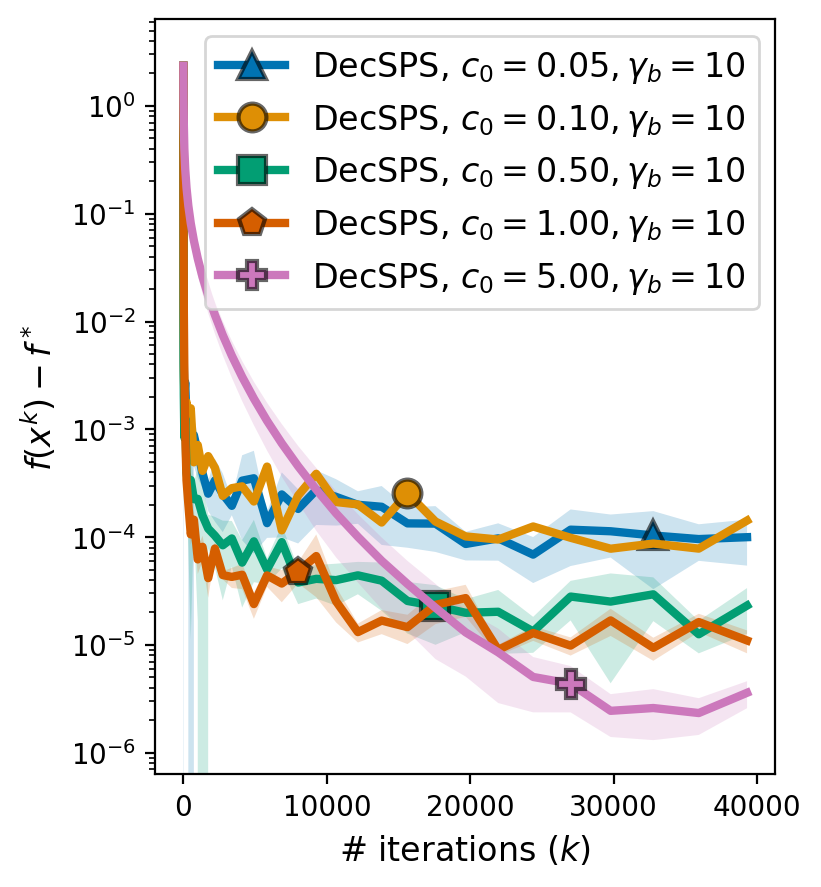}
    \vspace{-2mm}
    \caption{\small Tuning of DecSPS on the A1A and Breast cancer datasets.}
    \label{fig:A1A_Breast_tuning}
\end{figure}

\paragraph{Comparison with SGD.} In addition to Figure~\ref{Fig_SGDVsDecSPS_A1A}~(A1A dataset), in Figure~\ref{Fig_SGDVsDecSPS_app} we provide comparison of DecSPS with SGD with stepsize $\gamma_0/\sqrt{k+1}$ for the Synthetic and Breast Cancer datasets. From the results, it is clear that DecSPS with standard parameters $c_0=1, \gamma_b = 10$~(see discussion in main paper and paragraph above) is comparable if not faster than vanilla SGD with decreasing stepsize.

\begin{figure}[ht]
    \centering
\textbf{\scriptsize Synthetic Logistic Regression -- SGD Vs DecSPS \quad \quad  \quad \quad \quad  \quad  Regularized Breast Cancer -- SGD Vs DecSPS}\\
    \includegraphics[height=0.26\linewidth]{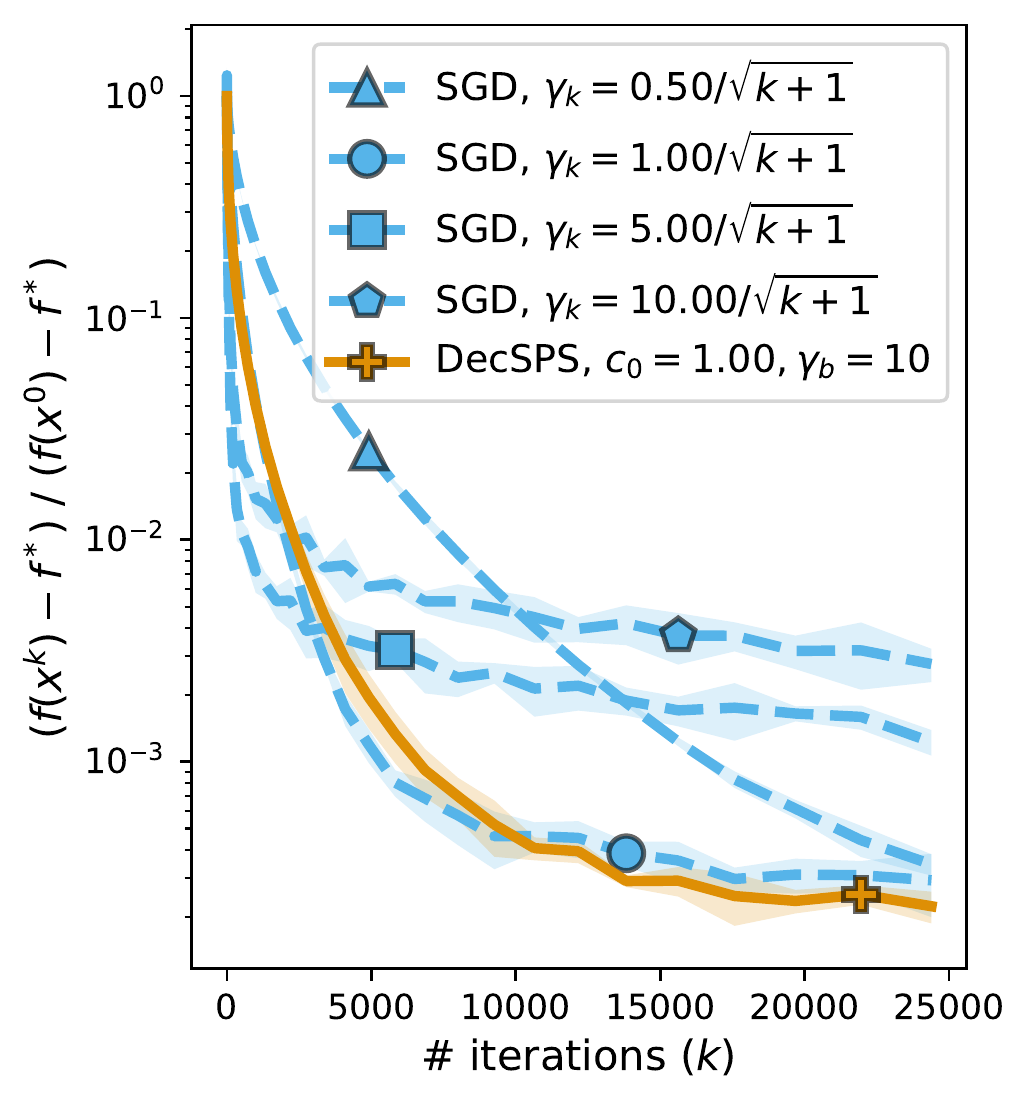}
    \includegraphics[height=0.26\linewidth]{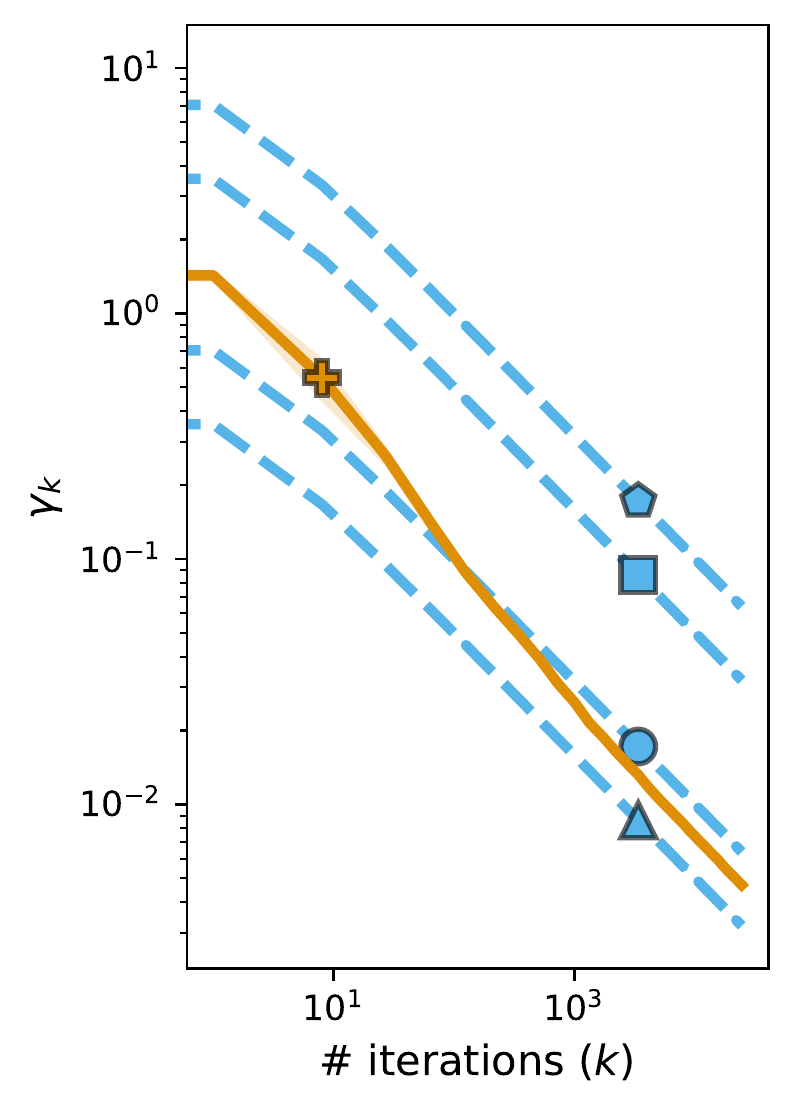}
    \includegraphics[height=0.26\linewidth]{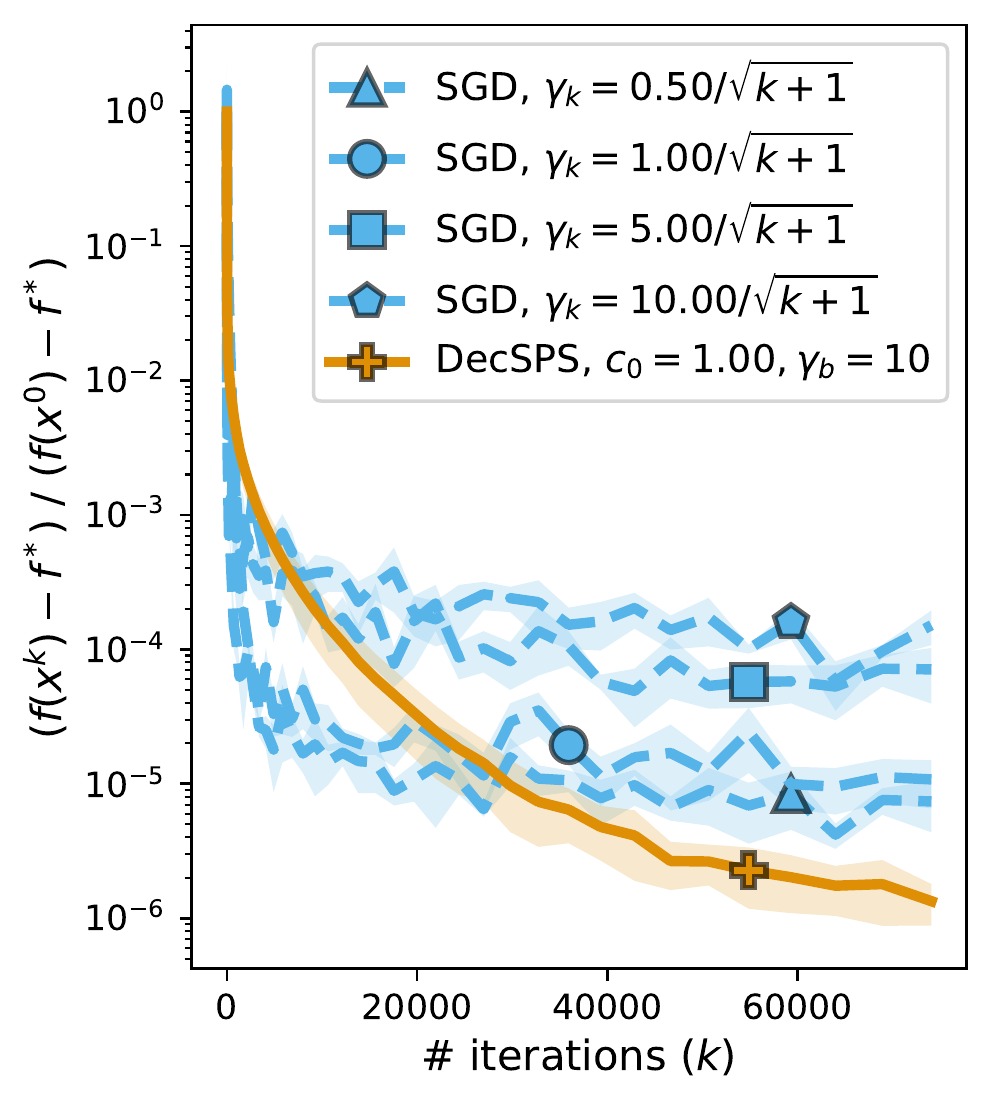}
    \includegraphics[height=0.26\linewidth]{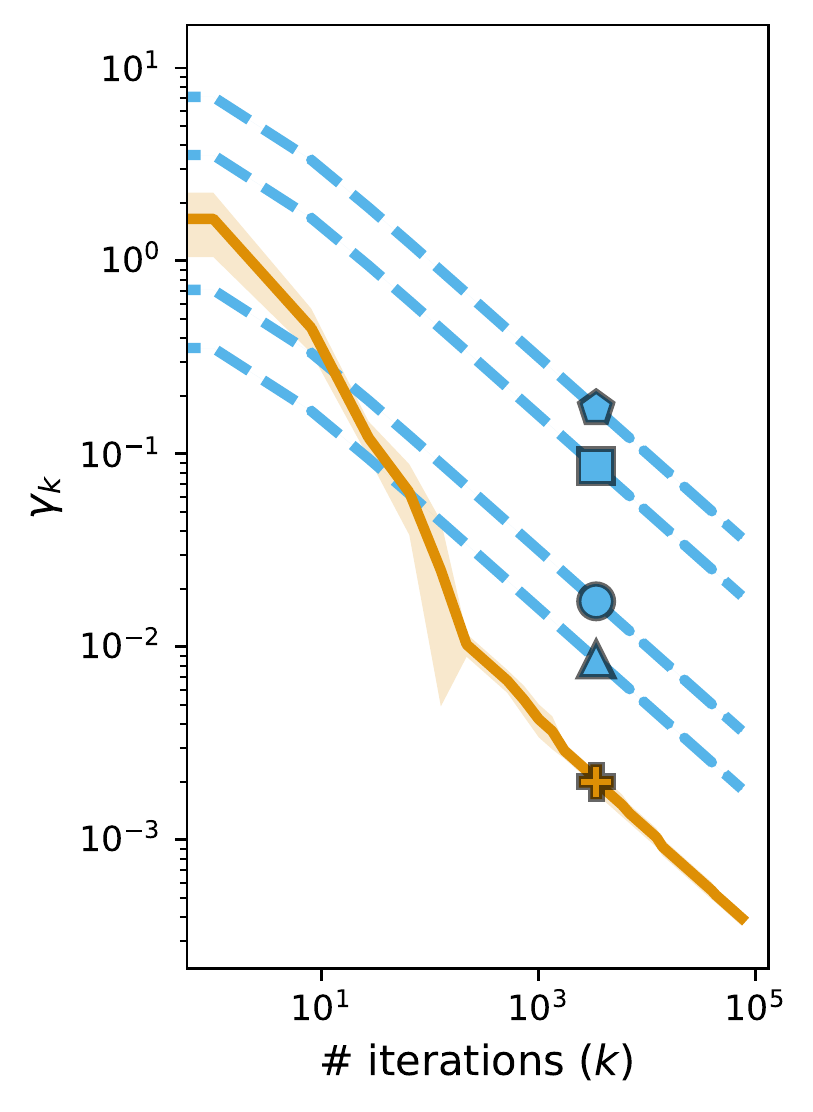}
    \vspace{-2mm}
    \caption{\small DecSPS on the Synthetic Dataset~($\lambda = 1e-4$) and the Breast Cancer Dataset~($\lambda = 1e-1$) .}
    \label{Fig_SGDVsDecSPS_app}
\end{figure}

\paragraph{Comparison with Adagrad-Norm.} In addition to Figure~\ref{Fig_SGDVsDecSPS_A1A}~(Breast Cancer dataset), in Figures~\ref{fig:synthetic_ada_tuning_1}\&\ref{fig:synthetic_ada_tuning_2}\&\ref{Fig_SGDVsADA_sklearn} we provide comparison of DecSPS with AdaGrad-Norm~\cite{ward2019adagrad} for the Synthetic and A1A datasets. AdaGrad-norm at each iteration updates the scalar $b_{k+1}^2=b_{k}^2+\|\nabla f_{\SC_k}(x_k)\|^2$ and then selects the next step as $x_{k+1}=x_{k}-\frac{\eta}{b_{k+1}}\nabla f_i(x_k)$.  Hence, it has tuning parameters $b_0$ and $\eta$; $b_0=0.1$ is recommended in~\citep{ward2019adagrad}~(see their Figure 3). Using this value for $b_0$ we show in Figure~\ref{Fig_SGDVsADA_sklearn} that the performance of DecSPS is competitive against a well-tuned value of the AdaGrad-norm stepsize $\eta$. In Figure~\ref{fig:synthetic_ada_tuning_1}\&\ref{fig:synthetic_ada_tuning_2} we show the effect of tuning $b_0$ on the synthetic dataset: no major improvement is observed.

\begin{figure}[ht]
    \centering
\textbf{\scriptsize \qquad \qquad Regularized A1A -- DecSPS Vs AdaGrad-Norm \qquad \qquad  Synthetic Dataset -- DecSPS Vs AdaGrad-Norm}\\
    \includegraphics[height=0.26\linewidth]{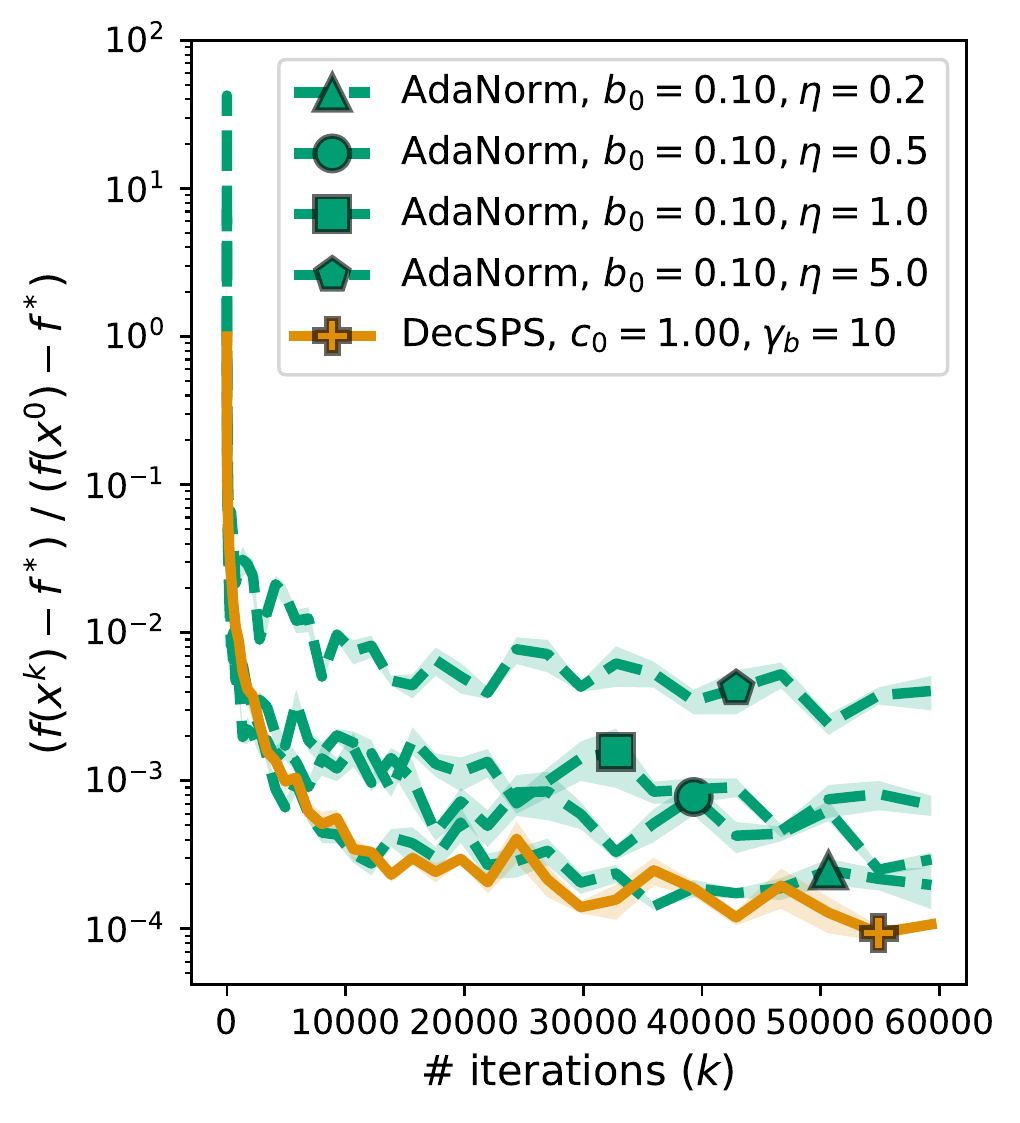}
    \includegraphics[height=0.26\linewidth]{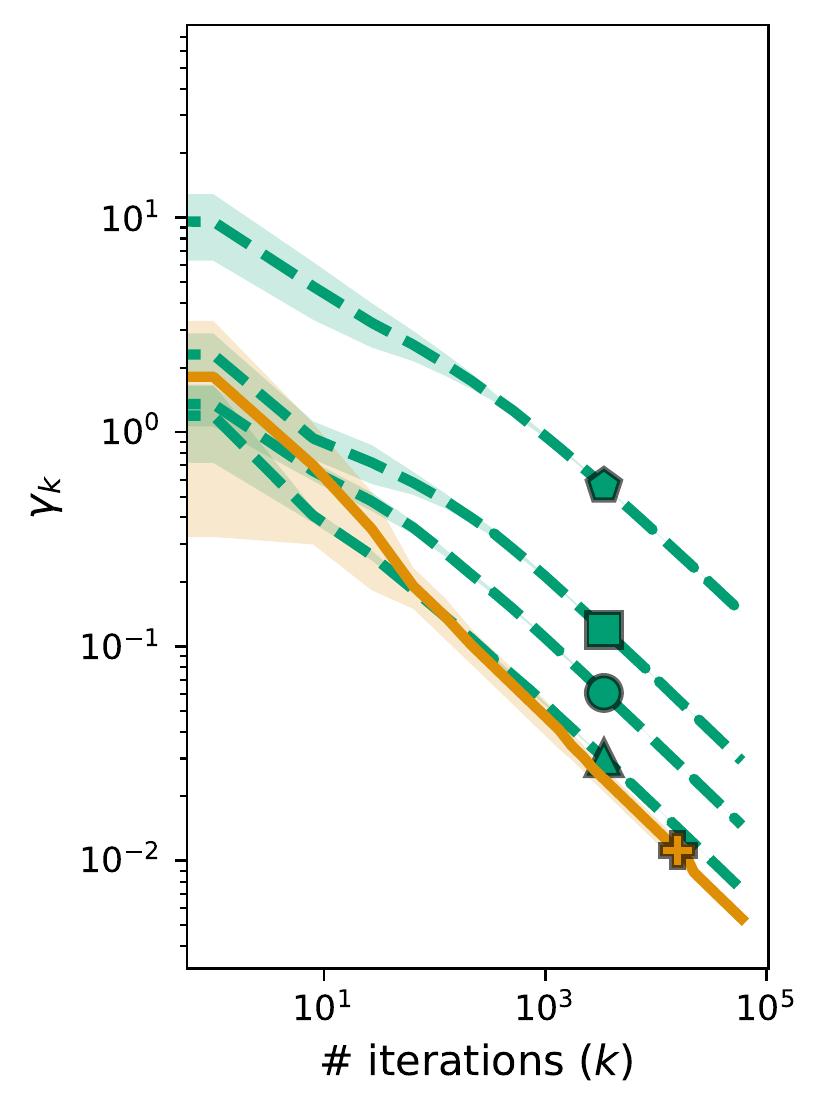}
\includegraphics[height=0.26\linewidth]{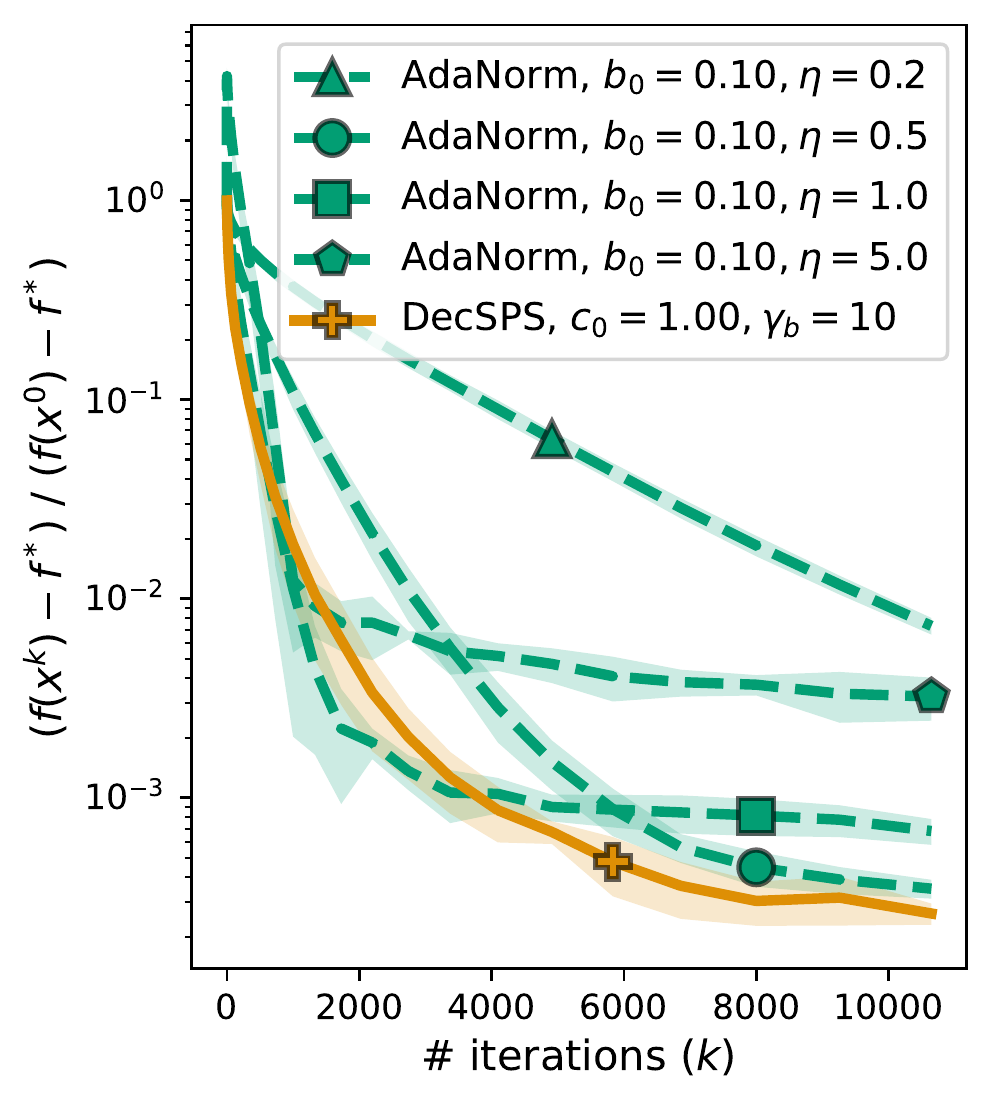}
    \includegraphics[height=0.26\linewidth]{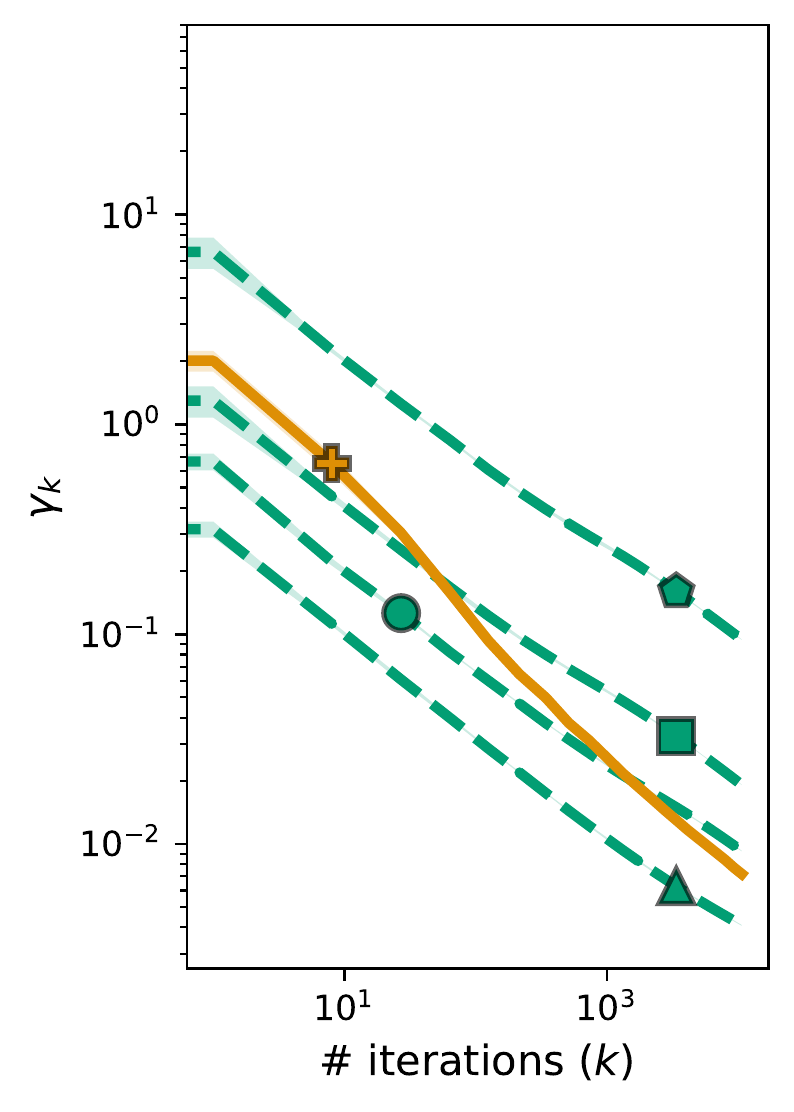}
    \vspace{-2mm}
    \caption{\small Performance of AdaGrad-Norm compared to DecSPS on the synthetic and A1A datasets. This figure is a complement to Figure~\ref{Fig_SGDVsDecSPS_A1A}.}
    \label{Fig_SGDVsADA_sklearn}
\end{figure}

\begin{figure}[ht]
    \centering
    \textbf{\scriptsize\quad \quad Synthetic Logistic Regression, $\boldsymbol{b_0 = 0.05}$\qquad\quad\quad \quad\quad  Synthetic Logistic Regression, $\boldsymbol{b_0 = 0.5}$}\\
    \includegraphics[height = 0.26\textwidth]{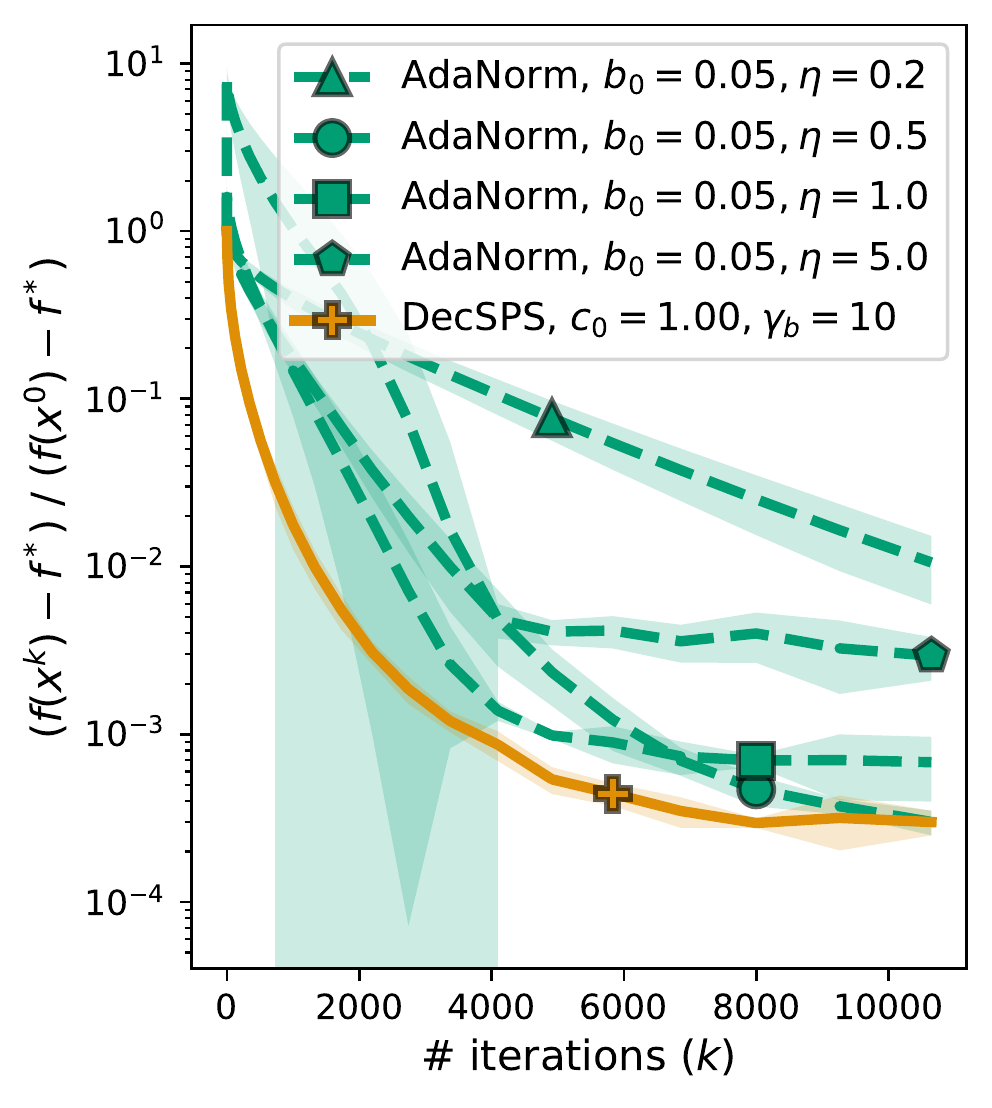}
    \includegraphics[height = 0.26\textwidth]{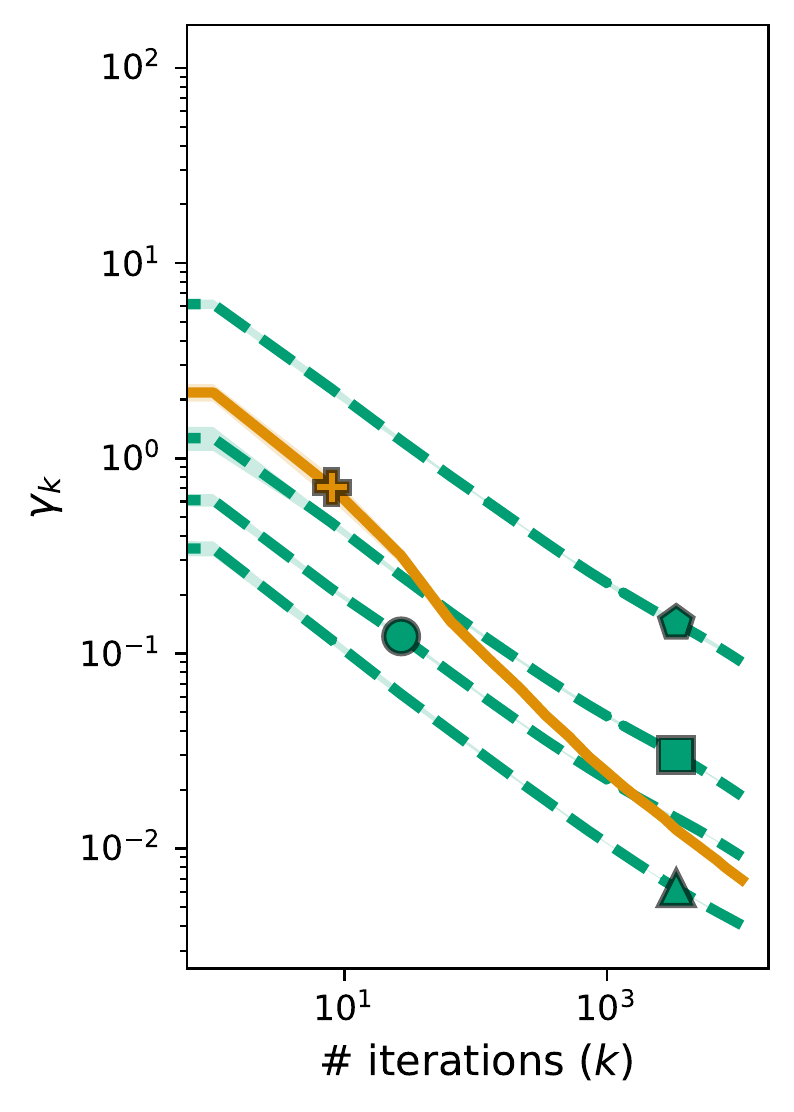}
    \includegraphics[height = 0.26\textwidth]{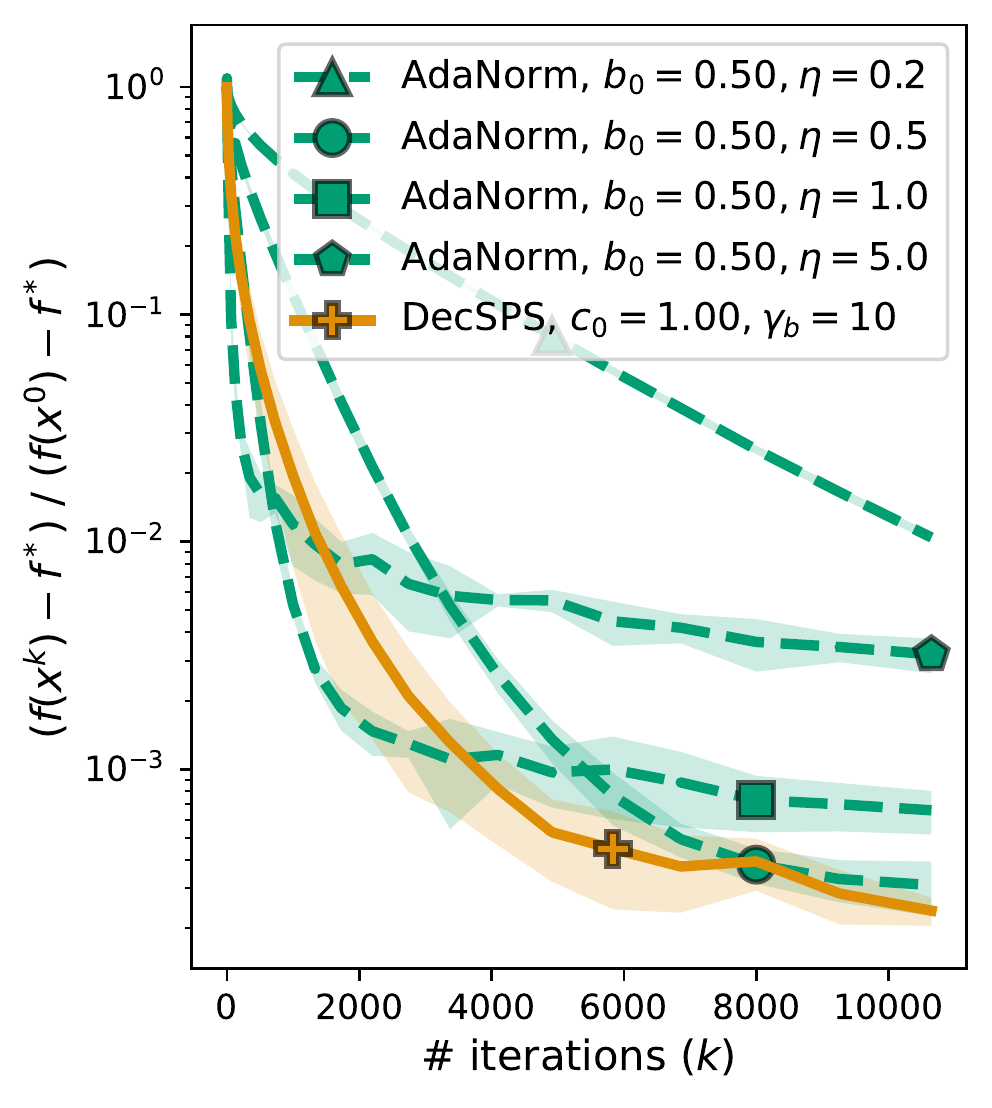}
    \includegraphics[height = 0.26\textwidth]{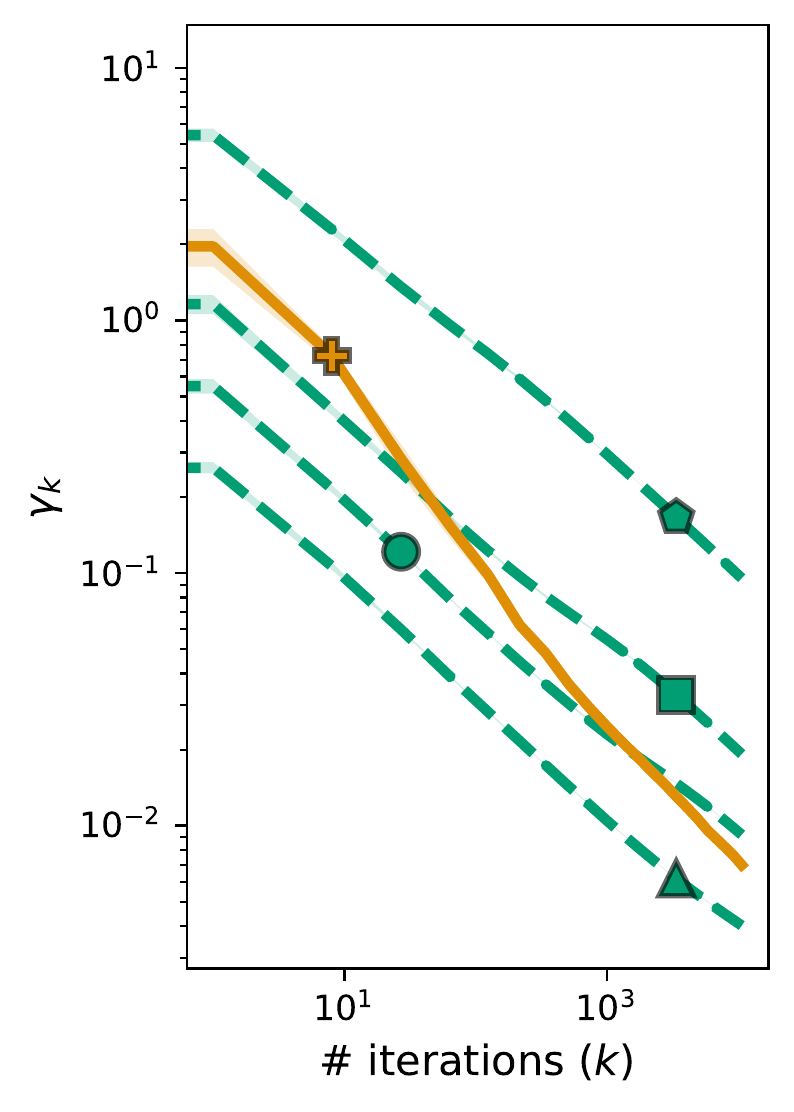}
    \vspace{-2mm}
    \caption{\small Performance of AdaGrad-Norm compared to DecSPS on the Synthetic dataset, for $b_0\ne 0.1$. This figure is a complement to Figure~\ref{Fig_SGDVsDecSPS_A1A}.}
    \label{fig:synthetic_ada_tuning_1}
\end{figure}

\begin{figure}[ht]
     \centering
    \textbf{\scriptsize\quad \quad \quad \quad Regularized A1A, $\boldsymbol{b_0 = 0.5, 0.05}$\qquad\quad\quad \quad\quad  \quad\quad Regularized Breast Cancer , $\boldsymbol{b_0 = 0.5, 0.05}$}\\
    \includegraphics[width = 0.22\textwidth]{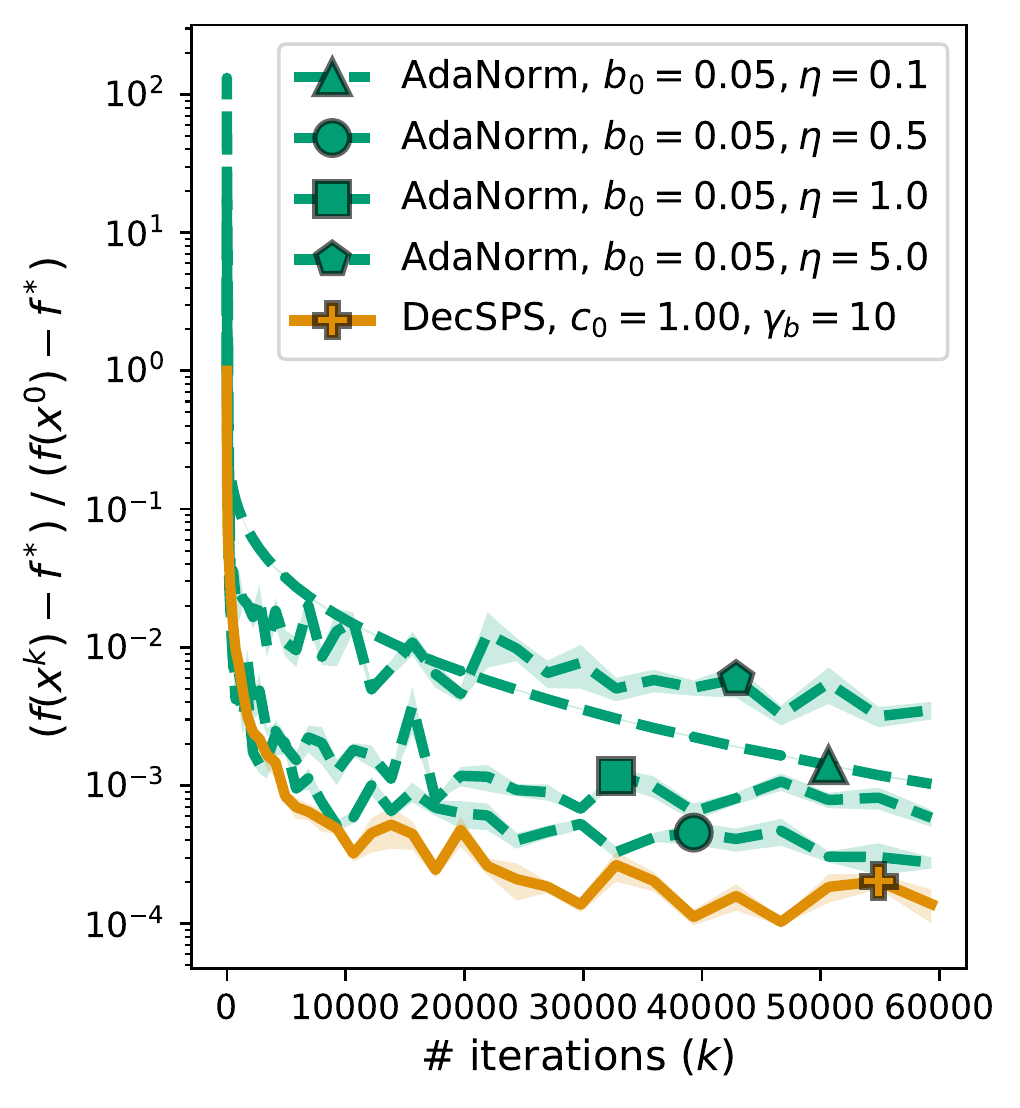}
     \includegraphics[width = 0.22\textwidth]{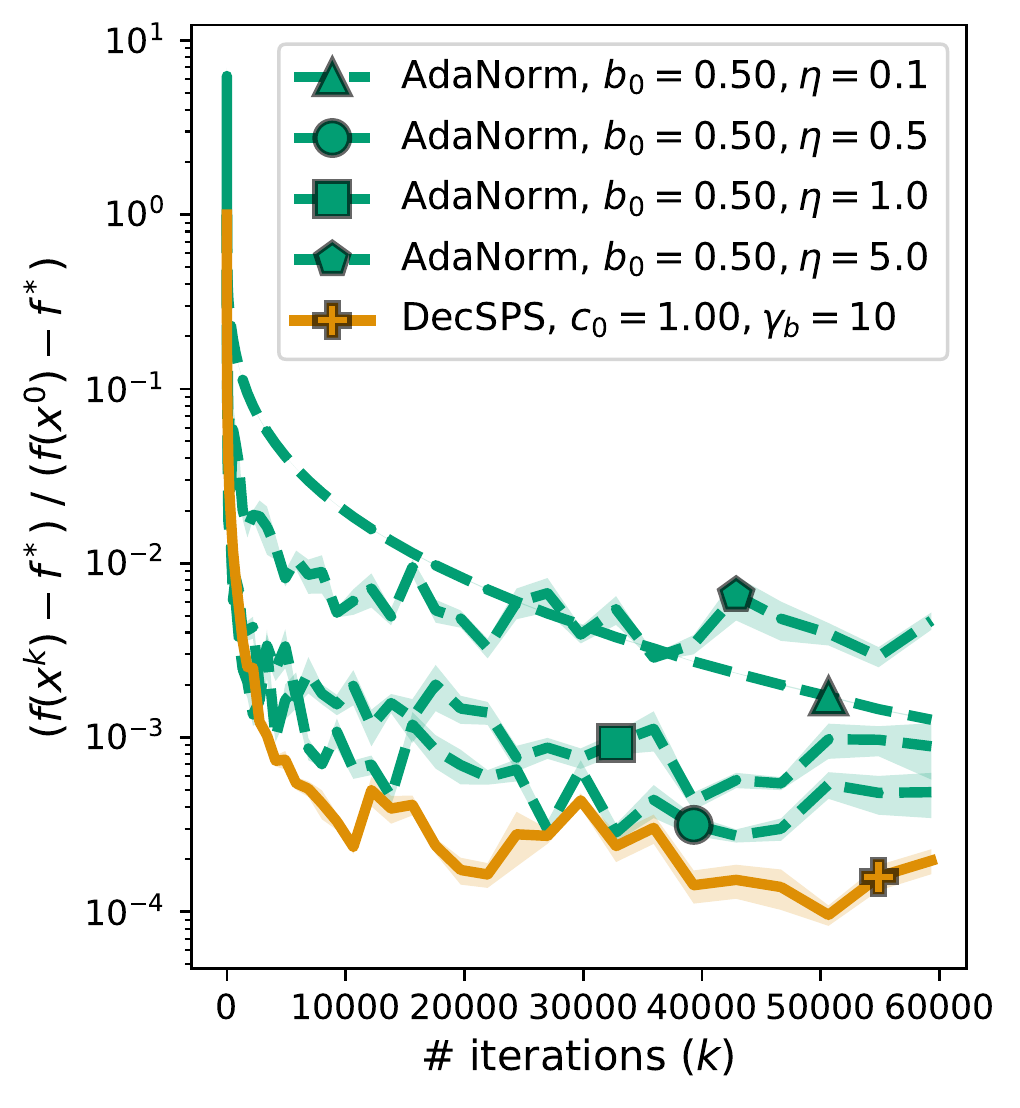}
     \includegraphics[width = 0.22\textwidth]{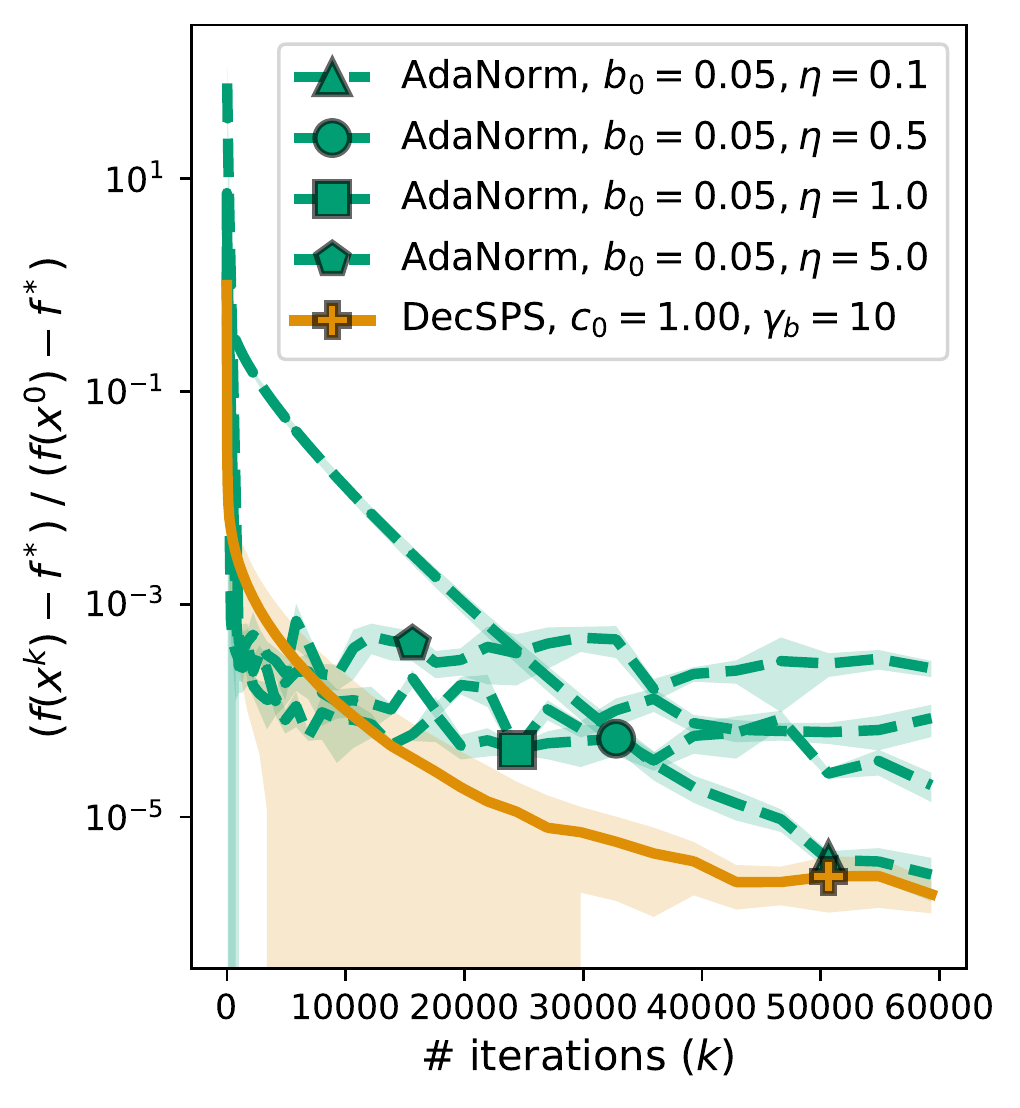}
     \includegraphics[width = 0.22\textwidth]{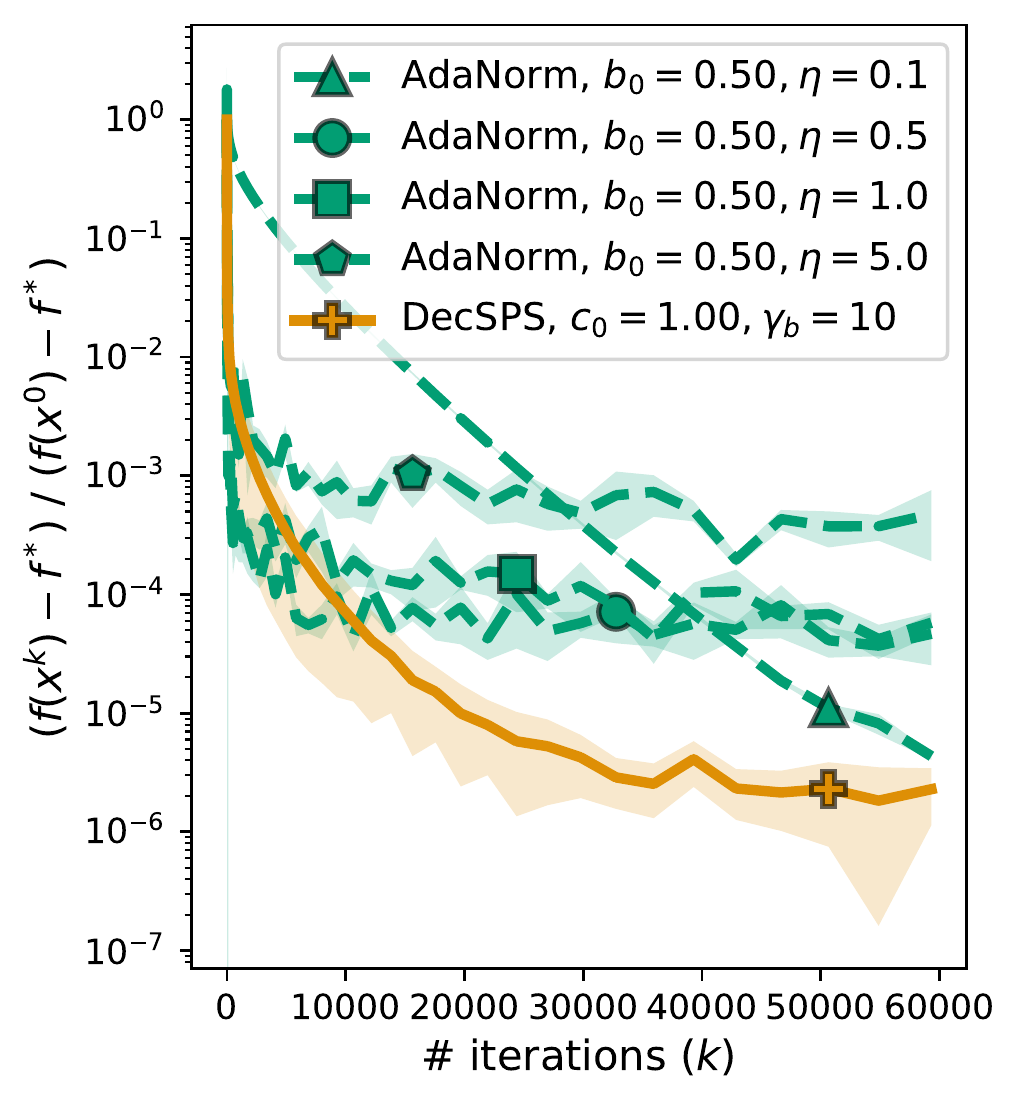}
     \caption{\small Performance of AdaGrad-Norm compared to DecSPS on A1A and Breast Cancer datasets, for $b_0\ne 0.1$. This figure is a complement to Figure~\ref{Fig_SGDVsDecSPS_A1A}.}
    \label{fig:synthetic_ada_tuning_2}
 \end{figure}

\paragraph{Comparison with Adam and AMSgrad.} We provide a comparison with Adam~\cite{kingma2014adam} and AMSgrad~\cite{reddi2019convergence}. For both methods, we set the momentum parameter to zero~(a.k.a RMSprop) for a fair comparison with DecSPS. For $\beta :=\beta_2$, the parameter that controls the moving average of the second moments, we select the value $0.99$ since we found that the standard $0.999$ leads to problematic~(exploding) stepsizes. Findings are pretty similar for both the A1A~(Figure~\ref{Fig_SGDVsAdam_A1A}) and Breast Cancer~(Figure~\ref{Fig_SGDVsAdam_Breast}) datasets: when compared to DecSPS with the usual parameters, fine-tuned Adam with fixed stepsize can reach the same performance after a few tens of thousand iterations --- however, it is much slower at the beginning of training. While deriving convergence guarantees for Adam is problematic~\cite{reddi2019convergence}, AMSgrad~\cite{reddi2019convergence} with stepsize $\eta/\sqrt{k+1}$ enjoys a convergence guarantee similar to Adagrad and Adagrad-Norm. This is reflected in the empirical convergence: fine-tuned AMSgrad is able to match the convergence of DecSPS with the usual parameters motivated at the beginning of this section. Yet, we recall that the convergence guarantees of AMSgrad require the iterates to live in a bounded domain, an assumption which is not needed in our DecSPS~(see \S~\ref{sec:no_bound}).

\begin{figure}[ht]
    \centering
\textbf{\scriptsize \qquad \qquad Regularized A1A -- DecSPS Vs Adam \qquad \qquad \quad Regularized A1A -- DecSPS Vs AMSgrad}\\
    \includegraphics[height=0.25\linewidth]{fig/real_new/A1A_Adam.pdf}
    \includegraphics[height=0.25\linewidth]{fig/real_new/A1A_stepsize_Adam.pdf}
\includegraphics[height=0.25\linewidth]{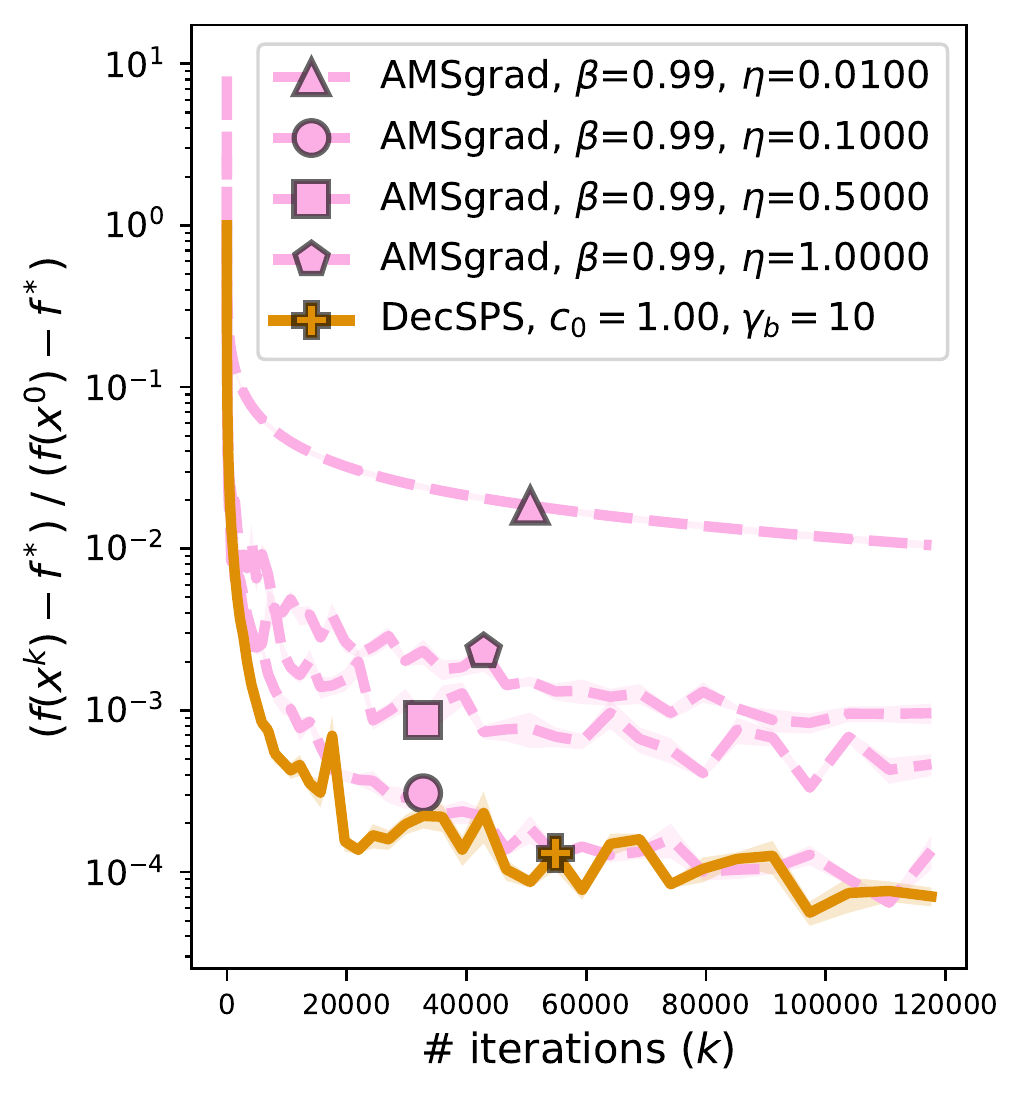}
    \includegraphics[height=0.25\linewidth]{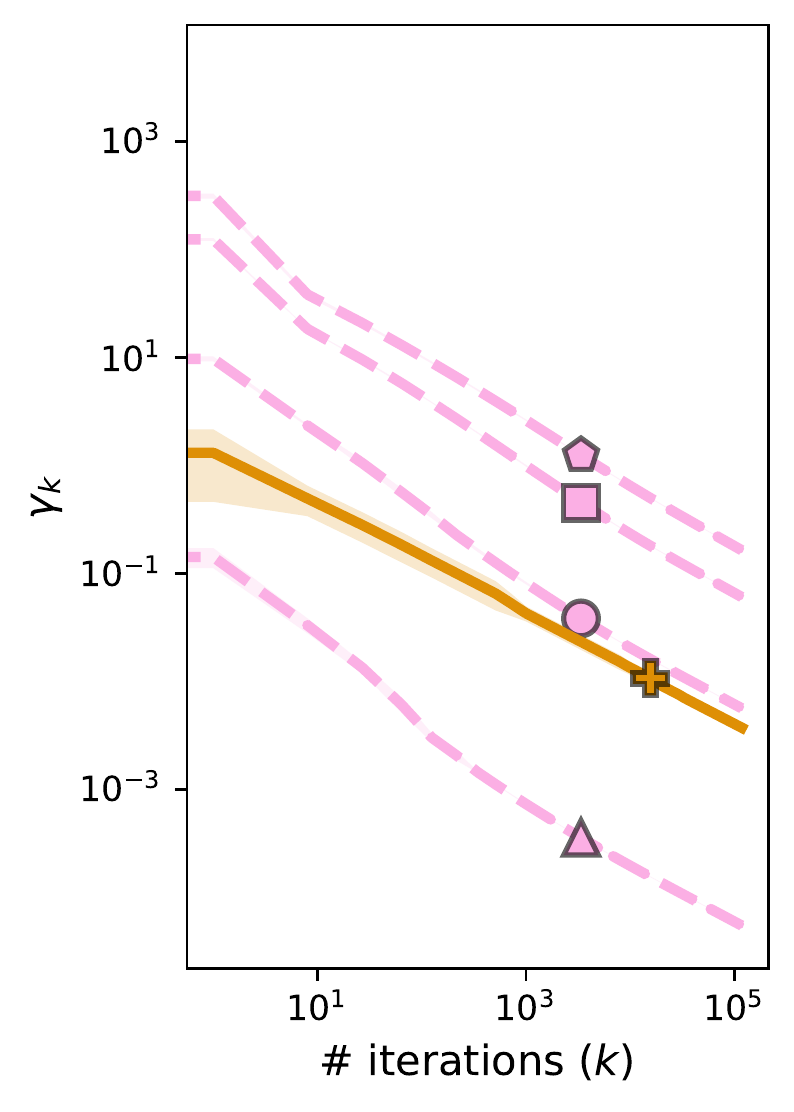}

    \vspace{-2mm}
\caption{\small Performance of Adam~(with fixed stepsize and no momentum) and AMSgrad~(with sqrt decreasing stepsize and no momentum) compared to DecSPS on the A1A dataset. Plotted is also the average stepsize~(each parameter evolves with a different stepsize).}
    \label{Fig_SGDVsAdam_A1A}
\end{figure}

\begin{figure}[h]
    \centering
\textbf{\scriptsize \qquad \qquad Regularized Breast Cancer -- DecSPS Vs Adam \qquad  Regularized Breast Cancer -- DecSPS Vs AMSgrad}\\
    \includegraphics[height=0.25\linewidth]{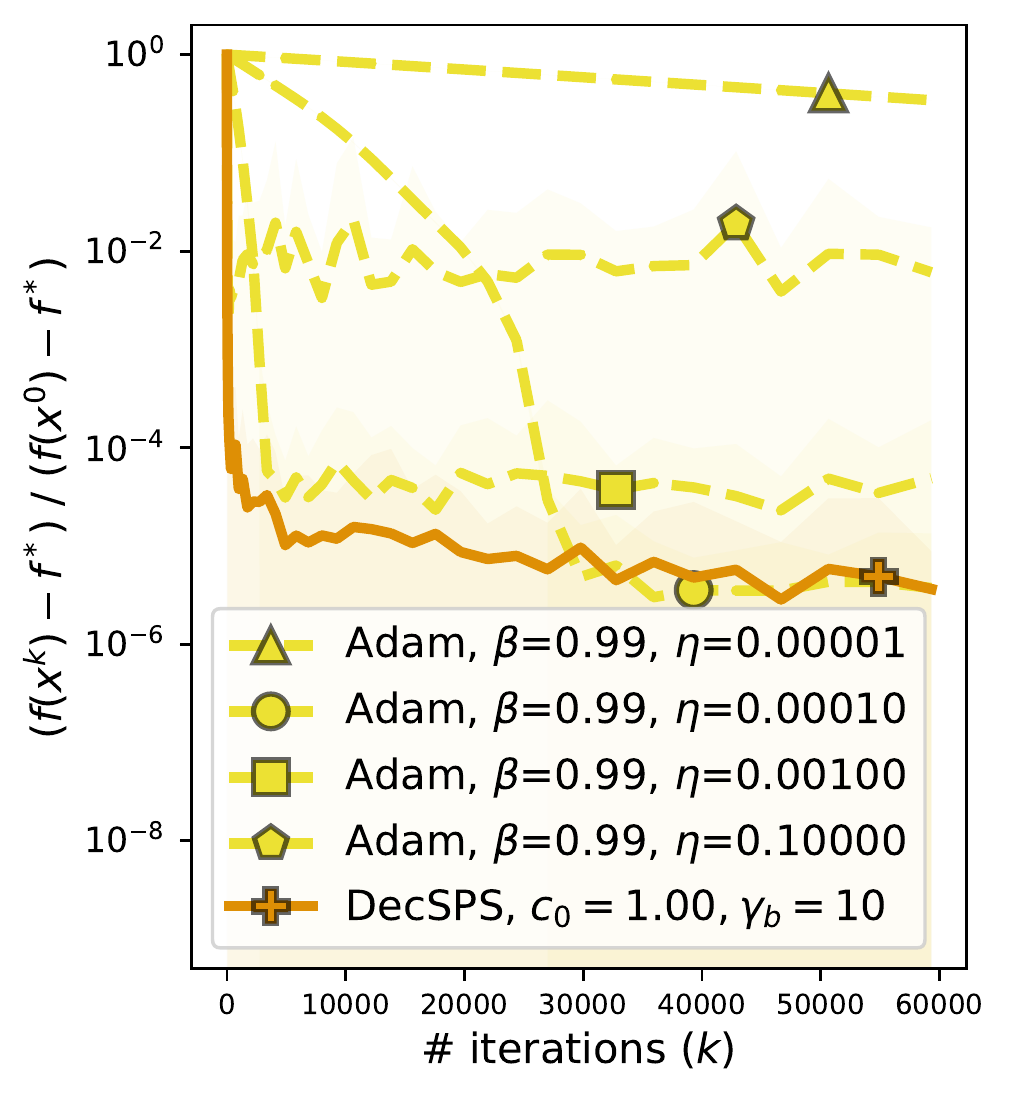}
    \includegraphics[height=0.25\linewidth]{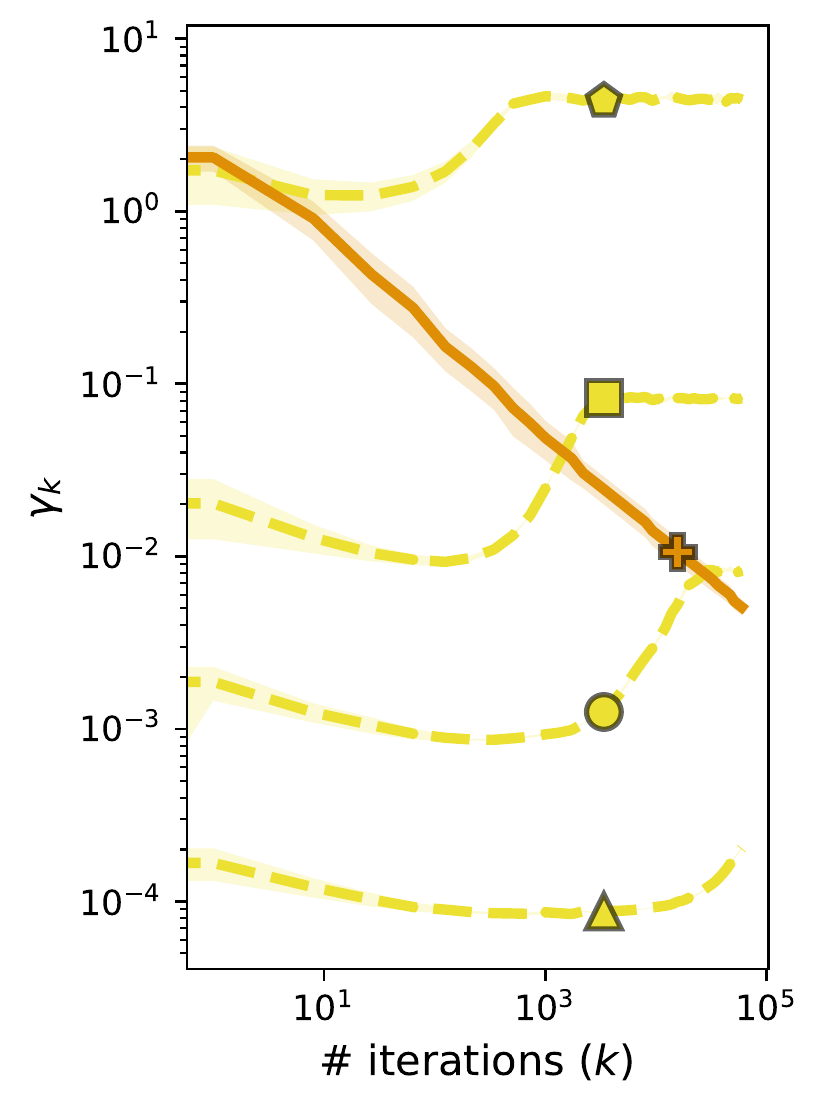}
\includegraphics[height=0.25\linewidth]{fig/real_new/Breast_AMSgrad.pdf}
    \includegraphics[height=0.25\linewidth]{fig/real_new/Breast_stepsize_AMSgrad.pdf}

    \vspace{-2mm}
\caption{\small Performance of Adam~(with fixed stepsize and no momentum) and AMSgrad~(with sqrt decreasing stepsize and no momentum) compared to DecSPS on the Breast Cancer dataset. Plotted is also the average stepsize~(each parameter evolves with a different stepsize).}    \label{Fig_SGDVsAdam_Breast}
\end{figure}

\paragraph{Performance under light regularization.} If the problem at hand does not have strong curvature information, e.g. there is very light regularization, then additional tuning of the DecSPS parameters is required. Figure~\ref{fig_light_reg} shows that it is possible to retrieve the performance of SGD also with light regularization parameters~($1e-4, 1e-6$) under additional tuning of $c_0$ and $\gamma_b$.

 \begin{figure}[h]
     \centering
\textbf{\scriptsize SGD vs DecSPS on A1A with lighter regularization}\\
\includegraphics[width = 0.22\textwidth]{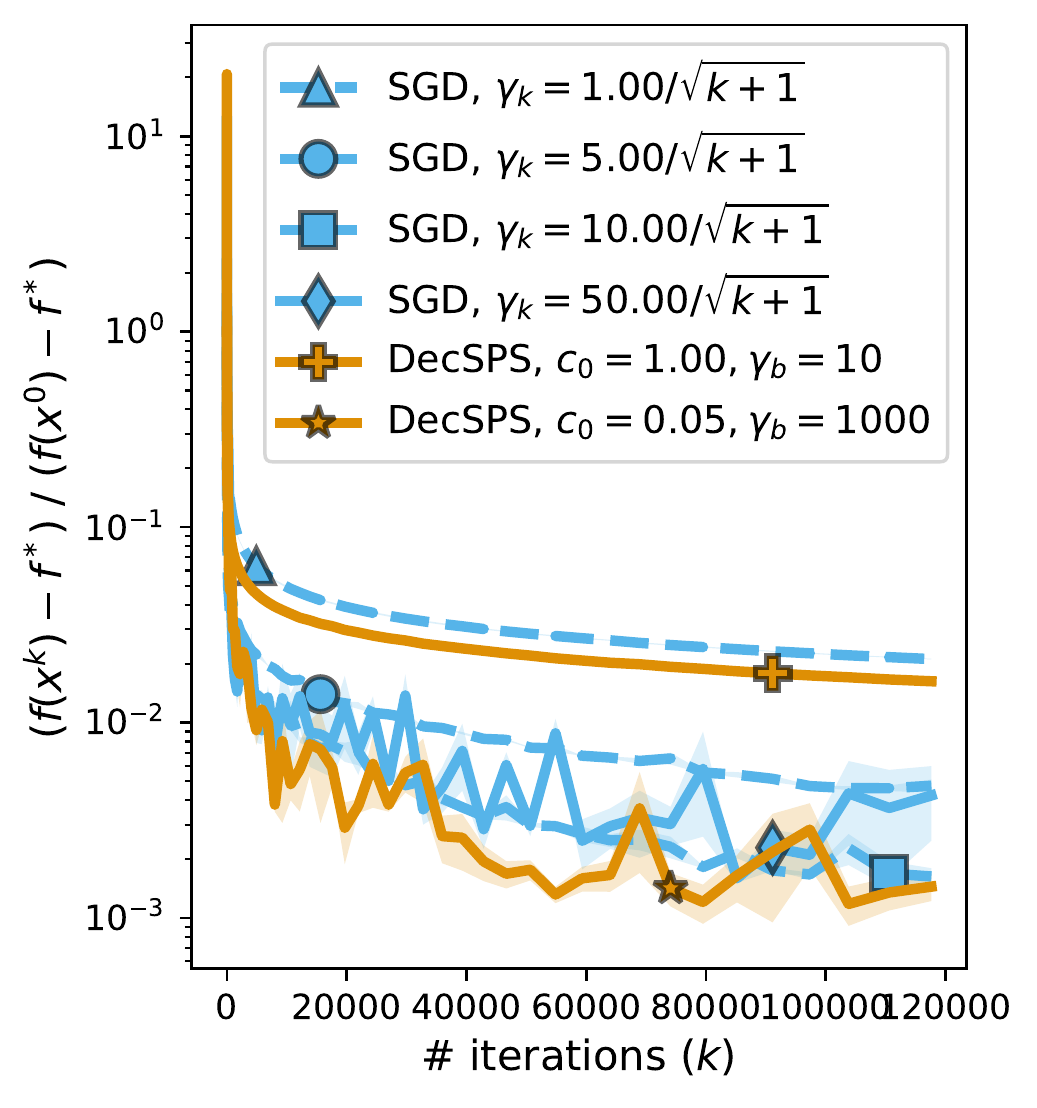}
     \includegraphics[width = 0.22\textwidth]{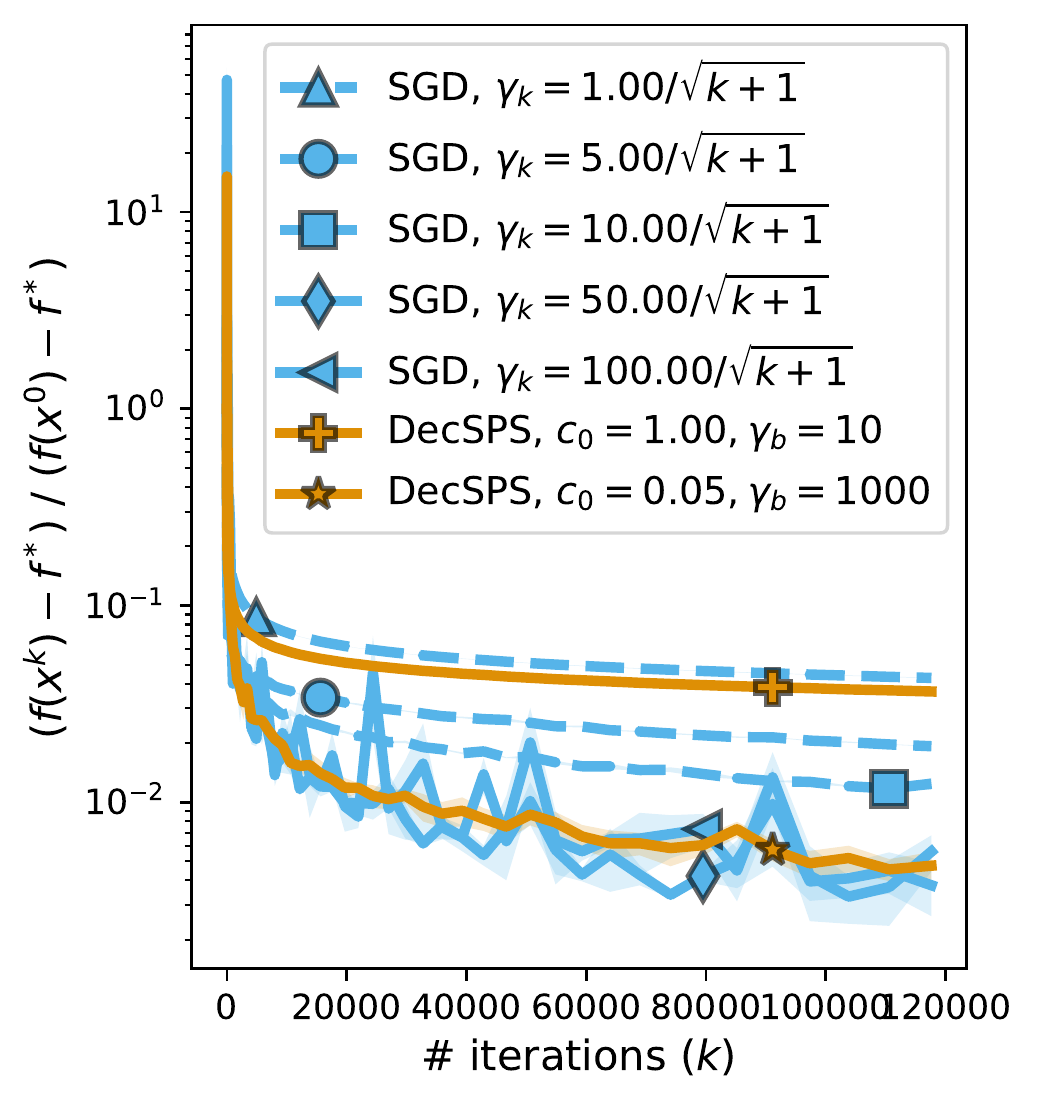}

     \caption{Results on A1A for $\lambda = 1e-4$ (left) and $\lambda = 1e-6$ (right). Additional tuning of SPS is required to match the tuned SGD performance.}
     \label{fig_light_reg}
 \end{figure}

\end{document}